\documentclass[12pt,leqno]{amsart}

\usepackage{amssymb}

\usepackage[all]{xy}
\usepackage{enumerate}

\numberwithin{equation}{section}
\usepackage{hyperref}

\newtheorem{theorem}{Theorem}[section]
\newtheorem{proposition}[theorem]{Proposition}
\newtheorem{lemma}[theorem]{Lemma}
\newtheorem{corollary}[theorem]{Corollary}

\theoremstyle{definition}

\newtheorem{definition}[theorem]{Definition}

\newtheorem{example}[theorem]{Example}

\newtheorem{remark}[theorem]{Remark}

\newcommand{\N}{{\mathbb{N}}}
\newcommand{\R}{{\mathbb{R}}}

\newcommand{\Z}{{\mathbb{Z}}}

\newcommand{\cor}{{\bf k}}

\renewcommand{\to}[1][]{\xrightarrow[]{#1}}
\newcommand{\from}[1][]{\xleftarrow[]{#1}}
\newcommand{\isoto}[1][]{\xrightarrow[#1]%
{{\raisebox{-.6ex}[0ex][-.6ex]{$\mspace{1mu}\sim\mspace{2mu}$}}}}
\newcommand{\isofrom}[1][]{\xleftarrow[#1]%
{{\raisebox{-.6ex}[0ex][-.6ex]{$\mspace{1mu}\sim\mspace{2mu}$}}}}

\newcommand{\RR}{\mathrm{R}}

\newcommand{\Hom}{\mathrm{Hom}}
\newcommand{\RHom}{\RR\mathrm{Hom}}
\newcommand{\Ext}{\mathrm{Ext}}
\renewcommand{\hom}{{\mathcal{H}om}}
\newcommand{\rhom}{{\RR\hom}}
\newcommand{\DD}{\mathrm{D}}
\newcommand{\tens}{\otimes}
\newcommand{\ltens}{\mathbin{\overset{\scriptscriptstyle\mathrm{L}}\tens}}

\newcommand{\End}{{\operatorname{End}}}
\newcommand{\Aut}{{\operatorname{Aut}}}

\newcommand{\sect}{\Gamma}
\newcommand{\rsect}{\mathrm{R}\Gamma}

\newcommand{\oim}[1]{{#1}_*}
\newcommand{\eim}[1]{{#1}_!}
\newcommand{\roim}[1]{\RR{#1}_*}
\newcommand{\reim}[1]{\RR{#1}_!}
\newcommand{\opb}[1]{#1^{-1}}
\newcommand{\epb}[1]{#1^{!}}

\newcommand{\eqdot}{\mathbin{:=}}
\newcommand{\cl}{\colon}

\newcommand{\scbul}{{\,\raise.4ex\hbox{$\scriptscriptstyle\bullet$}\,}}
\newcommand{\ol}{\overline}
\newcommand{\ul}{\underline}

\newcommand{\id}{\mathrm{id}}

\newcommand{\supp}{\operatorname{supp}}

\newcommand{\SSi}{\mathrm{SS}}
\newcommand{\Int}{\operatorname{Int}}

\newcommand{\Der}{\mathsf{D}}
\newcommand{\Derb}{\Der^{\mathrm{b}}}

\newcommand{\Mod}{\operatorname{Mod}}

\newcommand{\coker}{\operatorname{coker}}

\newcommand{\dT}{{\dot{T}}}

\newcommand{\mo}{\mathopen}
\newcommand{\mc}{\mathclose}

\newcommand{\corC}{\mathbb{C}}
\newcommand{\Rm}{\mathbb{R}^\times}
\newcommand{\Rp}{\mathbb{R}^+}
\newcommand{\RP}{\mathbb{RP}}

\newcommand{\Dersf}{\Der^{s,f}}
\newcommand{\rc}{{\mathrm{c}}}
\newcommand{\T}{\mathbf{T}}
\newcommand{\cer}{{\mathbb{S}^1}}
\newcommand{\sph}{{\mathbb{S}^2}}
\newcommand{\Mat}{\operatorname{Mat}}
\newcommand{\GL}{\operatorname{GL}}
\newcommand{\PAut}{\operatorname{Aut}^p}
\newcommand{\cecc}{\operatorname{C}}
\newcommand{\mon}{\operatorname{m}}
\newcommand{\im}{\operatorname{im}}
\newcommand{\catc}{\operatorname{\mathsf{C}}}
\newcommand{\demi}{\frac{1}{2}}

\newcommand{\sui}{\medskip\noindent}

\newcommand{\fib}{H}
\newcommand{\bb}{B}

\begin{document}

\begin{abstract}
  We prove Arnol'd's three cusps conjecture about the front of Legendrian curves
  in the projectivized cotangent bundle of the $2$-sphere.  We use the
  microlocal theory of sheaves of Kashiwara and Schapira and study the derived
  category of sheaves on the $2$-sphere with a given smooth Lagrangian
  microsupport.
\end{abstract}

\date{Mar 25, 2016 \\ The author is partially supported by the ANR project MICROLOCAL
(ANR-15CE40-0007-01).}
\title[The three cusps conjecture]{The three cusps conjecture}
\author{St{\'e}phane Guillermou}

\maketitle

\section{Introduction}

In this paper we use the microlocal theory of sheaves of Kashiwara and Schapira
to prove a conjecture of Arnol'd about the cusps of fronts of Legendrian curves
in the projectivized cotangent bundle of the sphere.

In~\cite{A96} Arnol'd states a theorem of M\"obius ``a closed smooth curve
sufficiently close to the projective line (in the projective plane) has at least
three points of inflection'' and conjectures that ``the three points of
flattening of an immersed curve are preserved so long as under the deformation
there does not arise a tangency of similarly oriented branches''.  This is a
statement about oriented curves in $\RP^2$.  Under the projective duality it is
turned into a statement about Legendrian curves in the projectivized cotangent
bundle of $\RP^2$: if $\{\Lambda_t\}_{t\in [0,1]}$ is a generic path in the
space of Legendrian knots in $PT^*\RP^2 = (T^*\RP^2 \setminus \RP^2)/\Rm$
such that $\Lambda_0$ is a fiber of the projection $\pi \cl PT^*\RP^2 \to
\RP^2$, then the front $\pi(\Lambda_t)$ has at least three cusps.  This
statement is given in~\cite{CP05}, where the authors prove a local version,
replacing $\RP^2$ by the plane $\R^2$, and also another similar conjecture
``the four cusps conjecture''.

Here we prove the conjecture when the base is the sphere $\sph$ (which implies
the case of $\RP^2$ since we can lift a deformation in $PT^*\RP^2$ to a
deformation in $PT^*\sph$). Actually this paper is mainly a study of
constructible sheaves in dimension $1$ and generically in dimension $2$.  The
idea is to use the category $\Derb_\Lambda(\cor_\sph)$ of sheaves with
microsupport $\Lambda$ (more precisely the simple sheaves in this category) and
a result of~\cite{GKS10} which says that two isotopic $\Lambda$, $\Lambda'$ have
equivalent associated categories.  We refer the reader to~\cite{STZ14} for
another study of the same categories where the authors compute the number of
simple sheaves with microsupport given by a Legendrian knot.

\smallskip

Before we explain the idea of the proof we recall some notions and results of
the microlocal theory of sheaves.  Let $M$ be a manifold of class $C^\infty$ and
let $\cor$ be a ring.  We denote by $\Derb(\cor_M)$ the derived category of
sheaves of $\cor$-modules over $M$.  The microsupport $\SSi(F)$ of an object $F$
of $\Derb(\cor_M)$ is a closed subset of the cotangent bundle $T^*M$, conic for
the action of $\Rp$ on $T^*M$ and coisotropic.  The easiest examples are given
by $F = \cor_N$ the constant sheaf on a submanifold of $M$ or $F = \cor_U$ the
constant sheaf on an open subset with smooth boundary. Then $\SSi(\cor_N)$ is
$T^*_NM$, the conormal bundle of $N$, and $\SSi(\cor_U)$ is, outside the
zero-section, the half-part of $T^*_{\partial U}M$ pointing in the outward
direction.  The microsupport is well-behaved with respect to the standard sheaf
operations.  It was introduced in~\cite{KS85} as a tool to analyze the local
structure of a sheaf and also to translate the geometric aspects of the theory
of $\mathcal{D}$-modules (systems of linear holomorphic PDE's) into a very
general framework of sheaves on real manifolds.

The link between the microsupport and the (local) symplectic structure of the
cotangent bundle is underlined in~\cite{KS90} but, in~\cite{T08}, Tamarkin gives
an application of the theory to global problems of symplectic geometry.
Building on his ideas the paper~\cite{GKS10} gives the following result.  Let
$\Phi \cl \dT^*M \times I \to \dT^*M$ be a homogeneous Hamiltonian isotopy,
where $\dT^*M = T^*M \setminus M$, $I$ is some interval containing $0$ and {\em
  homogeneous} means commuting with the multiplication by a {\em positive}
scalar in the fibers. Then there exists a unique $K \in \Derb(\cor_{M^2 \times
  I})$ such that $\dot\SSi(K)$ is the graph of $\Phi$, where $\dot\SSi(K) =
\SSi(K) \setminus (M^2\times I)$, and $K|_{M^2 \times \{0\}}$ is the constant
sheaf on the diagonal.  This theorem is a sheaf version of a result of Sikorav
in~\cite{S87} which says that the property of a Lagrangian submanifold of having
a generating function is stable by Hamiltonian isotopy.  It has the following
consequence.  We set $K_t = K|_{M^2 \times \{t\}}$. We can make $K_t$ act on
$\Derb(\cor_M)$ by composition $F \mapsto K_t \circ F = \reim{q_2}(K_t \otimes^L
\opb{q_1}F)$, where $q_1,q_2$ are the projections from $M^2$ to $M$.  Then we
have $\dot\SSi( K_t \circ F) = \Phi_t(\dot\SSi(F))$.  More precisely, for a
closed conic  subset $\Lambda \subset \dT^*M$ we denote by
$\Derb_\Lambda(\cor_M)$ the full subcategory of $\Derb(\cor_M)$ formed by the
$F$ such that $\dot \SSi(F) \subset \Lambda$.  Then $F \mapsto K_t \circ F$ is
in fact an equivalence of categories from $\Derb_\Lambda(\cor_M)$ to
$\Derb_{\Phi_t(\Lambda)}(\cor_M)$.

\medskip

Now we go back to the three cusps conjecture.  A Legendrian knot in $PT^*\sph$
pulls back to a conic Lagrangian submanifold of $\dT^*\sph$ which is stable by
the antipodal map $(x;\xi) \mapsto (x;-\xi)$.  A Legendrian isotopy of
$PT^*\sph$ pulls back to a homogeneous Hamiltonian isotopy of $\dT^*\sph$ which
is homogeneous for the action of $\Rm$ and not only $\Rp$.  Let $\Phi$ be such
an isotopy and let $x_0$ be a given point in $\sph$.  We set $\Lambda_0 =
T^*_{x_0}\sph$ and $\Lambda_t = \Phi_t(\Lambda_0)$.  Now the coefficient ring is
$\cor = \corC$.  It is not difficult to see that $\Derb_{\Lambda_0}(\corC_\sph)$
has countably many isomorphism classes of ``simple sheaves'' (see
Definition~\ref{def:simple_pure}).  We prove that, if the projection of
$\Lambda_t$ to the base is generic with only one cusp, then
$\Derb_{\Lambda_t}(\corC_M)$ has uncountably many isomorphism classes of simple
sheaves, which contradicts the equivalence mentioned above.

\smallskip

To see that $\Derb_{\Lambda_t}(\corC_\sph)$ has uncountably many objects we
modify a given object $F \in \Derb_{\Lambda_t}(\corC_\sph)$ by a Mayer-Vietoris
decomposition and give a criterion to check that we have built new objects.
More precisely, for a manifold $M$, two open subsets $U,V$ which cover $M$ and
$F \in \Derb(\corC_M)$, we have the Mayer-Vietoris distinguished triangle
$F_{U\cap V} \to[(a_U,a_V)] F_U \oplus F_V \to F \to[+1]$, where $a_U$ and $a_V$
are natural morphisms.  Let $\alpha$ be any automorphism of $F|_{U\cap V}$. We
replace $a_V$ by $\alpha$ in the triangle and define $F^\alpha_{U,V} \in
\Derb(\corC_M)$ as the cone of $(a_U,\alpha)$.  Here is an easy example:
$M=\R^2$ and $F = \corC_C$ where $C$ is a circle containing $0$. We set $U =
\R^2 \setminus \{0\}$ and we let $V$ be a small disc around $0$ such that $I = V
\cap C$ is an arc of $C$.  Then $F|_{U\cap V} = \corC_{I^+} \oplus \corC_{I^-}$
where $I^\pm$ are the two components of $I\setminus \{0\}$.  The morphisms $a_U$
and $a_V$ are both induced by $\id_{F|_{U\cap V}}$.  We keep $a_U$ unchanged and
change $a_V$ into $\alpha$ where $\alpha$ is the identity on $\corC_{I^+}$ and
the multiplication by a scalar $\lambda\not=0$ on $\corC_{I^-}$.  Then
$F^\alpha_{U,V}$ is the locally constant sheaf on $C$ of rank $1$ and monodromy
$\lambda$.

Our object $F^\alpha_{U,V}$ belongs to the subcategory $\Der(U,V;F)$ of
$\Derb(\corC_M)$ consisting of the objects $G$ such that there exist
isomorphisms $F|_U \isoto G|_U$ and $F|_V \isoto G|_V$.  For such a $G$ we can
define a class $\cecc(G)$ in a ``\v Cech space'' of the group of automorphisms
of $F$:
$$
H^1(U,V; \PAut(F)) = \Aut(F|_{U\cap V}) / \Aut(F|_U) \times \Aut(F|_V) .
$$
This class only depends on the isomorphism class of $G$ and we can see that
$\cecc(F^\alpha_{U,V}) = [\alpha]$, hence $\cecc$ is surjective. It remains to
see that $H^1(U,V; \PAut(F))$ is non trivial.  For this we define a ``monodromy
map'' $\mon_\gamma \cl H^1(U,V; \PAut(F)) \to \corC^\times$ along a path $\gamma
\cl [0,1] \to \SSi(F)$ and give a criterion to check that $\mon_\gamma$ is
surjective.

Let us explain the construction of $\mon_\gamma$.  The first point is to send
the presheaf of automorphisms $\PAut(F)$ to a more tractable sheaf.  We use the
bifunctor $\mu hom(F,G)$ which is a sheaf on $T^*M$ with support contained in
$\SSi(F) \cap \SSi(G)$ and which is a generalization of Sato's
microlocalization.  It is a refinement of the usual $\hom$ sheaf and comes with
the ``Sato's morphism''
$$
\rhom(F,G) \to \roim{\dot\pi_M{}}(\mu hom(F,G)) ,
$$
where $\dot\pi_M$ is the projection to the base.  When $\Lambda = \dot\SSi(F)$
is a smooth Lagrangian submanifold and $F$ is simple along $\Lambda$, we have a
canonical isomorphism $\mu hom(F,F)|_{\dT^*M} \simeq \corC_\Lambda$.  It follows
that $\PAut(F)$ has a natural morphism to
$\roim{\dot\pi_M{}}(\corC^\times_\Lambda)$.  We obtain a map $H^1(U,V; \PAut(F))
\to H^1(T^*U \cap \Lambda,T^*V \cap \Lambda; \corC^\times)$.  However, in the
case we are interested in, it sends $\cecc(G)$ to $0$: the curve
$\dot\pi_M(\Lambda)$ intersects $\sph \setminus U$, which will be a hypersurface
in our construction, twice with opposite signs. To obtain something non trivial
we restrict to a path in $\Lambda$ going only once through
$\opb{\dot\pi_M}(\sph \setminus U)$
and
we identify the ends.  We require a condition on the ends of our path so that we
can define a map $\mon_\gamma \cl H^1(U,V; \PAut(F)) \to
H^1(\cer;\corC^\times)$, where $\cer$ is the circle $\cer = [0,1]/(0\sim 1)$.
The condition is that $\gamma(0)$ and $\gamma(1)$ are ``$F$-linked points'': for
an open subset $U \subset M$, we say that $p_0$, $p_1 \in \Lambda \cap T^*U$ are
$F$-linked over $U$ if, for any automorphism $u$ of $F|_U$, we have $u^\mu_{p_0}
= u^\mu_{p_1}$, where $u^\mu$ is the section of $\mu hom(F,F)|_{\dT^*M} \simeq
\corC_\Lambda$ deduced from $u$ by Sato's morphism.  For such a path $\gamma$ we
can define our map $\mon_\gamma$ with value in $H^1(\cer;\corC^\times) \simeq
\corC^\times$.  We also give a criterion which implies that $\mon_\gamma$ is
surjective, hence that $\Der(U,V;F)$ has uncountably many isomorphism classes.
In particular this implies that the category of simple sheaves along $\Lambda$
has uncountably many isomorphism classes. Our criterion is roughly that the
situation looks like the above example of the constant sheaf on a circle in
$\R^2$.

\medskip

Now we describe the contents of the paper and say a bit more about $F$-linked
points.

In Section~\ref{sec:micshfth} we quickly recall some results of the microlocal
theory of sheaves.

In Section~\ref{sec:cohomdim2} we prove that the derived category of sheaves on
$\sph$ with constant cohomology sheaves has countably many objects and deduce
the same result for the category of simple sheaves on $\sph$ along a fiber
$T^*_{x_0}\sph$. This is in fact a general result about derived categories.  It
is probably already written but the author does not know a reference.

In Sections~\ref{sec:MV} and~\ref{sec:miclinkpt} we introduce the category
$\Der(U,V;F)$ and the monodromy map already described above.

In Section~\ref{sec:exmiclinkpt} we give two examples of $F$-linked points.  Let
$p_0 = (x_0;\xi_0)$ and $p_1 = (x_1;\xi_1) \in \Lambda$ be given.  We assume
that there is an embedding $i \cl \R \to M$ such that $i(0) = x_0$, $i(1) = x_1$
and $\opb{i}F$ has $\cor_I$ as direct summand where $I$ is an interval such that
$\ol{I} = [0,1]$. Then, under some genericity hypothesis, $p_0$ and $p_1$ are
$F$-linked.  In the other direction we assume that there is a map $q \cl M \to
\R$ such $p_0$ and $p_1$ are non-degenerate critical points of $q$ with respect
to $F$ in the sense that $\{(x;dq_x)\}$ meets transversally $\Lambda$ at $p_0$
and $p_1$.  We assume that $p_0$ and $p_1$ are the only critical points in this
sense on their level sets and that $\roim{q}F$ has a direct summand $\cor_I$
where $I$ is an interval with ends $q(p_0)$ and $q(p_1)$. Then, under some
genericity hypothesis, $p_0$ and $p_1$ are $F$-linked.  This last example is
similar to saying that two critical points of a Morse function are related by
the differential of the Morse-Barannikov complex.

In Section~\ref{sec:conshR} we apply Gabriel's Theorem and prove that the
constructible sheaves on $\R$ with coefficients in a field are direct sums of
constant sheaves on intervals.  This is false with integer coefficients as seen
with the example $\ker(m_2 \cl \Z_\R \to \Z_{[0,+\infty[})$, where $m_2 = 2 m_1$
and $m_1$ is the natural morphism $\Z_\R \to \Z_{[0,+\infty[}$.

In Section~\ref{sec:conshcercle} we extend the previous result to the circle:
the constructible sheaves on $\cer$ are sums of locally constant sheaves and
sheaves $\oim{e}\cor_I$ where $e \cl \R \to \cer$ is the natural quotient map
and $I$ is a bounded interval of $\R$.

In Section~\ref{sec:lccstshcercle} we introduce a group $H^i_{\alpha,r}(F)$ for
$F \in \Derb(\corC_\cer)$ and $(\alpha,r) \in \corC^\times\times\N^*$.  A
locally constant sheaf on $\cer$ is determined by its monodromy. Decomposing the
monodromy matrix into Jordan blocks, with eigenvalue $\alpha$ and size $r$, we
obtain a corresponding decomposition of our sheaf. We define $H^i_{\alpha,r}(F)$
so that its dimension is the multiplicity of the block $(\alpha,r)$ in $H^iF$.
Then we consider a morphism $u \cl \corC_I \to F$, where $I$ is an interval in
$\cer$, and describe in an easy case the link between $H^i_{\alpha,r}(F)$ and
$H^i_{\alpha,r}(\coker(u))$.

In Section~\ref{sec:shcyl} we consider a sheaf $F$ on $\cer \times \R$ with
finite dimensional stalks and whose microsupport is a Lagrangian submanifold in
generic position. For $t\in \R$ the restriction $F|_{\cer \times \{t\}}$ has a
decomposition into constant sheaves on intervals and locally constant sheaves.
For a very generic $t_0$ we can find an open interval $J$ around $t_0$ and a
diffeomorphism of $\cer \times J$ which sends $\SSi(F) \cap T^*(\cer \times J)$
to a product $\Lambda_0 \times J$. Then $F|_{\cer \times J}$ is of the form
$\opb{p}(G)$ for some $G\in \Derb(\corC_\cer)$ and $p$ the projection to $\cer$.
We consider the easiest singularity which can occur, namely the case where the
projection of $\SSi(F)$ to the base is a smooth curve in $\cer \times J$ which
is tangent at one point to $\cer \times \{t_0\}$.  Using the result on the
previous section, we explain very roughly how the decomposition of $F|_{\cer
  \times \{t\}}$ changes when $t$ crosses $t_0$.

In particular we can generalize the second example of $F$-linked points by
replacing the direct image $\roim{q}F$ with $R^iq_{\alpha,r}(F)$, where
$R^iq_{\alpha,r}(F)$ is defined so that its germs at $t$ compute
$H^i_{\alpha,r}(F|_{\cer \times \{t\}})$.  Adding an hypothesis on the
decomposition of $\roim{q}F$ we also prove the existence of a non trivial pair
of $F$-linked points.

In Section~\ref{sec:skyscrim} we apply the previous results to detect $F$-linked
points for $F \in \Derb(\cor_{\cer\times\R})$ when $\roim{q}(F)$ has a
skyscraper sheaf $\cor_{\{0\}}$ as direct summand and $\SSi(F)$ satisfies some
conditions.  Then we can apply the results on $\Der(U,V;F)$ and check that it
has uncountably many isomorphism classes. 

In Section~\ref{sec:proof} we prove the three cusps conjecture.  Since we
consider an isotopy of $\dT^*\sph$ which arises from an isotopy of $PT^*\sph$,
the equivalence from $\Derb_\Lambda(\cor_\sph)$ to
$\Derb_{\Phi_t(\Lambda)}(\cor_\sph)$ commutes with the duality.  Since
$\cor_{\{x_0\}}$ is self-dual, its image $F_t$ in
$\Derb_{\Phi_t(\Lambda)}(\cor_\sph)$ is also self-dual.  For any $q \cl \sph \to
\R$ we deduce that $\roim{q}(F_t)$ is self-dual.  We also have
$\rsect(\R;\roim{q}(F_t)) \simeq \corC$ and we deduce easily that, in the
decomposition of $\roim{q}(F_t)$ as a sum of $\cor_I[i]$'s, there is exactly one
interval $I$ which is not half-closed; moreover this interval is reduced to one
point.  Hence we can apply the results of Section~\ref{sec:skyscrim}.

\bigskip\noindent {\bf Acknowledgment.}  The starting point of the paper is
several discussions with Emmanuel Giroux and Emmanuel Ferrand. The author has
been very much stimulated by their interest for the question.  The author also
thanks Nicolas Vichery for pointing out that the structure of constructible
sheaves on $\R$ is an immediate application of Gabriel's Theorem, saving four
pages of pedestrian proof.

\section{Microlocal sheaf theory}
\label{sec:micshfth}

In this section, we quickly recall some definitions and results
from~\cite{KS90}, following its notations with the exception of slight
modifications.

Let $M$ be a manifold of class $C^\infty$.  We denote by $\pi_M \cl T^*M\to M$
the cotangent bundle of $M$.  If $N\subset M$ is a submanifold, we denote by
$T^*_NM$ its conormal bundle.  We identify $M$ with $T^*_MM$, the zero-section
of $T^*M$.  We set $\dT^*M = T^*M\setminus T^*_MM$ and we denote by
$\dot\pi_M\cl\dT^*M\to M$ the projection.  Let $f\cl M\to N$ be a morphism of
real manifolds. It induces morphisms on the cotangent bundles:
$$
T^*M \from[\; f_d\;] M\times_N T^*N \to[\;f_\pi\;] T^*N.
$$

In this paper the coefficient ring $\cor$ is assumed to be a field.  The results
of Section~\ref{sec:conshR} are false over $\Z$.  However the theory of
microsupport works for a commutative unital ring of finite global dimension.  We
denote by $\Mod(\cor_M)$ the category of sheaves of $\cor$-vector spaces on $M$.
We denote by $\Derb(\cor_M)$ the bounded derived category of $\Mod(\cor_M)$.

We use for short the term ``sheaf'' both for the objects of $\Mod(\cor_M)$ or
$\Derb(\cor_M)$, hoping the context will be clear enough.  In particular a
locally constant sheaf (or local system) in $\Derb(\cor_M)$ is an object $F$
such that any point of $M$ has a neighborhood $U$ such that $F|_U \simeq L_U$
for some $L\in \Derb(\cor)$.

If $M$ is real analytic we denote by $\Mod_\rc(\cor_M)$ the full subcategory of
$\Mod(\cor_M)$ of constructible sheaves, where a sheaf is constructible if there
exists a real analytic stratification of $M$ such that the restriction of the
sheaf to each stratum is constant of finite rank.  We let $\Derb_\rc(\cor_M)$ be
the full subcategory $\Derb(\cor_M)$ of complexes with constructible cohomology
sheaves.

We denote by $\cer$ the circle and by $\sph$ the $2$-sphere.  For a field $\cor$
we set $\cor^\times = \cor\setminus \{0\}$.  We let also $\Rp$ be the group of
positive numbers.

\subsubsection*{Microsupport}
We recall the definition of the microsupport (or singular support) $\SSi(F)$ of
$F\in \Derb(\cor_M)$, introduced by M.~Kashiwara and P.~Schapira in~\cite{KS82}
and~\cite{KS85}.

\begin{definition}{\rm (see~\cite[Definition~5.1.2]{KS90})}
Let $F\in \Derb(\cor_M)$ and let $p\in T^*M$. 
We say that $p\notin\SSi(F)$ if there exists an open neighborhood $U$ of $p$
such that, for any $x_0\in M$ and any real $C^1$-function $\phi$ on $M$
satisfying $d\phi(x_0)\in U$ and $\phi(x_0)=0$, we have
$(\rsect_{\{x;\phi(x)\geq0\}} (F))_{x_0}\simeq0$.

We set $\dot\SSi(F) = \SSi(F) \cap \dT^*M$.  If $\Lambda$ is a (closed conic)
subset of $\dT^*M$, we denote by $\Derb_\Lambda(\cor_M)$ the full subcategory of
$\Derb(\cor_M)$ consisting of the $F$ such that $\dot\SSi(F) \subset \Lambda$.
\end{definition}
In other words, $p\notin\SSi(F)$ if the sheaf $F$ has no cohomology 
supported by ``half-spaces'' whose conormals are contained in a 
neighborhood of $p$. The following properties are easy consequences of
the definition:
\begin{itemize}\renewcommand{\labelitemi}{-}
\item
$\SSi(F)$ is closed and $\R^+$-conic,
\item
$\SSi(F)\cap T^*_MM =\pi_M(\SSi(F))=\supp(F)$,
\item
the triangular inequality: if $F_1\to F_2\to F_3\to[+1]$ is a
distinguished triangle in  $\Derb(\cor_M)$, then 
$\SSi(F_2)\subset\SSi(F_1)\cup\SSi(F_3)$.
\end{itemize}

\begin{example}\label{ex:microsupport}
(i) Let $F\in \Derb(\cor_M)$. Then $\SSi(F) = \emptyset$ if and only if
$F\simeq 0$ and $\dot\SSi(F) = \emptyset$ if and only if the cohomology sheaves
$H^i(F)$ are local systems, for all $i\in \Z$.

\smallskip\noindent
(ii) If $N$ is a smooth closed submanifold of $M$ and $F=\cor_N$, then 
$\SSi(F)=T^*_NM$.

\smallskip\noindent
(iii) Let $\phi$ be $C^1$-function with $d\phi(x)\not=0$ when $\phi(x)=0$.
Let $U=\{x\in M;\phi(x)>0\}$ and let $Z=\{x\in M;\phi(x)\geq0\}$. 
Then 
\begin{align*}
\SSi(\cor_U) & =U\times_MT^*_MM
\cup \{(x;\lambda \, d\phi(x)); \; \phi(x)=0, \, \lambda\leq0\},\\
\SSi(\cor_Z) & =Z\times_MT^*_MM
\cup \{(x;\lambda \, d\phi(x));\; \phi(x)=0, \,  \lambda\geq0\}.
\end{align*}
(iv) Let $\lambda$ be a closed convex cone with vertex at $0$ in $E=\R^n$.
Then $\SSi(\cor_\lambda) \cap T^*_0E = \lambda^\circ$, the polar cone of
$\lambda$, that is,
\begin{equation}\label{eq:def_polar_cone}
\lambda^\circ = \{\xi\in E^*; \; \langle v,\xi \rangle \geq 0
 \text{ for all $v\in E\}$.}
\end{equation}
\end{example}

Let $M$ and $N$ be two manifolds. We denote by $q_i$ ($i=1,2$) the $i$-th
projection defined on $M\times N$ and by $p_i$ ($i=1,2$) the $i$-th projection
defined on $T^*(M\times N)\simeq T^*M\times T^*N$.

\begin{definition}
Let $f\cl M\to N$ be a morphism of manifolds and let $\Lambda\subset T^*N$
be a closed $\R^+$-conic subset. We say that $f$ is non-characteristic for
$\Lambda$ if $\opb{f_\pi}(\Lambda)\cap T^*_MN\subset M\times_NT^*_NN$.
\end{definition}

We denote by $\omega_M$ the dualizing complex on $M$.  Recall that $\omega_M$
is isomorphic to the orientation sheaf shifted by the dimension. We also use
the notation $\omega_{M/N}$ for the relative dualizing complex
$\omega_M\tens\opb{f}\omega_N^{\tens-1}$.  We have the duality functors
\begin{equation}\label{eq:dualfct}
\DD_M(\scbul)=\rhom(\scbul,\omega_M), \qquad
\DD'_M(\scbul)=\rhom(\scbul,\cor_M).
\end{equation}

For $\Lambda \subset T^*M$ we set $\Lambda^a = \{(x;-\xi)$; $(x;\xi) \in
\Lambda\}$.

\begin{theorem}[See {\cite[\S 5.4]{KS90}}]\label{th:opboim}
Let $f\cl M\to N$ be a morphism of manifolds, $F\in\Derb(\cor_M)$ and
$G\in\Derb(\cor_N)$.  Let $q_1\cl M\times N \to M$ and $q_2\cl M\times N \to N$
be the projections.
\begin{itemize}
\item [(i)] We have
\begin{gather*}
\SSi(\opb{q_1}F \ltens \opb{q_2}G) \subset \SSi(F)\times\SSi(G),  \\
\SSi(\rhom(\opb{q_1}F,\opb{q_2}G))  \subset \SSi(F)^a\times\SSi(G).
\end{gather*}
\item [(ii)] We assume that $f$ is proper on $\supp(F)$. Then
  $\SSi(\reim{f}F)\subset f_\pi\opb{f_d}\SSi(F)$, with equality if $f$ is a
  closed embedding.
\item [(iii)] We assume that $f$ is non-characteristic with respect to
  $\SSi(G)$. Then the natural morphism
  $\opb{f}G \tens \omega_{M/N}\to\epb{f}(G)$ is an isomorphism. Moreover
  $\SSi(\opb{f}G) \cup \SSi(\epb{f}G) \subset f_d\opb{f_\pi}\SSi(G)$.
\item [(iv)] We assume that $f$ is a submersion. Then
  $\SSi(F)\subset M\times_NT^*N$ if and only if, for any $j\in\Z$, the sheaves
  $H^j(F)$ are locally constant on the fibers of $f$.
\end{itemize}
\end{theorem}

\begin{corollary}\label{cor:opboim}
Let $F,G\in\Derb(\cor_M)$. 
\begin{itemize}
\item [(i)] We assume that $\SSi(F)\cap\SSi(G)^a\subset T^*_MM$. Then
  $\SSi(F\ltens G)\subset \SSi(F)+\SSi(G)$.
\item [(ii)] We assume that $\SSi(F)\cap\SSi(G)\subset T^*_MM$. Then \\
  $\SSi(\rhom(F,G))\subset \SSi(F)^a+\SSi(G)$.
\end{itemize}
\end{corollary}

\begin{corollary}\label{cor:opbeqv}
Let $I$ be a contractible manifold and let $p\cl M\times I\to M$ be the
projection.  If $F\in\Derb(\cor_{M\times I})$ satisfies
$\SSi(F)\subset T^*M\times T^*_II$, then $F\simeq\opb{p}\roim{p}F$.
\end{corollary}

\subsubsection*{Microlocalization, simple sheaves}
We refer to~\cite{KS90} for the definition and the main properties of the
microlocalization.  We recall some properties of the bifunctor $\mu hom$
introduced in~\cite{KS90}, which is a generalization of the microlocalization.
For $F,G \in \Derb(\cor_M)$, $\mu hom(F,G)$ is an object of
$\Derb(\cor_{T^*M})$.  We have
\begin{align}
\label{eq:proj_muhom_oim}
\roim{\pi_M} \mu hom (F,G) &\simeq \rhom (F,G)  , \\
\label{eq:proj_muhom}
\reim{\pi_M} \mu hom (F,G) &\simeq 
\DD'(F) \tens G \qquad \text{if $F$ is constructible,} \\
\label{eq:suppmuhom}
\supp \mu hom (F,G) & \subset \SSi(F)\cap \SSi(G) .
\end{align}
If $F$ is constructible, the restriction from $T^*M$ to $\dT^*M$ gives Sato's
distinguished triangle
\begin{equation*}\label{eq:SatoDTmuhom}
\DD'(F) \ltens G  \to \rhom (F,G) 
\to \roim{\dot\pi_M{}} (\mu hom (F,G) |_{\dT^*M}) \to[+1].
\end{equation*}
By the following result we can see $\mu hom$ as a microlocal version of
$\rhom$. Let $p\in T^*M$ be a given point. By the triangular inequality the full
subcategory $N_p$ of $\Derb(\cor_M)$ formed by the $F$ such that $p\not\in
\SSi(F)$ is triangulated and we can set $\Derb(\cor_M;p) = \Derb(\cor_M)/N_p$
(see the more general Definition~6.1.1 of~\cite{KS90}).  The functor $\mu
hom(\cdot,\cdot)_p$ induces a bifunctor on $\Derb(\cor_M;p)$ and we have

\begin{theorem}[Theorem~6.1.2 of~\cite{KS90}]
\label{thm:germemuhom}
  For all $F,G \in \Derb(\cor_M)$, the morphism $\Hom_{\Derb(\cor_M;p)}(F,G) \to
  H^0(\mu hom(F,G))_p$ is an isomorphism.
\end{theorem}

Let $\Lambda$ be a closed conic Lagrangian submanifold of $\dT^*M$.  We recall
the definition of simple and pure sheaves along $\Lambda$ and give some of their
properties.  We first recall some notations from~\cite{KS90}.  For a function
$\varphi\cl M\to \R$ of class $C^\infty$ we define
\begin{equation}
\label{eq:def_Lambdaphi}
  \Lambda_\varphi = \{(x; d\varphi(x)); \; x\in M\}  .
\end{equation}
We notice that $\Lambda_\varphi$ is a closed Lagrangian submanifold of $T^*M$.
For a given point $p=(x;\xi) \in \Lambda \cap \Lambda_\varphi$ we have the
following Lagrangian subspaces of $T_p(T^*M)$
\begin{equation}\label{eq:def_lambdas_p}
\lambda_0(p) = T_p(T_x^*M), \qquad 
\lambda_\Lambda(p) = T_p\Lambda, \qquad
\lambda_\varphi(p) = T_p\Lambda_\varphi.
\end{equation}
We recall the definition of the inertia index (see for example \S A.3
in~\cite{KS90}).  Let $(E,\sigma)$ be a symplectic vector space and let
$\lambda_1, \lambda_2, \lambda_3$ be three Lagrangian subspaces of $E$. We
define a quadratic form $q$ on $\lambda_1\oplus \lambda_2\oplus \lambda_3$
by $q(x_1,x_2,x_3) = \sigma(x_1,x_2) + \sigma(x_2,x_3) + \sigma(x_3,x_1)$.
Then $\tau_E(\lambda_1, \lambda_2, \lambda_3)$ is defined as the signature
of $q$, that is, $p_+-p_-$, where $p_\pm$ is the number of $\pm 1$ in
a diagonal form of $q$.
We set
\begin{equation}\label{eq:def_tau}
\tau_{\varphi} = \tau_{p,\varphi} 
= \tau_{T_pT^*M}(\lambda_0(p), \lambda_\Lambda(p), \lambda_\varphi(p)).
\end{equation}

\begin{proposition}[Proposition~7.5.3 of~\cite{KS90}]
\label{prop:inv_microgerm}
Let $\varphi_0, \varphi_1\cl M\to \R$ be functions of class $C^\infty$, let
$p=(x;\xi) \in \Lambda$ and let $F\in \Derb(\cor_M)$ be such that
$\Omega \cap \dot\SSi(F)$ is contained in $\Lambda$, for some neighborhood
$\Omega$ of $\Lambda$.  We assume that $\Lambda$ and $\Lambda_{\varphi_i}$
intersect transversally at $p$, for $i=0,1$.  Then $(\rsect_{\{ \varphi_1 \geq
  0\}}(F))_x$ is isomorphic to $(\rsect_{\{ \varphi_0 \geq 0\}}(F))_x [\demi
(\tau_{\varphi_0} - \tau_{\varphi_1})]$.
\end{proposition}

\begin{definition}[Definition~7.5.4 of~\cite{KS90}]
\label{def:simple_pure}
In the situation of Proposition~\ref{prop:inv_microgerm} we say that $F$ is
pure at $p$ if $(\rsect_{\{ \varphi_0 \geq 0\}}(F))_x$ is concentrated in a
single degree, that is, $(\rsect_{\{ \varphi_0 \geq 0\}}(F))_x \simeq L[d]$,
for some $L\in \Mod(\cor)$ and $d\in \Z$. If moreover $L\simeq \cor$, we say
that $F$ is simple at $p$.

If $F$ is pure (resp. simple) at all points of $\Lambda$ we say that it is pure
(resp. simple) along $\Lambda$.

We denote by $\Dersf_\Lambda(\cor_M)$ the full subcategory of
$\Derb_\Lambda(\cor_M)$ formed by the $F$ such that $F$ is simple along
$\dot\SSi(F)$ and the stalks of $F$ at the points of $M \setminus
\dot\pi_M(\Lambda)$ are finite dimensional.
\end{definition}
We know from~\cite{KS90} that, if $\Lambda$ is connected and $F\in
\Derb_{\Lambda}(\cor_M)$ is pure at some $p\in \Lambda$, then $F$ is in fact
pure along $\Lambda$. Moreover the $L\in \Mod(\cor)$ in the above definition is
the same at every point.

Since $\cor$ is a field, we know that $F$ is pure along $\Lambda$ if and only if
$\mu hom(F,F)|_{\dT^*M}$ is concentrated in degree $0$ and $F$ is simple along
$\Lambda$ if and only if the natural morphism $\cor_\Lambda \to \mu
hom(F,F)|_{\dT^*M}$ induced by~\eqref{eq:proj_muhom_oim} and the section $\id_F$
of $\rhom(F,F)$ is an isomorphism:
\begin{equation}
\label{eq:carc_Fsimple}
\cor_\Lambda  \isoto \mu hom(F,F) .
\end{equation}
For coefficients in a field the property~\eqref{eq:carc_Fsimple} could be a
definition of a simple sheaf.  However, using Definition~\ref{def:simple_pure}
it is possible to give a meaning to the integer $d$.  In~\cite{KS90} we find a
notion of ``simple sheaf with shift $k$ at the point $p$'', where $k$ is an
half-integer (the shift in Proposition~\ref{prop:inv_microgerm} is a sum of
half-integers).  We do not recall the convention about fixing the shift. We only
give the following examples.

\begin{example}
\label{ex:shift}
In Example~\ref{ex:microsupport}~(iii) the sheaf $\cor_U$ has shift $-1/2$ and
the sheaf $\cor_Z$ has shift $1/2$ (see Example~7.5.5 of~\cite{KS90}).

For $i \in \N$ we let $\Lambda_i \subset \Lambda$ be the set of points such that
the rank of $d\pi_M|_\Lambda$ is $(\dim M -1 -i)$.  For a generic closed conic
Lagrangian connected submanifold $\Lambda$ in $\dT^*M$, $\Lambda_0$ is an open
dense subset of $\Lambda$ and, for a given simple sheaf $F \in
\Derb_\Lambda(\cor_M)$, the shift of $F$ at $p$ is locally constant on
$\Lambda_0$ and changes by $1$ when $p$ crosses $\Lambda_1$.
\end{example}

\section{Cohomological dimension $2$}
\label{sec:cohomdim2}

The following result is well-known (see for example~\cite[Ex. 13.22]{KS05}).
\begin{lemma}
\label{lem:compl_scinde}
Let $\catc$ be an abelian category and $X \in \Derb(\catc)$ a complex such that
$\Ext^k(H^iX , H^jX) \simeq 0$ for all $i,j \in \Z$ and all $k\geq 2$.  Then
$X$ is split, that is, there exists an isomorphism
$X \simeq \bigoplus_{i \in \Z} H^iX [-i]$ in $\Derb(\catc)$.
\end{lemma}

This applies in particular to constructible sheaves in dimension $1$.  Indeed,
if $M = \R$ or $M = \cer$ and $\cor$ is a field, we have $\Ext^k(F, G) \simeq 0$
for all $F,G \in \Mod_\rc(\cor_M)$ and for all $k\geq 2$.  We deduce
\begin{lemma}
\label{lem:faiscdim1_compl_scinde}
Let $M = \R$ or $M = \cer$ and let $\cor$ be a field.  Then, for all
$F \in \Derb_\rc(\cor_M)$ we have $F \simeq \bigoplus_{i \in \Z} H^iF [-i]$ in
$\Derb_\rc(\cor_M)$.
\end{lemma}

In this section we consider a situation where $\Ext^2(H^iX , H^jX)$ is easy to
describe and $\Ext^k(H^iX , H^jX) \simeq 0$ for all $i,j \in \Z$ and $k\geq 3$.
We will use the following remark (see for
example~\cite[Ex.~13.20,~13.21]{KS05}).

For integers $a\leq b$ we denote by $\Der^{[a,b]}(\catc)$ the full subcategory
of $\Derb(\catc)$ formed by the complexes $X$ such that $H^iX \simeq 0$ for all
$i\not\in [a,b]$.  We have the truncation functors, for $n \in \Z$, $\tau_{\leq
  n} \cl \Derb(\catc) \to \Der^{]-\infty,n]}(\catc)$ and $\tau_{\geq n} \cl
\Derb(\catc) \to \Der^{[n,+\infty[}(\catc)$.  For any $X,Y \in \Derb(\catc)$ and
any isomorphism $u \cl X \isoto Y$ we have a natural isomorphism of
distinguished triangles
\begin{equation}
\label{eq:diag_ext2}
\vcenter{\xymatrix{
\tau_{\leq n} (X)  \ar[r] \ar[d]^{\tau_{\leq n}(u)}_\wr & X \ar[r] \ar[d]^u_\wr 
& \tau_{\geq n+1} (X) \ar[r]^{c(X)} \ar[d]^{\tau_{\geq n+1} (u)}_\wr 
& \tau_{\leq n} (X)[1] \ar[d]^{\tau_{\leq n}(u)[1]}_\wr \\
\tau_{\leq n} (Y)  \ar[r]  & Y \ar[r] 
& \tau_{\geq n+1} (Y)  \ar[r]^{c(Y)}  & \tau_{\leq n} (Y)[1] .
}}
\end{equation}
Let us set $A = \tau_{\geq n+1} (X)$ and $B = \tau_{\leq n} (X)[1]$.  It follows
from~\eqref{eq:diag_ext2} that $c(X)$ and $c(Y)$ give the same class in the
quotient $\Hom(A,B) / \Aut(A) \times \Aut(B)$, where $\Aut(A)$ is the
isomorphism group of $A$ and $\Aut(A) \times \Aut(B)$ acts on $\Hom(A,B)$ by
composition.

Conversely, let $\phi \in \Hom(A,B)$ be given and define $X_\phi$ such that
$X_\phi[1]$ is the cone of $\phi$.  Then $\tau_{\geq n+1} (X_\phi) \simeq A$,
$\tau_{\leq n} (X_\phi)[1] \simeq B$ and $c(X_\phi) = \phi$ through these
isomorphisms.  The object $X_\phi$ is determined by $\phi$ up to isomorphism.
Moreover, by the axiom of triangulated categories saying that we can extend a
square to a morphism of triangles, we see that $X_\phi \simeq X_\psi$ if $\phi$
and $\psi$ have the same image in $\Hom(A,B) / \Aut(A) \times \Aut(B)$.
It follows that the set of isomorphism classes of objects $Y$ such that
$\tau_{\geq n+1} (Y) \simeq A$ and $\tau_{\leq n} (Y)[1] \simeq B$ is in
bijection with $\Hom(A,B) / \Aut(A) \times \Aut(B)$.

\medskip

Let $\cor$ be a field and let $\catc$ be a $\cor$-linear abelian category.  By
this we mean that the $\Hom$ groups in $\catc$ are $\cor$-vector spaces and the
composition is $\cor$-bilinear.  Let $\Derb(\catc)$ be its bounded derived
category.  We assume that there exists $A\in \catc$ such that
\begin{equation}
\label{eq:hyp_A}
\left\{ \begin{aligned}
&\Hom(A,A) \simeq \Ext^2(A,A) \simeq \cor ,  \\
&\Ext^i(A,A) \simeq 0 \qquad \text{for $i\not=0,2$.}
\end{aligned} \right.
\end{equation}
We denote by $\Derb(\catc) \langle A \rangle$ the full subcategory
of $\Derb(\catc)$ formed by the complexes $X$ such that
$H^iX \simeq A^{d_i}$ for some $d_i \in \N$, for all $i\in \Z$.

The main result of this section is Proposition~\ref{prop:descr_DCA} which says
that $\Derb(\catc) \langle A \rangle$ contains countably many isomorphism
classes.  We first give some intermediate results.

\begin{lemma}
\label{lem:exuni_An}
  For each $n\in \N$ there exists a complex $A_n \in \Der^{[0,n]}(\catc)$ such
  that $H^i(A_n) \simeq A$ for all $i \in [0,n]$ and 
  \begin{equation}
    \label{eq:cohomAn}
    \RHom(A,A_n) \simeq \cor  \oplus \cor[-n-2] .
  \end{equation}
Moreover $A_n$ is unique up to isomorphism.
\end{lemma}
\begin{proof}
  We argue by induction on $n$.  For $n=0$ we must have $A_0 \simeq A$ since
  $A_0$ is concentrated in degree $0$ and $H^0(A_0) \simeq A$.  Conversely $A_0
  = A$ satisfies $\RHom(A,A_0) \simeq \cor \oplus \cor[-2]$ by the hypothesis on
  $A$. We prove the existence in~(i) and the unicity in~(ii).

  \sui (i) We assume the existence of $A_n$ for some $n\geq 0$.  By the
  hypothesis on $\RHom(A,A_n)$ there exists a non zero morphism $u_n \cl A[-n-1]
  \to A_n[1]$ which is unique up to multiplication by a scalar.  We define
  $A_{n+1}$ by the distinguished triangle
\begin{equation}
\label{eq:defAn}
A_{n+1} \to A[-n-1] \to[u_n] A_n[1] \to[+1] .
\end{equation}
The cohomology sequence gives $H^i(A_n) \isoto H^i(A_{n+1})$ for $i<n+1$,
$H^{n+1}(A_{n+1}) \isoto H^{n+1}( A[-n-1]) \simeq A$ and $H^{i}(A_{n+1}) \simeq
0$ for $i>n+1$.  Hence $A_{n+1} \in \Der^{[0,n+1]}(\catc)$ and $H^i(A_{n+1})
\simeq A$ for all $i \in [0,n+1]$.  Applying $\RHom(A,\cdot)$
to~\eqref{eq:defAn} we obtain the distinguished triangle
\begin{align*}
\RHom(A,A_{n+1}) \to (\cor&[-n-1] \oplus \cor[-n-3])  \\
& \to[u'_n] (\cor[1] \oplus \cor[-n-1])  \to[+1] ,
\end{align*}
where $u'_n$ is the composition with $u_n$.  Since $u_n \circ \id_{A[-n-1]} = u_n
\not=0$ we have $u'_n\not=0$.  Since $\cor$ is a field, the only possibility for
$u'_n \not=0$ is that $u'_n$ induces an isomorphism from $\cor[-n-1]$ to itself.
We deduce $\RHom(A,A_{n+1}) \simeq \cor \oplus \cor[-n-3]$, as required.

\sui (ii) We assume the unicity of $A_n$ for some $n\geq 0$.  Let us prove the
unicity of $A_{n+1}$. We assume that $A'$ satisfies the same properties as
$A_{n+1}$ and we set $A'' = \tau_{\leq n} A'$.  Then $A'' \in
\Der^{[0,n]}(\catc)$ and $H^i(A'') \simeq A$ for all $i \in [0,n]$.  This
implies $H^i\RHom(A,A'') \simeq 0$ for $i\geq n+3$.  We also have a
distinguished triangle
\begin{equation}
\label{eq:unicAn}
A' \to[w] A[-n-1] \to[v] A''[1] \to[+1] .
\end{equation}
Applying $\RHom(A,\cdot)$ we obtain the distinguished triangle
\begin{align*}
(\cor \oplus \cor[-n-3]) \to[w'] (\cor&[-n-1] \oplus \cor[-n-3])  \\
& \to[v'] \RHom(A,A''[1])  \to[+1] .
\end{align*}
If $w'=0$, then $\cor[-n-3]$ is a direct summand of $\RHom(A,A''[1])$,
contradicting $H^{n+4}\RHom(A,A'') \simeq 0$.  Hence $w'\not=0$ and it induces
an an isomorphism from $\cor[-n-3]$ to itself.  It follows that $\RHom(A,A'')
\simeq \cor \oplus \cor[-n-2]$.  Hence $A'' \simeq A_n$ by the unicity of $A_n$.

We have $v \not= 0$: otherwise $A[-n-1]$ is a direct summand of $A'$ and
$H^{n+1}\RHom(A,A') \not= 0$ which contradicts the hypothesis on $A'$.  It
follows that $v$ is a non zero multiple of $u_n$ and then that $A'' \simeq
A_{n+1}$.
\end{proof}

By the unicity part of Lemma~\ref{lem:exuni_An} we have isomorphisms $A_m[m-n]
\simeq \tau_{\geq n- m} A_n$ for all $m \leq n \in \N$.  However this
isomorphism is not canonical. Let $N\in \N$ be given.  For $n \leq N$ we set
$A_{N,n} = \tau_{\geq N-n} A_N[N-n]$.  The natural morphisms $\tau_{\geq k} \to
\tau_{\geq l}$ for $l\geq k$ give morphisms
\begin{equation}
\label{eq:defamn}
a_{m,n} \cl A_{N,n} \to  A_{N,m}[m-n]  \quad \text{for all $m\leq n \leq N$}
\end{equation}
with the property that $H^k(a_{m,n})$ is an isomorphism for all $k \geq n-m$.
We have the composition law
\begin{equation}
\label{eq:comp_amnp}
(a_{m,n}[n-p]) \circ a_{n,p} = a_{m,p}
\quad \text{for all $m\leq n \leq p \leq N$.}
\end{equation}
For $d,e \in \N$ and $m\leq n$, the functor $H^n$ gives a morphism
\begin{equation}
\label{eq:def_hn}
\begin{split}
\Hom(A_{N,n}^e, A_{N,m}^d[m-n]) &\to \Hom(H^n(A_{N,n}^e), H^m(A_{N,m}^d) ) \\
&\simeq \Hom(A^e, A^d)  \simeq \Hom(\cor^e ,\cor^d) 
\end{split}
\end{equation}
which has a splitting given by
\begin{equation}
\label{eq:splitting_hn}
\begin{split}
\alpha_{m,n}^{d,e} \cl \Hom(\cor^e ,\cor^d) & \to \Hom(A_{N,n}^e, A_{N,m}^d[m-n]) \\
u = (u_{ij}) &\mapsto \alpha_{m,n}^{d,e}(u) = (u_{ij} \cdot a_{m,n}),
\end{split}
\end{equation}
where we write $u\cl \cor^e \to \cor^d$ in matrix form.  The
morphism~\eqref{eq:def_hn} is of course compatible with the composition; the
morphism~\eqref{eq:splitting_hn} also by~\eqref{eq:comp_amnp}.

\begin{lemma}
\label{lem:act=upptriang}
Let $\{(n_i, d_i, s_i)\}_{i\in I}$ be a finite family of triples of integers
with $n_i, d_i \in \N\setminus \{0\}$ and $s_i \in \Z$.  We set $B =
\bigoplus_{i\in I} A_{n_i}^{d_i}[s_i]$ and we let $s$ be the maximal degree such
that $H^s(B)\not=0$. We let $J \subset I$ be the set of indices $i$ such that $s
= n_i-s_i$.  We order $J$ by $j \leq k$ if $n_{j} \leq n_{k}$.  Then $H^s(B)
\simeq \bigoplus_{j\in J} A^{d_j}$ and taking the cohomology in degree $s$ gives
a morphism
\begin{equation}
\label{eq:hn_surj2}
\Hom(B,B) \to \Hom(H^s(B),H^s(B))
\simeq \bigoplus_{j,k \in J} \Hom(\cor^{d_j},\cor^{d_k})
 \end{equation}
whose image contains all upper block triangular matrices (with respect to the
order on $J$).  Moreover the image of $\Aut(B)$ contains all invertible upper
block triangular matrices.
\end{lemma}
\begin{proof}
  We set $N = \max_{i\in I} n_i$ and we replace $A_{n_i}$ by $A_{N,n_i}$.  We
  set $B_0 = \bigoplus_{i\in J} A_{N,n_i}^{d_i}[s_i]$ and
  $B_1 = \bigoplus_{i\in I\setminus J} A_{N,n_i}^{d_i}[s_i]$.

  \sui (i) Let $u = (u_{jk})$ be a triangular matrix in the right hand side
  of~\eqref{eq:hn_surj2}. Using the notation~\eqref{eq:splitting_hn} we define
  $\alpha_0(u) \cl B_0 \to B_0$ such that its $(j,k)$ component is
  $(\alpha_0(u))_{jk} = \alpha_{n_{k}, n_{j}}^{d_k,d_j}(u_{jk})$ (this makes
  sense since $n_{j} \leq n_{k}$).  We choose $b_1 \cl B_1 \to B_1$ arbitrarily
  and we set $b = \alpha_0(u) \oplus b_1$.  Then the image of $b$
  by~\eqref{eq:hn_surj2} is $u$, which proves the first assertion of the lemma.

  \sui(ii) Now we assume that the matrix $u$ in~(i) is invertible.  Its inverse
  is also an upper triangular matrix and we can define $\alpha_0(u^{-1}) \cl B_0
  \to B_0$.  Since~\eqref{eq:splitting_hn} is compatible with the composition,
  we have $\alpha_0(u^{-1}) \circ \alpha_0(u) = \alpha_0(\id) = \id$, which
  proves that $\alpha_0(u)$ is an isomorphism.  We choose $b_1 = \id_{B_1}$ and
  we set $b = \alpha_0(u) \oplus b_1$.  Then $b$ is an isomorphism and its image
  by~\eqref{eq:hn_surj2} is $u$, which proves the second assertion of the lemma.
\end{proof}

\begin{proposition}
\label{prop:descr_DCA}
Let $A\in \catc$ be an object satisfying~\eqref{eq:hyp_A}. Then, for any $B \in
\Derb(\catc) \langle A \rangle$ there exists a finite family of triples of
integers $\{(n_i, d_i, s_i)\}_{i\in I}$ with $n_i, d_i \in \N\setminus \{0\}$
and $s_i \in \Z$ such that $B = \bigoplus_{i\in I} A_{n_i}^{d_i}[s_i]$.  In
particular $\Derb(\catc) \langle A \rangle$ contains countably many isomorphism
classes.
\end{proposition}
\begin{proof}
  We prove the proposition by induction on $l(B) = b-a$ where $a$ and $b$ are
  the minimal and maximal $k$ such that $H^k(B) \not=0$.  If $l(B) = 0$, the
  result is obvious.

  Let $n\geq 1$ be given.  We assume the result is true for the $B'$ with $l(B')
  \leq n-1$.  Let $B \in \Derb(\catc) \langle A \rangle$ with $l(B) = n$ be
  given. Up to shifting $B$ we can assume that $B \in \Der^{[0,n]}(\catc)$.  We
  set $B' = \tau_{\leq n-1}B$ and we write $H^n(B) = A^d$ for some $d\in \N$.  We
  have a distinguished triangle
\begin{equation}
\label{eq:tddescrDCA}
B' \to B \to  A^d[-n] \to[v] B'[1] .
\end{equation}
Then $B$ is determined up to isomorphism by the class of $v$ in the quotient
$$
Q = \Hom(A^d[-n] , B'[1]) / \Aut(A^d) \times \Aut(B').
$$
By the induction hypothesis we can write $B' = \bigoplus_{i\in I}
A_{n_i}^{d_i}[s_i]$ with the notations of the proposition.  We have $n_i-s_i
\leq n-1$ for all $i\in I$ and we let $J \subset I$ be the set of indices $j$
such that $n_j-s_j = n-1$.  We set $e = \sum_{j\in J} d_j$.
By~\eqref{eq:cohomAn} we have
\begin{align*}
 \Hom(A^d[-n] , B'[1]) 
& \simeq  \Hom(A^d[-n] ,  \bigoplus_{j\in J} A_{n_j}^{d_j}[1+s_j]) \\
& \simeq  \Hom(A^d , \bigoplus_{j\in J} A^{d_j}) \\
& \simeq  \Hom(\cor^d , \cor^e) .
\end{align*}
Moreover $ \Aut(A^d) \simeq \GL(\cor^d)$ acts in the obvious way on $\cor^d$
and, by Lemma~\ref{lem:act=upptriang}, the action of any matrix in $\T_e$ is
realized by $\Aut(B')$, where $\T_e$ is the set of upper triangular matrices in
$\GL(\cor^e)$.  Hence we have a surjective map
$$
\Hom(\cor^d, \cor^e) / \GL(\cor^d) \times \T_e \twoheadrightarrow Q.
$$
Now any element in the left hand side is represented by a matrix $P$ such that
$P_{ij} = 0$ if $i-j \not = d-e$ and $P_{ij} = 0$ or $1$ else.  In other words
we can decompose the morphism $v$ in~\eqref{eq:tddescrDCA} as $v = \sum_{j' \in
  J'} v_{j'}$, where $J'$ is a subset of $J$ and $v_{j'}$ is a non zero morphism
$v_{j'} \cl A[-n] \to A_{n_{j'}}[1+s_{j'}]$.  Then $v_{j'}$ is a non zero
multiple of the morphism $u_{n_{j'}}$ in the triangle~\eqref{eq:defAn} and the
cone of $v_{j'}$ is $A_{n_{j'}+1}[1+s_{j'}]$.  It follows that
$$
B = \bigoplus_{j' \in J'} A_{n_{j'}+1}[s_{j'}] 
\;\oplus\;
 A^{d-|J'|}[-n] 
\;\oplus\;
\bigoplus_{i\in I \setminus J'} A_{n_i}^{d_i}[s_i]
$$
and the induction proceeds.
\end{proof}

\subsection*{Example: sheaves on the sphere}

We let $\sph$ be the sphere in dimension $2$ and we let $P$ be a point of
$\sph$.  We set $\Lambda = \dT^*_P \sph$.  We recall the notation
$\Dersf_\Lambda(\cor_M)$ of Definition~\ref{def:simple_pure}.  We also denote by
$\Derb_{lc}(\cor_\sph)$ the full subcategory of $\Derb(\cor_\sph)$ formed by the
$F$ such that $\dot\SSi(F) = \emptyset$.  If $F \in \Derb_{lc}(\cor_\sph)$, then
$H^iF$ is a locally constant sheaf, for any $i\in \Z$. On $\sph$ the locally
constant sheaves are constant and we obtain $H^iF \simeq \cor_\sph^{d_i}$ for
some integer $d_i$.  Hence $\Derb_{lc}(\cor_\sph) = \Derb(\cor_\sph)\langle
\cor_\sph \rangle$.

\begin{corollary}
\label{cor:fais_sphere0}
For any field $\cor$ the category $\Dersf_\Lambda(\cor_\sph)$ has countably many
isomorphism classes.
\end{corollary}
\begin{proof}
  We write $\cor_P$ instead of $\cor_{\{P\}}$ for short.  We first prove that
  any object of $\Dersf_\Lambda(\cor_\sph)$ is up to shift the cone of a
  morphism between $\cor_P$ and an object of $\Derb_{lc}(\cor_\sph)$.  We
  conclude with Proposition~\ref{prop:descr_DCA}.

  \sui (i) Since $\cor_P$ and $F$ are both simple along $\Lambda$, up to
  shifting $F$ we have $\mu hom(\cor_P,F)|_\Lambda \simeq \cor_\Lambda$ and
  $\mu hom(F, \cor_P)|_\Lambda \simeq \cor_\Lambda$.  Let
  $\dot\pi \cl \dT^*\sph \to \sph$ be the projection.  We have the Sato
  distinguished triangle
$$
\DD'\cor_P \otimes F \to \rhom(\cor_P,F) 
\to[u] \roim{\dot\pi} \mu hom(\cor_P,F)
\to[v]  \DD'\cor_P \otimes F[1] .
$$
Applying $\DD'(\cdot)$ and the isomorphisms, for constructible sheaves $A,B$,
\begin{gather*}
\DD'(\DD'A \otimes B) \simeq \rhom(B,A) , \\
\DD'(\mu hom(A,B))|_{\dT^*\sph}
 \simeq \mu hom(B,A)|_{\dT^*\sph} \otimes \opb{\dot\pi} \omega_\sph   
\end{gather*}
we find the dual Sato triangle
$$
\DD'F \otimes \cor_P \to  \rhom(F, \cor_P) 
\to[v'] \roim{\dot\pi} \mu hom(F,\cor_P) \to[u']
\DD'F \otimes \cor_P [1] .
$$
The third term for both triangles is
$$
\roim{\dot\pi} \mu hom(\cor_P,F)
 \simeq \roim{\dot\pi} \mu hom(F,\cor_P)
 \simeq \cor_P \oplus \cor_P[-1]
$$
and we have either $H^0(u) \not= 0$ or $H^0(v') \not= 0$.  This gives a morphism
either $a \cl \cor_P \to F$ or $b \cl F \to \cor_P$ which is an isomorphism
along $\Lambda$.  Hence we have a distinguished triangle $\cor_P \to[a] F \to L
\to[c] \cor_P[1]$ or $F \to[b] \cor_P \to[d] L \to[+1]$ where $\SSi(L) \subset
T^*_\sph\sph$, that is, $L \in \Derb_{lc}(\cor_\sph)$.  In both cases $F$ appears
as the cone of a morphism $c \cl L \to \cor_P[1]$ or $d \cl \cor_P \to L$.

\sui (ii) We have already noticed that $\Derb_{lc}(\cor_\sph) =
\Derb(\cor_\sph)\langle \cor_\sph \rangle$.  Let us denote by $A_n$ the object
constructed from $A=\cor_\sph$ in Lemma~\ref{lem:exuni_An}.  By
Proposition~\ref{prop:descr_DCA} any $L \in \Derb_{lc}(\cor_\sph)$ is a sum $L =
\bigoplus_{i\in I} A_{n_i}[s_i]$ for some $n_i\in \N\setminus \{0\}$ (here we
may have $n_i = n_j$) and $s_i \in \Z$. For $n\in \N$ and $s\in \Z$ the group
$\Hom(A_n[s],\cor_P)$ is $\cor$ if $s \in [0,n]$ and $0$ else.  Hence
$\Hom(L,\cor_P) \simeq \cor^{J}$ where $J$ is some subset of $I$.  The group
$\Aut(L)$ contains at least $(\cor^\times)^I$ acting diagonally.  Hence the
quotient $\Hom(L,\cor_P) / \Aut(L) \times \Aut(\cor_P)$ is finite.  It follows
that, for a given $L$, we have finitely many $F$ appearing as the cone of a
morphism $L \to \cor_P$.  The same holds for morphisms $\cor_P \to L$ and the
corollary is proved.
\end{proof}

\section{Mayer-Vietoris}
\label{sec:MV}

Let $M$ be a manifold and $U,V$ two open subsets such that $M=U\cup V$.  For a
given $F\in \Derb(\cor_M)$ we consider the objects $G$ of $\Derb(\cor_M)$ which
are isomorphic to $F$ over $U$ and over $V$ and we associate a ``\v Cech class''
$\cecc(G)$ to such $G$.  In the next section we will give a criterion which
insures that $\cecc(G)$ takes at least $|\cor^\times|$ values, implying that we
have at least $|\cor^\times|$ non isomorphic such $G$.

\begin{definition}
  We let $\Der(U,V;F)$ be the full subcategory of $\Derb(\cor_M)$ consisting of
  the objects $G$ such that there exist two isomorphisms
  $\alpha_U \cl F|_U \isoto G|_U$ and $\alpha_V \cl F|_V \isoto G|_V$.
\end{definition}

\begin{definition}\label{def:HUVF}
  We let $\PAut(F)$ be the presheaf $U \mapsto \sect(U; \PAut(F)) = \Aut(F|_U)$.
  Then $\sect(U; \PAut(F))$ is a (non-commutative) group for the composition.
  For two open subsets $U,V$ of $M$ we define an action of $\sect(U; \PAut(F))
  \times \sect(V; \PAut(F))$ on $\sect(U\cap V; \PAut(F))$ by $(\alpha,\beta)
  \cdot \gamma = \alpha|_{U\cap V} \circ \gamma \circ \beta^{-1}|_{U\cap V}$ and
  we define
\begin{equation}
\label{eq:defHUVF}
\begin{split}  
H^1&(U,V; \PAut(F)) \\
&= \sect(U\cap V; \PAut(F)) / \sect(U; \PAut(F)) \times \sect(V; \PAut(F)) .
\end{split}
\end{equation}
\end{definition}
We remark that the inverse map $\gamma \mapsto \gamma^{-1}$ induces an
isomorphism $H^1(U,V; \PAut(F)) \simeq H^1(V,U; \PAut(F))$.

\begin{lemma}\label{lem:classbiendef}
  Let $G \in \Derb(\cor_M)$ be given.  We assume that we have four isomorphisms
  $\alpha_U, \beta_U \cl F|_U \isoto G|_U$ and
  $\alpha_V, \beta_V \cl F|_V \isoto G|_V$.  Then
  $\alpha_U^{-1}|_{U\cap V} \circ \alpha_V|_{U\cap V}$ and
  $\beta_U^{-1}|_{U\cap V} \circ \beta_V|_{U\cap V}$ have the same image in
  $H^1(U,V; \PAut(F))$.
\end{lemma}
\begin{proof}
This follows from  the equality in $\sect(U\cap V; \PAut(F))$
\begin{equation*}
\beta_U^{-1}|_{U\cap V} \circ \beta_V|_{U\cap V} 
= (\beta_U^{-1} \circ \alpha_U , \beta_V^{-1} \circ \alpha_V )
\cdot (\alpha_U^{-1}|_{U\cap V} \circ \alpha_V|_{U\cap V}) .
\end{equation*}
where $\cdot$ is the action introduced in Definition~\ref{def:HUVF}.
\end{proof}

\begin{definition}
\label{def:ceccG}
By Lemma~\ref{lem:classbiendef} it makes sense to define
\begin{equation}
\label{eq:defceccG}
\cecc(G) \eqdot [\alpha_U^{-1}|_{U\cap V} \circ \alpha_V|_{U\cap V}]
\quad \in \quad H^1(U,V; \PAut(F))
\end{equation}
for $G \in \Der(U,V;F)$ together with isomorphisms
$\alpha_U \cl F|_U \isoto G|_U$ and $\alpha_V \cl F|_V \isoto G|_V$. 
\end{definition}
Of course, if we have an isomorphism $\varphi \cl G \isoto G'$ in $\Der(U,V;F)$,
then $\cecc(G) = \cecc(G')$.  Indeed, setting $\alpha'_U = \varphi|_{U} \circ
\alpha_U$ and $\alpha'_V = \varphi|_{V} \circ \alpha_V$ we have
$(\alpha'_U)^{-1}|_{U\cap V} \circ \alpha'_V|_{U\cap V} = \alpha_U^{-1}|_{U\cap
  V} \circ \alpha_V|_{U\cap V}$.

Now, for a given $\alpha \in \sect(U\cap V; \PAut(F))$ we construct an object
$F_{U,V}^\alpha \in \Der(U,V;F)$ such that $\cecc(F_{U,V}^\alpha) = [\alpha]$.
We will use several times the isomorphisms
\begin{equation}
\label{eq:isoUcapV}
\Hom(F_{U\cap V}, F_U) \simeq \Hom(F_{U\cap V}, F_V) \simeq 
\Hom(F|_{U\cap V},F|_{U\cap V}).  
\end{equation}
We consider the Mayer-Vietoris triangle
\begin{equation}
\label{eq:MVtd}
F_{U\cap V} \to[\left(\begin{smallmatrix} a_U \\ a_V \end{smallmatrix}\right)]
F_U \oplus F_V \to[(b_U, -b_V)] F \to[+1],  
\end{equation}
where $a_U \in \Hom(F_{U\cap V}, F_U)$ is the natural morphism corresponding to
$\id_{F|_{U\cap V}}$ through~\eqref{eq:isoUcapV} and $a_V,b_U,b_V$ are defined
in the same way.

\begin{definition}
\label{def:modifMV}
Let $F\in \Derb(\cor_M)$ and let $U,V$ be a covering of $M$.  For a morphism
$\alpha \cl F|_{U\cap V} \to F|_{U\cap V}$ we define
$F_{U,V}^\alpha\in \Derb(\cor_M)$ as the cone of $(\begin{smallmatrix} a_U \\
  \alpha \end{smallmatrix})$, where we use the notations of~\eqref{eq:MVtd} and
the isomorphism~\eqref{eq:isoUcapV}.  Hence we have a distinguished triangle
\begin{equation}
\label{eq:MVtd2}
F_{U\cap V} \to[\left(\begin{smallmatrix} a_U \\ \alpha \end{smallmatrix}\right)]
F_U \oplus F_V \to[(b^\alpha_U, -b^\alpha_V)] F_{U,V}^\alpha \to[+1],  
\end{equation}
with $b^\alpha_U \in \Hom(F_U,F_{U,V}^\alpha)$ and $b^\alpha_V \in
\Hom(F_V,F_{U,V}^\alpha)$.  We remark that $F_{U,V}^\alpha$ is defined only up
to isomorphism.
\end{definition}

\begin{lemma}\label{lem:FUValpha}
  Let $\alpha \cl F|_{U\cap V} \isoto F|_{U\cap V}$ be an isomorphism.  Then
  $F_{U,V}^\alpha$ belongs to $\Der(U,V;F)$. More precisely the morphisms
  $b^\alpha_V|_V \cl F|_V \to F_{U,V}^\alpha|_V$ and $b^\alpha_U|_U \cl F|_U \to
  F_{U,V}^\alpha|_U$ are isomorphisms.  Moreover we have $\cecc(F_{U,V}^\alpha)
  = [\alpha^{-1}]$.
\end{lemma}
\begin{proof}
  (i) We first prove that $b^\alpha_V|_V$ is an isomorphism.  Through the
  natural identification $F_U|_V \simeq F_{U\cap V}|_V$ the morphism $a_U|_V$ is
  $\id_{F_{U\cap V}}|_V$. Hence the result follows from Lemma~\ref{lem:scindage}
  below applied to the restriction of~\eqref{eq:MVtd2} to $V$.

  \sui (ii) In the same way $\alpha|_U$ is an isomorphism from $F_{U\cap V}$ to
  $F_V|_{U\cap V} \simeq F_{U\cap V}$.  Hence $b^\alpha_U|_U$ is an isomorphism
  by Lemma~\ref{lem:scindage} again.

  \sui (iii) The composition of two consecutive morphisms in any triangle
  vanishes. Hence~\eqref{eq:MVtd2} gives $b^\alpha_U \circ a_U = b^\alpha_V
  \circ \alpha$.  Since $a_U|_{U\cap V} = \id_{F|_{U\cap V}}$ the last assertion
  follows.
\end{proof}

\begin{lemma}
\label{lem:scindage}
Let $\catc$ be an abelian category and let
\begin{equation*}
A
\to[\left(\begin{smallmatrix} u \\ v \end{smallmatrix}\right)]
B \oplus C \to[(b, c)] D \to[+1] 
\end{equation*}
be a distinguished triangle in $\Derb(\catc)$.  We assume that $u \cl A \to B$
is an isomorphism.  Then $c \cl C \to D$ is an isomorphism.
\end{lemma}
\begin{proof}
  Let us set $w= \left(\begin{smallmatrix} u \\ v \end{smallmatrix}\right)$ and
  $d = (b,c)$.  Since $d \circ w = 0$ and $u$ is an isomorphism, we find $b = -
  c \circ v \circ u^{-1}$.

  The morphism $ w' = (u^{-1}, 0) \cl B \oplus C \to A$ is a splitting of $w$.
  Hence there exists $d' = \left(\begin{smallmatrix} d_1 \\
      d_2 \end{smallmatrix}\right) \cl D \to B \oplus C$ such that $d \circ d' =
  \id_D$ and $\id_{B \oplus C} = w \circ w' + d' \circ d$.  This gives $b \circ
  d_1 + c \circ d_2 = \id_D$ and
\begin{equation}
\label{eq:scindage2}
\begin{pmatrix} \id_B & 0 \\  0 & \id_C \end{pmatrix}
= 
\begin{pmatrix} \id_B  +  d_1 \circ b &  d_1 \circ c
\\  v \circ u^{-1}  +  d_2 \circ b &   d_2 \circ c
\end{pmatrix} .
\end{equation}
The equalities $b = - c \circ v \circ u^{-1}$ and $b \circ d_1 + c \circ d_2 =
\id_D$ give
$$
c \circ (- v \circ u^{-1} \circ d_1 + d_2) = \id_D
$$
and the $(2,2)$-term in~\eqref{eq:scindage2} gives $d_2 \circ c =\id_C$.  Hence
$c$ is an isomorphism.
\end{proof}

\section{Microlocal linked points}
\label{sec:miclinkpt}

In this section we still consider a manifold $M$ with an open covering $M = U
\cup V$.  We let also $\Lambda$ be a smooth Lagrangian submanifold of $\dT^*M$
and we consider $F\in \Derb(\cor_M)$ such that $ \dot\SSi(F) = \Lambda$ and $F$
is simple along $\Lambda$.  We recall that $\mu hom(F,F)|_{\dT^*M} \simeq
\cor_\Lambda$.

In order to map the space $H^1(U,V; \PAut(F))$ to more easily described spaces,
we use the natural morphism from $\rhom(F,F)$ to $\roim{(\dot\pi_M)} \mu
hom(F,F)$.  Since $\mu hom(F,F)|_{\dT^*M} \simeq \cor_\Lambda$, we deduce a
morphism from $\PAut(F)$ to $\roim{(\dot\pi_M)} (\cor^\times_\Lambda)$ and then
a map
$$
H^1(U,V; \PAut(F)) \to H^1(\{U',V'\}; \cor^\times_\Lambda) ,
$$
where $U' = T^*U\cap \Lambda$, $V' = T^*V \cap \Lambda$ and $H^1(\{U',V'\};
\cor^\times_\Lambda)$ denotes the \v Cech cohomology.  However we loose to much
information in this way and we consider a map to another \v Cech group
(see~\eqref{eq:defmongamma}).  The construction of this map rely on the notion
of {\em linked points} in $\Lambda$ introduced in
Definition~\ref{def:linkedpair} below.

We first introduce a general notation.  For $G,G' \in \Derb(\cor_M)$ we recall
the canonical isomorphism~\eqref{eq:proj_muhom_oim}
\begin{equation}
\label{eq:Sato-bis}
\rhom(G,G') \simeq \roim{(\pi_M)} \mu hom(G,G') .
\end{equation}
For an open subset $W$ of $M$ and $p\in T^*W$, we deduce the morphisms
\begin{alignat}{2}
\label{eq:def_umu_general}
\Hom(G|_W,G'|_W) &\to H^0(T^*W; \mu hom(G,G')) , &\qquad u &\mapsto u^{\mu+} , \\
\label{eq:def_umup_general}
\Hom(G|_W,G'|_W) &\to  \mu hom(G,G')_p , &\qquad u &\mapsto u^{\mu+}_p . 
\end{alignat}
For our simple sheaf $F$ and for $p\in \dot\SSi(F)$ we obtain
\begin{alignat}{2}
\label{eq:def_umu}
\Hom(F|_W,F|_W) &\to H^0(\dT^*W; \cor_\Lambda) , &\qquad u &\mapsto u^\mu , \\
\label{eq:def_umup}
\Hom(F|_W,F|_W) &\to  \cor , &\qquad u &\mapsto u^\mu_p . 
\end{alignat}
In this case we have $u^{\mu+}_p = u^\mu_p \cdot (\id_F)^{\mu+}_p$.

\begin{definition}
\label{def:linkedpair}
Let $W\subset M$ be an open subset and let $p,q \in \Lambda \cap T^*W$ be given
points. We say that $p$ and $q$ are $F$-linked over $W$ if $u^\mu_p = u^\mu_q$
for all $u \in \sect(W; \PAut(F))$.
\end{definition}

We remark that $u^\mu_p$ only depends on the component of $\Lambda \cap T^*W$
containing $p$ and we could also speak of $F$-linked connected components of
$\Lambda \cap T^*W$.

By Theorem~\ref{thm:germemuhom} the functors $u \mapsto u^{\mu+}_p$ or
$u \mapsto u^\mu_p$ are well-defined in the quotient category
$\Derb(\cor_M;p)$. We can also express this result as follows.

\begin{lemma}
\label{lem:umup_DMp}
Let $G,G' \in\Derb(\cor_M)$ be such that
$\dot\SSi(G), \dot\SSi(G') \subset \Lambda$ and let $p\in \Lambda$ be given. We
assume that there exists a distinguished triangle $G \to[g] G' \to H \to[+1]$
and that $p\not\in \SSi(H)$.  Then the composition with $g$ induces
isomorphisms
$$
\mu hom(G,G)_p \isoto[a] \mu hom(G,G')_p \isofrom[b] \mu hom(G',G')_p
$$
and we have $a((\id_G)^{\mu+}_p) = g^{\mu+}_p = b((\id_{G'})^{\mu+}_p)$.  In
particular, if $G$ and $G'$ are simple and $u\cl G \to G$ and $v\cl G' \to G'$
satisfy $v \circ f = f \circ u$ then $u^\mu_p = v^\mu_p$.
\end{lemma}
\begin{proof}
  We apply the functor $\mu hom(G,\cdot)$ to the given distinguished triangle
  and we take the germs at $p$.  By the bound~\eqref{eq:suppmuhom} we have $\mu
  hom(G,H)_p \simeq 0$ and we deduce the isomorphism $a$. The isomorphism $b$ is
  obtained in the same way.  Then the last assertion follows from the relations
  $u^{\mu+}_p = u^\mu_p \cdot (\id_G)^{\mu+}_p$ and $v^{\mu+}_p = v^\mu_p \cdot
  (\id_{G'})^{\mu+}_p$.
\end{proof}

We recall that $U,V$ are two open subsets of $M$ such that $M = U \cup V$.  Let
$\gamma \cl [0,1] \to \Lambda$ be a path such that
\begin{equation}
\label{eq:hyppathgamma}
\begin{minipage}[c]{11cm}
  $p_0 = \gamma(0)$ and $p_1 = \gamma(1)$ belong to $(T^*U \setminus T^*V) \cap
  \Lambda$ and are $F$-linked over $U$.
  \end{minipage} 
\end{equation}
We define a circle $C$ by identifying $0$ and $1$ in $[0,1]$. The natural
orientation of $[0,1]$ induces an orientation on $C$.  We let $U'$ and $V'$ be
the images of $\opb{\gamma}(T^*U \cap \Lambda)$ and $\opb{\gamma}(T^*V \cap
\Lambda)$ by the quotient map $[0,1] \to C$.  We recall that $H^1(\{U',V'\};
\cor^\times_C)$ is the \v Cech cohomology of the covering $C = U' \cup V'$.
Since $C$ is oriented we have a canonical isomorphism $H^1(\{U',V'\};
\cor^\times_C) \simeq \cor^\times$.

For $u \in \sect(U;\PAut(F))$ the inverse image of $u^\mu$ by $\gamma$ gives a
well-defined section of $H^0(U';\cor^\times_C)$ because $p_0$ and $p_1$ are
$F$-linked over $U$. An element of $\sect(V;\PAut(F))$ induces a section of
$H^0(V';\cor^\times_C)$ in the same way because
$p_0, p_1 \not\in T^*V \cap \Lambda$.  We deduce a well-defined map
\begin{equation}
\label{eq:defmongamma}
\mon_\gamma \cl H^1(U,V; \PAut(F)) 
\to H^1(\{U',V'\}; \cor^\times_C) \simeq \cor^\times.
\end{equation}

Now we describe a situation where the map $\mon_\gamma$ is surjective.  We
assume that our open set $V$ contains a hypersurface $H$ and that there exist
$t_1 < t_2 < t_3 \in \mo]0,1[$ such that
\begin{equation}
  \label{eq:hypUVF}
\left\{  \hspace{-4mm} \begin{minipage}[c]{11cm}
    \begin{itemize}
\item [(a)] $H' = M \setminus U$ is contained in $H$,
\item [(b)] $F|_V$ has a decomposition $F|_V \simeq F' \oplus F''$ such that $H
  \cap \supp F' \subset H'$,
\item [(c)] $V \setminus H$ has two connected components, say $V^+$ and $V^-$,
\item [(d)] $\opb{\gamma}(\SSi(F')) = \mo]t_1,t_3[,
\quad \opb{\gamma}(\SSi(F') \cap T^*V^+) = \mo]t_1,t_2[$
and $\opb{\gamma}(\SSi(F') \cap T^*V^-) = \mo]t_2,t_3[$.
\end{itemize}
  \end{minipage} \right.
\end{equation}
\begin{remark}
\label{rem:hypUVF_comconn}
By~(b) we have $\SSi(F|_V) = \SSi(F') \cup \SSi(F'')$ and, moreover, both $\mu
hom(F',F')$ and $\mu hom(F'',F'')$ are direct summands of $\mu hom(F,F)$.  For
any sheaf $G$ we have $\SSi(G) = \supp(\mu hom(G,G))$.  Since $F$ is simple we
deduce that a point of $\dot\SSi(F)\cap \dT^*V$ cannot belong to $\SSi(F')$ and
$\SSi(F'')$.  Hence, outside the zero section we actually have a disjoint union
$\dot\SSi(F|_V) = \dot\SSi(F') \sqcup \dot\SSi(F'')$.  In particular
$\dot\SSi(F')$ is a union of connected components of $T^*V \cap \Lambda$.
Then~(d) says that $\opb{\gamma}(\SSi(F')) = \mo]t_1,t_3[$ is a single connected
component of $\opb{\gamma}(T^*V \cap \Lambda)$.
\end{remark}

By~(b) and~(c) we have $F'|_{V\setminus H'} \simeq F'_{V^+} \oplus
F'_{V^-}$. Hence
\begin{equation}
\label{eq:decompFUV}
F|_{U \cap V} \simeq F'_{V^+} \oplus F'_{V^-} \oplus F''|_{U\cap V} .
\end{equation}
Let $a\in \cor^\times$ be given.  According to the
decomposition~\eqref{eq:decompFUV} we define a diagonal automorphism
$\alpha(a) \cl F|_{U\cap V} \to F|_{U\cap V}$ by
\begin{equation}
\label{eq:defalphaFUV}
\alpha(a) = \begin{pmatrix}
  a \cdot \id_{F'_{V^+}}  \\ &  \id_{F'_{V^-}} \\ &&\id_{F''} 
\end{pmatrix}.
\end{equation}

\begin{proposition}
\label{prop:mongammasurj}
We make the hypothesis~\eqref{eq:hyppathgamma} and~\eqref{eq:hypUVF} and, for a
given $a\in \cor^\times$, we define $\alpha(a) \cl F|_{U\cap V} \to F|_{U\cap
  V}$ by~\eqref{eq:defalphaFUV}.  We also denote by $\alpha(a)$ its class in
$H^1(U,V; \PAut(F))$.  Then $\mon_\gamma(\alpha(a)) = a$.
\end{proposition}
\begin{proof}
  We use the notations $C$, $U'$ and $V'$ described
  after~\eqref{eq:hyppathgamma}.  We identify the points $t_i$ introduced
  in~\eqref{eq:hypUVF} with their images in $C$.  By~(d) in~\eqref{eq:hypUVF}
  the intervals $]t_1,t_2[$ and $]t_2,t_3[$ are connected components of $U'\cap
  V'$.  By Remark~\ref{rem:hypUVF_comconn} and by the definition of $\alpha(a)$,
  the section $\alpha(a)^\mu$ of $\mu hom(F,F) \simeq \cor_\Lambda$ is the
  scalar $a$ over $]t_1,t_2[$ and $1$ over all other components of $U'\cap
  V'$. It follows that the \v Cech cocycle defined by $\alpha(a)^\mu$ represents
  $a \in H^1(\{U',V'\}; \cor^\times_C) \simeq \cor^\times$.
\end{proof}

\begin{corollary}
\label{cor:clisoDUVF}
Under the hypothesis~\eqref{eq:hyppathgamma} and~\eqref{eq:hypUVF} the category
$\Der(U,V;F)$ contains at least $| \cor^\times |$ isomorphism classes.
\end{corollary}

\section{Examples of microlocal linked points}
\label{sec:exmiclinkpt}

We begin with the following easy example over $\R$. In higher dimension we give
Propositions~\ref{prop:mulink-restrdim1} and~\ref{prop:mulink-projdim1} below
which reduce to this case by inverse or direct image.

\begin{proposition}
\label{prop:mulink-dim1}
  Let $t_0 \leq t_1 \in \R$ and $p_0 = (t_0;\tau_0)$, $p_1 = (t_1;\tau_1) \in
  \dT^*\R$ be given.  Let $F \in \Derb(\cor_\R)$ be a constructible sheaf with
  $p_0 , p_1 \in \SSi(F)$ such that $F$ is simple at $p_0, p_1$.  We assume that
  there exists a decomposition $F \simeq G \oplus \cor_I[d]$ where $G \in
  \Derb(\cor_\R)$, $d\in \Z$ and $I$ is an interval with ends $x_0,x_1$ such
  that $p_0,p_1 \in \SSi(\cor_I)$.  Then $p_0$ and $p_1$ are $F$-linked over any
  interval containing $\ol{I}$.
\end{proposition}
\begin{proof}
  Let $i \cl \cor_I[d] \to F$ and $q \cl F \to \cor_I[d]$ be the morphism
  associated with the decomposition of $F$.  Then $i$ and $q$ give a morphism
  $\Hom(F,F) \to \Hom(\cor_I,\cor_I) \simeq \cor$.  Since $F$ is simple at
  $p_k$, for $k=0,1$, and $p_k \in \SSi(\cor_I)$, we have $p_k \not\in \SSi(G)$.
  Hence $u^\mu_{p_k} = ({}_Iu)^\mu_{p_k}$ by Lemma~\ref{lem:umup_DMp}.  Since
  $\Hom(\cor_I,\cor_I) \simeq \cor$ we have $({}_Iu)^\mu_{p_0} =
  ({}_Iu)^\mu_{p_1}$ for any $u$ defined in a neighborhood of $\ol{I}$.  The
  result follows.
\end{proof}

Let $f \cl M \to N$ be a morphism of manifolds. We recall the notations $f_d \cl
M \times_N T^*N \to T^*M$ and $f_\pi \cl M \times_N T^*N \to T^*N$ for the
natural maps induced on the cotangent bundles.

We study easy cases of inverse or direct image of simple sheaves.  We let $F \in
\Derb(\cor_M)$ be such that $\Lambda = \dot\SSi(F)$ is a smooth Lagrangian
submanifold and $F$ is simple along $\Lambda$.  More general, but local,
statements are given in~\cite{KS90} (see Corollaries~7.5.12 and~7.5.13).

\begin{lemma}
\label{lem:bonne_im_inv}
Let $i \cl L \to M$ be an embedding and let $x_0$ be a point of $L$.  We assume
that there exist a neighborhood $U \subset M$ of $i(x_0)$ and a submanifold $N
\subset U$ such that $\Lambda \cap \dT^*U \subset \dT^*_NU$ and $N$ is
transversal to $L$ at $x_0$.  Then, up to shrinking $U$ around $x_0$, the set
$\Lambda' = \dot\SSi((\opb{i}F)|_{U \cap L})$ is contained in $\dT^*_{N \cap
  L}L$ and $(\opb{i}F)|_{U \cap L}$ is simple along $\Lambda'$.

Moreover, if $u \cl F \to F$ is defined in a neighborhood of $x_0$ and $v \cl
\opb{i}F \to \opb{i}F$ is the induced morphism, then for any $p = (x_0;\xi_0)
\in \Lambda$ and $q = (x_0; i_d(\xi_0))$ we have $v^\mu_q = u^\mu_p$.
\end{lemma}
We note that the inclusion $\Lambda \cap T^*U \subset \dT^*_NU$ implies that
$\Lambda$ is a union of components of $\dT^*_NU$ (hence this is an equality if
$N$ is connected of codimension $\geq 2$).
\begin{proof}
  Up to shrinking $U$ we can find a submersion $f \cl U \to L\cap U$ such that
  $N = \opb{f} (L \cap N)$.  By Corollary~\ref{cor:opbeqv} we can write $F|_U
  \simeq \opb{f}G$ for some $G \in \Derb(\cor_{L\cap U})$.  Then $\mu hom(F,F)
  \simeq \oim{(f_d)} \opb{f_\pi}(\mu hom(G,G))$ and the lemma follows.
\end{proof}

\begin{lemma}
\label{lem:bonne_im_dir}
Let $q \cl M \to \R$ be a function of class $C^2$ and let $t_0$ be a regular
value of $q$.  We assume that there exist an open interval $J$ around $t_0$, a
connected hypersurface $L$ of $U \eqdot \opb{q}(J)$ and a connected component
$\Lambda_0$ of $\dT^*_LU$ such that
\begin{itemize}
\item [(i)] $q|_U$ is proper on $\supp F$,
\item [(ii)] $q|_L$ is Morse with a single critical point $x_0$ and $q(x_0) =
  t_0$,
\item [(iii)] $\Lambda_0 \subset \Lambda \cap T^*U$ and $T^*_{x_0}M \cap \Lambda
  = T^*_{x_0}M \cap \Lambda_0$,
\item [(iv)] $((\Lambda \cap T^*U) \setminus \Lambda_0) \cap q_d(M \times_\R
  \dT^*\R) = \emptyset$.
\end{itemize}
Let $p = (x_0;\xi_0)$ be the point of $\Lambda_0$ ($\xi_0$ is unique up to a
positive scalar) above $x_0$. We have $p \in \im(q_d)$ and we set $p' =
(t_0;\tau_0) = q_\pi(\opb{q_d}(p))$.  Then $p' \in \SSi(\roim{q}F)$ and
$\SSi(\roim{q}F)$ is simple at $p'$.  Moreover, for any morphism $u \cl F|_U \to
F|_U$, denoting by $v = \roim{q}(u) \cl \roim{q}F|_J \to \roim{q}F|_J$ the
induced morphism, we have $v^\mu_{p'} = u^\mu_p$.
\end{lemma}
\begin{proof}
  We set $N = \opb{q}(t_0)$.  Then $N$ is a smooth hypersurface of $M$ and, up
  to shrinking $J$ and restricting to a neighborhood of $\supp(F) \cap U$ we can
  assume that $U = N \times J$.  We can also find a neighborhood $V$ of $x_0$ in
  $N$ and a Morse function $\varphi \cl V \to \R$ such that, in some
  neighborhood of $x_0$, $L$ is the graph of $\varphi$.

  By~(iii) we can find an small open ball $B \subset N$ around $x_0$ such that
  $T^*(B \times J) \cap \Lambda = T^*(B \times J) \cap \Lambda_0$.  We set $W =
  B \times J$ and $Z = U \setminus W$.  We have a distinguished triangle $F_W
  \to F \to F_Z \to[+1]$.  By~(iii) again $\SSi(F_Z) \subset \SSi(F) \cup
  (\Lambda_0 + T^{*,out}_{\partial W}U)$, where $T^{*,out}_{\partial W}U$ is the
  outer conormal bundle of $\partial W$ in $U$.  By~(iv) we deduce that
  $\dot\SSi(F_Z) \cap q_d(M \times_\R \dT^*\R) = \emptyset$. Hence
  $\roim{q}(F_Z)$ is a constant sheaf on $J$ and we can assume from the
  beginning that $F = F_W$.

  By~(iii) we have $F|_W \simeq G \oplus \cor_{W'}[d]$, where $G \in
  \Derb(\cor_W)$ has constant cohomology sheaves, $d$ is some integer and $W'$
  is one of the following subsets of $W$: $W^+ = \{(x,t) \in W$;
  $t>\varphi(x)\}$, $\ol{W^+}$, $W \setminus W^+$ or $W \setminus \ol{W^+}$.
  Now the problem is reduced to the computation of $\roim{q}(\cor_{W'})$ which
  is a classical computation in Morse theory (in our case this is done in the
  proof of Proposition~7.4.2 of~\cite{KS90}).
\end{proof}

Proposition~\ref{prop:mulink-dim1} and Lemmas~\ref{lem:bonne_im_inv}
and~\ref{lem:bonne_im_dir} imply the following results.
\begin{proposition}
\label{prop:mulink-restrdim1}
Let $I$ be an interval containing $0$ and $1$ and let $i \cl I \to M$ be an
embedding.  We set $x_0 = i(0)$ and $x_1 = i(1)$.  For $k=0,1$ we assume that
there exist a neighborhood $U_k \subset M$ of $x_k$ and a hypersurface $N_k$ of
$U_k$ such that $\Lambda \cap \dT^*U_k \subset \dT^*_{N_k}U_k$ and $N_k$ is
transversal to $I$ at $x_k$.  We assume moreover that $\opb{i}F$ has a
decomposition $\opb{i}F \simeq G \oplus \cor_J[d]$ where $G \in \Derb(\cor_\R)$,
$d\in \Z$ and $J$ is an interval with ends $0,1$.  Let $p_0 , p_1 \in \Lambda$
be such that $i_d(p_0), i_d(p_1) \in \dot\SSi(\cor_J)$.  Then $p_0$ and $p_1$
are $F$-linked over any open subset containing $i(I)$.
\end{proposition}

\begin{proposition}
\label{prop:mulink-projdim1}
Let $q \cl M \to \R$ be a map of class $C^2$ and let $t_0 \leq t_1$ be regular
values of $q$.  For $k=0,1$ we assume that there exist an open interval $J_k$
around $t_k$, a connected hypersurface $L_k$ of $U_k \eqdot \opb{q}J_k$ and a
connected component $\Lambda_k$ of $\dT^*_{L_k}U_k$ such that the
hypothesis~(i)-(iv) of Lemma~\ref{lem:bonne_im_dir} are satisfied.  We assume
moreover that $\roim{q}F$ has a decomposition $\roim{q}F \simeq G \oplus
\cor_J[d]$ where $G \in \Derb(\cor_\R)$, $d\in \Z$ and $J$ is an interval with
ends $t_0, t_1$.  Let $p_k = (x_k;\xi_k) \in \Lambda$ be such that $q(x_k) =
t_k$ and $q_\pi(\opb{q_d}(p_k)) \in \dot\SSi(\cor_J)$. Then $p_0$ and $p_1$ are
$F$-linked over any open subset containing $\opb{q}([t_0,t_1])$.
\end{proposition}

\section{Constructible sheaves on $\R$}
\label{sec:conshR}

In this section we apply Gabriel's theorem and obtain that a constructible sheaf
on $\R$ with coefficients in a field $\cor$ has a decomposition as a direct sum
of constant sheaves on intervals.

We first recall the part of Gabriel's theorem that we need.  We follow the
presentation of Brion's lecture~\cite{B12} on the subject and we refer the
reader to {\em loc. cit.} for further details.

A {\em quiver} is a finite directed graph, that is, a quadruple $Q = (Q_0$,
$Q_1$, $s, t)$ where $Q_0$, $Q_1$ are finite sets (the set of vertices,
resp. arrows) and $s, t \cl Q_1 \to Q_0$ are maps assigning to each arrow its
source, resp. target.  A {\em representation} of a quiver $Q$ consists of a
family of $\cor$-vector spaces $V_i$ indexed by the vertices $i\in Q_0$,
together with a family of linear maps $f_\alpha \cl V_{s(\alpha)} \to
V_{t(\alpha)}$ indexed by the arrows $\alpha\in Q_1$.

We are only interested in the following example given by the constructible
sheaves on $\R$.  Let $\ul{x} = \{x_1 < \cdots < x_n \}$ be a finite family of
points in $\R$.  We set $x_0 = -\infty$, $x_{n+1} = +\infty$ and we denote by
$\Mod_{\ul{x}}(\cor_\R)$ the category of sheaves $F$ on $\R$ such that the
stalks $F_y$ are finite dimensional for all $y\in \R$ and the restrictions
$F|_{]x_k,x_{k+1}[}$ are constant for $k = 0,\ldots,n$.  Such a sheaf $F$ is
determined by the data of the spaces of sections $V'_i = F(]x_i,x_{i+2}[) \isoto
F_{x_{i+1}}$ for $i=0,\ldots,n-1$ and $V''_i = F(]x_i,x_{i+1}[)$ for
$i=0,\ldots,n$, together with the restriction maps $V'_i \to V''_i$ and $V'_i
\to V''_{i+1}$ for $i=0,\ldots,n-1$.  Conversely, any such family of vector
spaces $V'_i$, $V''_i$ and linear maps defines a sheaf in
$\Mod_{\ul{x}}(\cor_\R)$.  Hence $\Mod_{\ul{x}}(\cor_\R)$ is equivalent to the
category of representations of the following quiver $Q = (Q_0$, $Q_1$, $s, t)$
where $Q_0 = \{0,\ldots,2n\}$ and there is exactly one arrow in $Q_1$ from
$2j-1$ to $2j-2$ and from $2j-1$ to $2j$, for $j=1,\ldots,n$.  In this
equivalence a sheaf $F \in \Mod_{\ul{x}}(\cor_\R)$ gives a representation
$(\{V_i\}, f_\alpha)$ of $Q$ where
\begin{alignat*}{2}
V_{2j-1} &= F(]x_{j-1} ,x_{j+1}[) &\qquad &\text{for $j=1,\ldots,n$,} \\
V_{2j} &= F(]x_{j} ,x_{j+1}[) &\qquad &\text{for $j=0,\ldots,n$,}
\end{alignat*}
and the $f_\alpha$ are the restriction maps.

Our quiver $Q$ is of type $A_{2n+1}$ and we can apply Gabriel's theorem. We only
recall the part we need.  For a representation $(\{V_i\}, f_\alpha)$ the
dimension vector is $(\dim V_i)_{i \in Q_0}$.  The Tits form is the quadratic
form on $\R^{Q_0}$ defined by $q_Q(\ul{d}) = \sum_{i\in Q_0} d_i^2 -
\sum_{\alpha\in Q_1} d_{s(\alpha)} d_{t(\alpha)}$.  In our case we have $Q_0 =
\{0,\ldots,2n\}$ and we find
\begin{equation}
\label{eq:Titsform}
q_Q(\ul{d}) = \frac{1}{2} 
( d_0^2 + \sum_{i=1}^{2n} (d_i - d_{i-1})^2 + d_{2n}^2) .  
\end{equation}

Now we can state part of Gabriel's theorem.  A representation is said {\em
  indecomposable} if it cannot be split as the sum of two non zero
representations. A representation $V$ is {\em Schur} if $\Hom(V,V) \simeq \cor
\cdot \id_V$.

\begin{theorem}[see Theorem 2.4.3 in~\cite{B12}]
\label{thm:gabriel}
Assume that the Tits form $q_Q$ is positive definite. Then:
\\
(i) Every indecomposable representation is Schur and has no non-zero
self-extensions.
\\
(ii) The dimension vectors of the indecomposable representations are exactly
those $\ul{d} \in \N^{Q_0}$ such that $q_Q(\ul{d}) = 1$.
\\
(iii) Every indecomposable representation is uniquely determined by its
dimension vector, up to isomorphism.
\end{theorem}

By~\eqref{eq:Titsform} a dimension vector $\ul{d}$ satisfies the condition
$q_Q(\ul{d}) = 1$ if and only if there exist $i\leq j \in \{0,\ldots,2n\}$ such
that $d_k = 1$ if $i\leq k \leq j$ and $d_k = 0$ else.  By Gabriel's Theorem,
for any such $\ul{d}$ we have exactly one indecomposable representation whose
dimension vector is $\ul{d}$.  We see easily that the corresponding sheaves in
$\Mod_{\ul{x}}(\cor_\R)$ are the constant sheaves $\cor_I$ on the intervals $I$
with ends $-\infty$, $x_1,\dots, x_n$ or $+\infty$ (the intervals can be open,
closed or half-closed).

We also remark that~(i) of Theorem~\ref{thm:gabriel} is given in our case by the
following easy result and the fact that $\Ext^1(\cor_I,\cor_I) \simeq 0$.
\begin{lemma}
\label{lem:morph_deux_int0}
Let $I,J$ be two intervals of $\R$.  Then
$$
\Hom(\cor_I,\cor_J) \simeq
\begin{cases}
  \cor  & \text{if $I\cap J$ is closed in $I$ and open in $J$,} \\
  0 & \text{else.}
\end{cases}
$$
In particular, if $I$ and $J$ are distinct, then we have $\Hom(\cor_I,\cor_J)
\simeq 0$ or $\Hom(\cor_J,\cor_I) \simeq 0$ and, for all morphisms $f \cl \cor_I
\to \cor_J$, $g \cl \cor_J \to \cor_I$, we have $g \circ f = 0$.
\end{lemma}

Any constructible sheaf with compact support belongs to $\Mod_{\ul{x}}(\cor_\R)$
for some finite family of points $\ul{x}$. Hence Gabriel's Theorem gives the
following decomposition result for constructible sheaves.  The unicity follows
from~(i) of Theorem~\ref{thm:gabriel} and the Krull-Schmidt theorem.

\begin{corollary}
\label{cor:cons_sh_R}
We recall that $\cor$ is a field. Let $F \in \Mod_\rc(\cor_\R)$.  We assume
that $F$ has compact support.  Then there exist a finite family of intervals
$I_a$ and integers $n_a \in \N$, $a\in A$, such that
\begin{equation}
\label{eq:cons_sh_R1}
F \simeq \bigoplus_{a\in A} \cor^{n_a}_{I_a} .
\end{equation}
Moreover this decomposition is unique in the following sense.  If we have
another decomposition $F \simeq \bigoplus_{b\in B} \cor^{m_b}_{J_b}$
like~\eqref{eq:cons_sh_R1} and we assume that the intervals $I_a$, $a\in A$, are
distinct as well as the intervals $J_b$, $b\in B$, then there exists a bijection
$\sigma \cl A \isoto B$ such that $J_{\sigma(a)} = I_a$ and $m_{\sigma(a)} =
n_a$ for all $a\in A$.
\end{corollary}

\section{Constructible sheaves on the circle}
\label{sec:conshcercle}

In this section we extend Corollary~\ref{cor:cons_sh_R} to the circle.  Like in
the case of $\R$ the result is also a particular case of quiver representation
theory and Auslander-Reiten theory. The author is grateful to Claire Amiot for
explanations and references (see for example~\cite{R84} \S~3.6 p.153 and
Theorem~5 p.158).  However it is quicker to prove the facts we need than recall
these general results.

We denote by $\cer$ the circle and we let $e \cl \R \to \cer \simeq \R/2\pi\Z$
be the quotient map.  We use the coordinate $\theta$ on $\cer$ defined up to a
multiple of $2\pi$.  We also denote by $T \cl \R \to \R$ the translation $T(x) =
x+2\pi$.  We recall that $\cor$ is a field.

Since $e$ is a covering map, we have an isomorphism of functors $\epb{e} \simeq
\opb{e}$, hence an adjunction $(\eim{e},\opb{e})$.  For any $n\in \Z$ we have
$e\circ T^n = e$, hence natural isomorphisms of functors $\opb{(T^n)} \opb{e}
\simeq \opb{e}$ and $\eim{e} \eim{(T^n)} \simeq \eim{e}$.  For $G \in
\Mod(\cor_\R)$ the isomorphism $\eim{e}( \eim{(T^n)} (G)) \isoto \eim{e}(G)$
gives by adjunction $i_n(G) \cl \eim{(T^n)} (G) \to \opb{e} \eim{e}(G)$.  For
$x\in \R$ we have $( \opb{e} \eim{e}(G))_x \simeq (\eim{e}(G))_x \simeq
\bigoplus_{n\in \Z} G_{T^n(x)}$. We deduce that the sum of the $i_n(G)$ gives an
isomorphism
\begin{equation}
\label{eq:iminvimdire}
\bigoplus_{n\in \Z} \oim{(T^n)}(G) \isoto \opb{e} \eim{e}(G) .
\end{equation}

Let $I$ be a bounded interval of $\R$.  We have $\eim{e}(\cor_I) \isoto
\oim{e}(\cor_I)$.  Let $A_I$ be the algebra $A_I = \Hom(\oim{e}(\cor_I),
\oim{e}(\cor_I))$.  The adjunction $(\opb{e},\oim{e})$ gives a morphism $\opb{e}
\oim{e}(\cor_I) \to \cor_I$. Using also the adjunction $(\eim{e},\opb{e})$ we
obtain a natural morphism
\begin{equation}
\label{eq:defvarepsilon_I}
\varepsilon_I \cl A_I \simeq \Hom(\cor_I, \opb{e} \oim{e}(\cor_I))
\to  \Hom(\cor_I, \cor_I) \simeq \cor .
\end{equation}

\begin{lemma}
\label{lem:hom-imdir-interv}
Let $I$ be a bounded interval of $\R$ and let $A_I$ be the algebra $A_I =
\Hom(\oim{e}(\cor_I), \oim{e}(\cor_I))$.  Then the morphism $\varepsilon_I$
defined in~\eqref{eq:defvarepsilon_I} is an algebra morphism and
$\ker(\varepsilon_I)$ is a nilpotent ideal of $A$.

More precisely, if $I$ is closed or open, then $\varepsilon_i$ is an
isomorphism.  If $I$ is half-closed, say $I = [a,x[$ or $I = \mo]x,a]$ and we
set $E_a = I \cap \opb{e}(e(a))$, then the identification
$(\oim{e}(\cor_I))_{e(a)} \simeq \cor^{E(a)}$ induces a morphism
$$
A_I \to \Hom(\cor^{E(a)}, \cor^{E(a)}),
\qquad \varphi \mapsto \varphi_{e(a)}
$$
which identifies $A_I$ with the subalgebra of matrices generated by the standard
nilpotent matrix of order $|E(a)|$; the eigenvalue of $\varphi_{e(a)}$ is then
$\varepsilon_I(\varphi)$.  Moreover,
\begin{itemize}
\item [(a)] a morphism $u \in A_I$ is an isomorphism if and only if
  $\varepsilon_I(u) \not= 0$,
\item [(b)] for any open interval $J \subset \cer$ and any $\varphi \in A_I$,
  the morphism $\sect(J; \varphi) \cl \sect(J;\oim{e}(\cor_I)) \to
  \sect(J;\oim{e}(\cor_I))$ has only one eigenvalue which is
  $\varepsilon_I(\varphi)$,
\item [(c)] for any $p \in \dot\SSi(\oim{e}(\cor_I))$ and any $\varphi \in A_I$,
  we have $\varphi^\mu_p = \varepsilon_I(\varphi)$.
\end{itemize}
\end{lemma}
\begin{proof}
  Using $\opb{e} \oim{e}(\cor_I) \simeq \bigoplus_{n\in \Z} \oim{(T^n)}(\cor_I)$
  and Lemma~\ref{lem:morph_deux_int0}, the cases $I$ closed or open are obvious.
  If $I$ is half-closed of length $l$, we find $A_I \simeq \cor^{|E(a)|}$ as a
  vector space.  Assuming $I = [a,x[$ (the case $]x,a]$ is similar), a basis of
  $A_I$ is given by the morphisms $e_*(u_n)$, for $n=0,\ldots, |E(a)|-1$, where
  $u_n \cl \cor_I \to \cor_{T^n(I)}$ is the natural morphism and we use the
  natural identification $\phi_n \cl e_*(\cor_{T^n(I)}) \simeq e_*(\cor_I)$.
  At the level of germs $\phi_n$ identifies the summands $(\cor_I)_{a+2\pi k}$
  of $(e_*(\cor_I))_{e(a)}$ and $(\cor_{T^n(I)})_{a+2\pi (k+n)}$ of
  $(e_*(\cor_{T^n(I)}))_{e(a)}$.  We obtain that $(e_*(u_n))$ acts on $
  \cor^{E(a)}$ by $(s_1,s_2,\ldots) \mapsto (s_{1+n},s_{2+n},\ldots)$. We deduce
  that the image of $A_I$ in $\End(\cor^{E(a)})$ is as claimed in the lemma.

  The assertions~(a)-(c) follow from the structure of $A_I$.  For~(b) and~(c)
  we remark that $A_I \to \End(\sect(J;\oim{e}(\cor_I)))$, $\varphi \mapsto
  \sect(J; \varphi)$, and $A_I \to \cor$, $\varphi \mapsto \varphi^\mu_p$, are
  algebra morphisms.
\end{proof}

\begin{lemma}
\label{lem:interv_facteur}
Let $F\in \Mod(\cor_\cer)$. Let $I$ be a bounded interval of $\R$ such that
$\cor_I$ is a direct summand of $\opb{e}F$. Then $\oim{e}\cor_I$ is a direct
summand of $F$.
\end{lemma}
\begin{proof}
  We let $i_0 \cl \cor_I \to \opb{e}F \simeq \epb{e}F$ and
  $p_0 \cl \opb{e}F \to \cor_I$ be morphisms such that
  $p_0 \circ i_0 = \id_{\cor_I}$ and we denote by
  $i_0' \cl \oim{e} \cor_I \simeq \eim{e} \cor_I \to F$ and
  $p_0' \cl F \to \oim{e} \cor_I$ their adjoint morphisms.  Let
  $v \cl \cor_I \to \opb{e} \oim{e}\cor_I$ be the adjoint of $p_0' \circ i_0'$.
  Then $v = \opb{e}(p_0') \circ i_0$.  Composing with the natural morphism
  $a\cl \opb{e} \oim{e}\cor_I \to \cor_I$ we find
  $\varepsilon_I(p_0' \circ i_0') = a \circ v = p_0 \circ i_0 = \id_{\cor_I}$.
  Hence $p_0' \circ i_0'$ is an isomorphism by Lemma~\ref{lem:hom-imdir-interv}
  and this gives the result.
\end{proof}

\begin{lemma}
\label{lem:interv_facteur2}
Let $F\in \Mod_\rc(\cor_\cer)$.  We assume that $F$ is not locally constant.
Then there exists a bounded interval $I$ of $\R$ such that $\cor_I$ is a direct
summand of $\opb{e}F$.
\end{lemma}
\begin{proof}
  (i) We set $N = \max_{\theta \in \cer}\{ \dim F_\theta\}$.  We choose
  $\theta_0 \in \cer$ such that $F$ is not constant around $\theta_0$ and $x_0
  \in \R$ such that $e(x_0) =\theta_0$.  We choose $r> 2 \pi N$ and we define $J
  = \mo]x_0 - r,x_0 + r[$. We 
set $G = \opb{e} F$ and we
apply Corollary~\ref{cor:cons_sh_R} to $G_J$,
  writing $G_J \simeq \bigoplus_{a\in A} \cor^{n_a}_{I_a}$ as
  in~\eqref{eq:cons_sh_R1} for some finite family of intervals $I_a$ and
  integers $n_a \in \N$, $a\in A$.  One of the intervals $I_a$ has an end at
  $x_0$, say the left end (the other case being similar): there exists $a_0 \in
  A$ and $y_0 \in [x_0,x_0+r]$ such that $\ol{I_{a_0}} = [x_0,y_0]$.

  \sui (ii) Let us prove that $y_0 < x_0+r$.  Since $\opb{T} G\isoto G$, the
  sheaf $\cor_{T^n(I_{a_0})}$ is a direct summand of $G_{T^n(J)}$, for any $n\in
  \Z$.  We set $I'_n = T^n(I_{a_0}) \cap [x_0,x_0 + r[$ and $J'_n = T^n(J) \cap
  [x_0,x_0 + r[$.  We obtain that $\cor_{I'_n}$ is a direct summand of
  $G_{J'_n}$ for all $n \in \Z$.  For $n=0,\ldots,N$ we have $J'_n = [x_0,x_0 +
  r[$. Hence $\cor_{I'_n}$ is a direct summand of $G_{[x_0,x_0 + r[}$ for
  $n=0,\ldots,N$.  If $y_0 = x_0+r$, then the intervals $I'_n$, for
  $n=0,\ldots,N$, contain $z_0 = x_0+ (2\pi N +r)/2$.  It follows that $\dim
  F_{e(z_0)} = \dim G_{z_0} \geq N+1$ which contradicts the definition of $N$.
  Hence $y_0 < x_0 +r$.

  \sui (iii) By~(ii) we have $\ol{I_{a_0}} \subset J$.  Let $i \cl
  \cor_{I_{a_0}} \to G_J$ and $p\cl G_J \to \cor_{I_{a_0}}$ be given by the
  decomposition of $G_J$.  Then $p$ factorizes through $G_{\ol{I_{a_0}}}$ and
  extends to $G$ through the natural morphism $G \to G_{\ol{I_{a_0}}}$.  Since
  $i$ extends to $G$ through $G_J \to G$, we see that $\cor_{I_{a_0}}$ is a
  direct summand of $G$.
\end{proof}

\begin{proposition}
\label{prop:cons_sh_cercle}
Let $F \in \Mod_\rc(\cor_\cer)$.  Then there exist a finite family
$\{(I_a,n_a)\}_{a \in A}$, of bounded intervals and integers, and a locally
constant sheaf of finite rank $L\in \Mod(\cor_\cer)$ such that
\begin{equation}
\label{eq:cons_sh_cer1}
F \simeq L \oplus \bigoplus_{a\in A} \oim{e}(\cor^{n_a}_{I_a}) .
\end{equation}
\end{proposition}
\begin{proof}
  (i) Let $\theta \in \cer$ be a point around which $F$ is not constant. We can
  find a small arc $I$ around $\theta$ such that $F$ is constant on the two
  components of $I \setminus \{\theta\}$, of rank say $n_\theta^1$ and
  $n_\theta^2$.  We set $r_\theta(F) = n_\theta^1 + n_\theta^2 +
  \dim(F_\theta)$.  We define $r(F) = \sum_\theta r_\theta(F)$, where $\theta$
  runs over the points around which $F$ is not constant.  If $F$ is locally
  constant we set $r(F) = 0$.  We prove the proposition by induction on $r(F)$.

  \sui (ii) We assume $r(F)\not=0$.  Then $F$ is not locally constant. By
  Lemmas~\ref{lem:interv_facteur2} and~\ref{lem:interv_facteur} there exists a
  bounded interval $I$ of $\R$ such that $F \simeq \oim{e}(\cor_I) \oplus F'$
  for some $F' \in \Mod_\rc(\cor_\cer)$.  Then $r(F') < r(F)$ and the induction
  proceeds.
\end{proof}

\section{Locally constant sheaves on the circle}
\label{sec:lccstshcercle}

From now on we assume that $\cor = \corC$.  We assume that the
circle $\cer$ is oriented. This gives an isomorphism $\pi_1(\cer) \simeq \Z$
independent of a choice of base point in $\cer$.

A locally constant sheaf $L$ (also called local system) on a connected manifold
$M$ is determined by its monodromy, which is a representation of $\pi_1(M,x)$ in
$\Aut(L_x)$ for a given base point $x$.  We have in fact an equivalence of
categories between local systems on $M$ and representations of $\pi_1(M,x)$.
When $M=\cer$ the monodromy is determined by the image of $1\in \pi_1(\cer)
\simeq \Z$. Hence we obtain an equivalence between the category of local systems
over $\corC$ on $\cer$ and the category of pairs $(V,A)$, where $V$ is a
$\corC$-vector space and $A \in \GL(V)$ (a morphism from $(V,A)$ to $(V',A')$ is
a morphism $V\to V'$ commuting with $A,A'$).  In particular local systems of
finite rank are classified up to isomorphism by matrices in Jordan form.

We write $\N^* = \N \setminus \{0\}$.  Let $\alpha \in \corC^\times$ and $r\in
\N^*$.  We denote by $N_r \in \Mat(r,r;\corC)$ the standard nilpotent matrix of
order $r$ and we set $A_{\alpha,r} = \alpha I_r + N_r$.  For $\alpha=1$ we write
for short $A_r = A_{1,r}$.  We choose $1\in \cer$ as base point and we let
$L_{\alpha,r} \in \Mod(\corC_{\cer})$ be the irreducible local system with
monodromy $A_{\alpha,r}$.  We also write $L_r = L_{1,r}$ and $L_{\alpha,0} = 0$
for all $\alpha$.  Since any matrix of finite rank has a Jordan form, we obtain
the following decomposition. Let $L\in \Mod(\corC_\cer)$ be a local system of
finite rank. Then there exists a finite subset $B \subset \corC^\times \times
\N^*$ and integers $n_{\alpha,r} \in \N^*$, for $(\alpha,r) \in B$, such that
\begin{equation}
\label{eq:dec_loca_syst}
L \simeq \bigoplus_{  (\alpha,r) \in B} L_{\alpha,r}^{n_{\alpha,r}} .
\end{equation}
The equivalence between monodromy matrices and local systems is compatible with
duality. In particular, for any $(\alpha,r) \in \corC^\times \times \N^*$ there
exists an isomorphism $\DD'(L_{\alpha,r}) \simeq L_{1/\alpha,r}$.

In this section we define a group $H^i_{\alpha,r}(F)$ for all $F\in
\Derb(\corC_\cer)$ such that, if $F$ is constructible, the dimension of
$H^i_{\alpha,r}(F)$ is the multiplicity of $L_{\alpha,r}$ in the locally
constant component of $H^i(F)$ given by Proposition~\ref{prop:cons_sh_cercle}.
We have chosen the definition of $H^i_{\alpha,r}(F)$ so that we can extend it
easily to a relative situation in the next section.

In Proposition~\ref{prop:morph_elem0} we explain how these multiplicities are
modified when we quotient a locally constant sheaf on $\cer$ by a constant sheaf
on an arc. This will be used to study sheaves on the cylinder near a circle
which is tangent to the projection of the microsupport.

\medskip

For $(\alpha,r) \in \corC^\times \times \N^*$ the matrix $A_{\alpha,r}$ has an
eigenspace of dimension $1$.  The inclusion of this eigenspace gives the
morphism $a_r$ in the exact sequence of local systems: $0 \to L_{\alpha,1}
\to[a_r] L_{\alpha,r} \to[b_r] L_{\alpha,r-1} \to 0$.  We write the induced
distinguished triangle as
\begin{equation}
\label{eq:dtnilp}
L_{\alpha,1} \to[a_r] L_{\alpha,r} \to[b_r] L_{\alpha,r-1} \to[c_r] L_{\alpha,1}[1].
\end{equation}
Taking the cohomology gives the exact sequence
\begin{equation}
\label{eq:exseqnilp}
\begin{split}
& 0 \to H^0(\cer; L_{\alpha,1}) \to[a_r^0] H^0(\cer; L_{\alpha,r}) 
 \to[b_r^0] H^0(\cer; L_{\alpha,r-1})  \\
& \hspace{1cm}  \to[c_r^0] H^1(\cer; L_{\alpha,1}) 
\to[a_r^1] H^1(\cer; L_{\alpha,r}) \to[b_r^1] H^1(\cer; L_{\alpha,r-1}) \to 0 .
\end{split}
\end{equation}

\begin{lemma}
\label{lem:cohom_Lalphar}
For any $(\alpha,r) \in \corC^\times \times \N^*$ with $\alpha\not= 1$ and any
$i\in \Z$, we have $H^i(\cer; L_{\alpha,r}) \simeq 0$.

For $\alpha=1$ and for any $r\in \N^*$ we have $H^0(\cer; L_{r}) \simeq \corC$
and $H^1(\cer; L_{r}) \simeq \corC$. Moreover, if $r\geq 2$, the morphisms
$a_r^0$, $c_r^0$, $b_r^1$ in~\eqref{eq:exseqnilp} are isomorphisms and $b_r^0 =
a_r^1 =0$.
\end{lemma}
\begin{proof}
  (i) Through the equivalence between local systems and monodromy matrices,
  taking global sections corresponds to finding the invariant vectors, that is,
  the eigenspace with eigenvalue $1$.  We deduce the assertions on $H^0(\cer;
  L_{\alpha,r})$ for all $\alpha$. Since $\DD'(L_{\alpha,r}) \simeq
  L_{1/\alpha,r}$ the assertions on $H^1(\cer; L_{\alpha,r}) \simeq (H^0(\cer;
  L_{1/\alpha,r}))^*$ follow.

  \sui (ii) For $\alpha=1$ and $1 \leq r' \leq r$ we have only one injective
  morphism $(\corC^{r'},A_{r'}) \to (\corC^r,A_{r})$ up to a non-zero
  multiple. It sends the $1$-eigenspace to the $1$-eigenspace.  This means that
  $H^0(\cer; L_{r'}) \to H^0(\cer; L_{r})$ is an isomorphism.  In particular
  $a_r^0$ is an isomorphism and, by duality, so is $b_r^1$.  It follows that
  $b_r^0 = a_r^1 =0$ and then that $c_r^0$ is an isomorphism.
\end{proof}

Let $F\in \Derb(\corC_{\cer})$.  The tensor product of $F$ with~\eqref{eq:dtnilp}
for $\alpha=1$ and $r=2$ gives the distinguished triangle
\begin{equation}
\label{eq:dtnilpF}
F \to L_2 \otimes F \to F \to[c(F)] F[1] ,
\end{equation}
where we set $c(F) = c_2 \otimes \id_F \cl F \to F[1]$. Taking the global
sections we obtain the morphisms
\begin{equation}
\label{eq:multFc2}
H^i(\cer; c(F)) \cl H^i(\cer;F) \to H^{i+1}(\cer;F) .
\end{equation}
We remark that $c(\cdot)$ is a morphism of functors from $\id$ to $[1]$, that
is, we have the commutative square, for any morphism $u \cl F\to G$ in
$\Derb(\corC_{\cer})$,
\begin{equation}
\label{eq:multc2_func}
\vcenter{\xymatrix@C=2cm{
F \ar[r]^{u} \ar[d]_{c(F)} &  G \ar[d]^{c(G)}  \\
F[1] \ar[r]^{u[1]} &  G[1] 
}}
\end{equation}
which is a particular case of $(a\otimes \id) \circ (\id \otimes u) = a\otimes u
= (\id \otimes u) \circ (a\otimes \id)$ for any other morphism $a$ in
$\Derb(\corC_{\cer})$.

\begin{lemma}
\label{lem:constcomp}
Let $F\in \Mod(\corC_\cer)$ and $i\in \Z$. We assume $F = L_{\alpha,r}$ for some
$(\alpha,r) \in \corC^\times\times \N^*$ or $F = \oim{e}(\corC_I)$ where $I$ is a
bounded interval of $\R$.  Then the morphism $H^i(\cer; c(F))$
in~\eqref{eq:multFc2} vanishes except in the case $F = L_{1,1} = \corC_\cer$ and
$i=0$.  In this case
$$
H^0(\cer; c(\corC_\cer)) \cl H^0(\cer; \corC_\cer) \isoto H^1(\cer; \corC_\cer)
$$
is an isomorphism.
\end{lemma}
\begin{proof}
  When $F = \oim{e}( \corC_{I})$ or $F = L_{\alpha,r}$ with $\alpha\not= 1$ we
  have $H^0(\cer;F) \simeq 0$ or $H^1(\cer;F) \simeq 0$. Hence $H^i(\cer; c(F))$
  is trivially $0$ for all $i$.

When $F = \corC_\cer$ the result follows from Lemma~\ref{lem:cohom_Lalphar}
since $c(\corC_\cer)$ is the morphism $c_2$ of~\eqref{eq:exseqnilp}.

It remains to consider $F = L_{1,r}$ with $r\geq 2$.  We take the cohomology of
the square~\eqref{eq:multc2_func} with $u=a_r \cl \corC_\cer \to L_r$.  We remark
that $H^0(\cer; c(\corC_\cer)) = c_2^0$ in the notations of
Lemma~\ref{lem:cohom_Lalphar} and we obtain the commutative square
$$
\xymatrix@C=2cm{
H^0(\cer; \corC_\cer) \ar[r]^{a_r^0} \ar[d]_{c_2^0}
 &  H^0(\cer; L_{r}) \ar[d]^{H^0(\cer; c(L_{r}))}  \\
H^1(\cer; \corC_\cer) \ar[r]^{a_r^1}  &  H^1(\cer; L_{r}) . }
$$
By Lemma~\ref{lem:cohom_Lalphar} the morphisms $a_r^0$ and $c_2^0$ are
isomorphisms and $a_r^1 = 0$.  Hence $H^0(\cer; c(L_{r})) = 0$.  Since
$H^i(\cer; c(L_{r}))$ vanishes trivially for $i\not=0$, the lemma is proved.
\end{proof}

\begin{lemma}
\label{lem:decprodlocsyst}
Let $p\leq q \in \N^*$ and $\alpha,\beta \in \corC^\times$ be given.
We have a decomposition
\begin{equation}
\label{eq:decprodnilp}
L_{\alpha,p} \otimes L_{\beta,q} \simeq \bigoplus_{i = 1}^p L_{\alpha\beta,q-p+2i-1} .
\end{equation}
In particular, if $p=q$ and $\alpha = 1/\beta$, then
$L_{\alpha,p} \otimes L_{\beta,q}$ has a direct summand $\corC_\cer$ with
multiplicity $1$.  In the other cases $L_{\alpha,p} \otimes L_{\beta,q}$ does
not have $\corC_\cer$ as direct summand.
\end{lemma}
\begin{proof}
  By the correspondence between local systems on $\cer$ and monodromy matrices,
  the result follows from a similar statement about the tensor product of two
  unipotent matrices.
\end{proof}

\begin{definition}
\label{def:alpharcomp}
Let $(\alpha,r) \in \corC^\times\times\N^*$ be given.  For $F\in
\Derb(\corC_\cer)$ we define 
\begin{equation}
\label{eq:defalpahrcom}
H^i_{\alpha,r}(F) = \im(H^i(\cer; c(F \otimes L_{1/\alpha,r}))) ,
\end{equation}
which is a subspace of $H^{i+1}(\cer; F \otimes L_{1/\alpha,r})$.
Using~\eqref{eq:multc2_func} we see that $F \mapsto H^i_{\alpha,r}(F)$ is a
functor.
\end{definition}

Then Lemmas~\ref{lem:constcomp} and~\ref{lem:decprodlocsyst} translate as
follows.

\begin{proposition}
\label{prop:loc_cst_comp}
  Let $F \in \Mod(\corC_\cer)$ be a constructible sheaf.  We write $F = G \oplus
  L$, where $G$ is a sum of sheaves of the type $\oim{e}(\corC_I)$ for some
  bounded interval $I$ of $\R$, and $L$ is a local system, decomposed as
  in~\eqref{eq:dec_loca_syst} $L \simeq \bigoplus_{ (\alpha,r) \in B}
  L_{\alpha,r}^{n_{\alpha,r}}$, for some finite subset $B \subset \corC^\times
  \times \N^*$ and integers $n_{\alpha,r}$. Then, for all $(\alpha,r) \in B$,
  we have $H^0_{\alpha,r}(F) \simeq \corC^{n_{\alpha,r}}$.  We also have
  $H^i_{\alpha,r}(F) \simeq 0$ if $i \not= 0$ or $(\alpha,r) \not\in B$.
\end{proposition}

\begin{proposition}
\label{prop:morph_elem0}
Let $L \in \Mod(\corC_\cer)$ be a local system of finite rank. Let $I\subset
\cer$ be an open arc and let $u \cl \corC_I \to L$ be a non-zero morphism.  We
set $F = \coker(u)$.  We assume that $\dim(H^0(\cer; F)) > \dim(H^0(\cer; L))$.
Then there exists $(\alpha,r) \in \corC^\times \times \N^*$ such that $(\alpha,r)
\not= (1,1)$ and $\dim(H^0_{\alpha,r}(F)) < \dim(H^0_{\alpha,r}(L))$.
\end{proposition}
\begin{proof}
(i)  We decompose $L$ as in~\eqref{eq:dec_loca_syst} $L \simeq \bigoplus_{
    (\alpha,r) \in B} L_{\alpha,r}^{n_{\alpha,r}}$, for some finite subset $B
  \subset \corC^\times \times \N^*$ and integers $n_{\alpha,r}$.  By
  Proposition~\ref{prop:loc_cst_comp} we have $n_{\alpha,r} =
  \dim(H^0_{\alpha,r}(L))$.

  We can also write $F = G \oplus L'$, where $G$ is a sum of sheaves of the type
  $\oim{e}(\corC_J)$ for some bounded interval $J$ of $\R$, and $L'$ is a local
  system.  We decompose $L'$ as $L \simeq \bigoplus_{ (\alpha,r) \in B'}
  L_{\alpha,r}^{n'_{\alpha,r}}$, for some finite subset $B' \subset \corC^\times
  \times \N^*$ and integers $n'_{\alpha,r}$.  We also have $n'_{\alpha,r} =
  \dim(H^0_{\alpha,r}(F))$.

  \sui (ii) We argue by contradiction and assume that $n'_{\alpha,r} \geq
  n_{\alpha,r}$ for all $(\alpha,r) \not= (1,1)$.  Let $u_0 \cl \corC_I \to
  L_{1,1}^{n_{1,1}} \simeq \corC_\cer^{n_{1,1}}$ be the morphism obtained by
  composing $u$ with the projection $L \to L_{1,1}^{n_{1,1}}$ given by the
  decomposition of $L$. Let us prove that $u_0 \not= 0$.  If $u_0 = 0$, then
  $L_{1,1}^{n_{1,1}}$ is a direct summand of $F$ and we have $n'_{1,1} \geq
  n_{1,1}$. Hence $n'_{\alpha,r} \geq n_{\alpha,r}$ for all $(\alpha,r) \in B$.
  This implies $\dim(L'_{x}) \geq \dim(L_{x})$ for all $x\in \cer$.  However, if
  $x\in I$, we have $\dim(F_{x}) = \dim(L_{x}) - 1$.  In particular
  $\dim(L'_{x}) < \dim(L_{x})$ and we obtain a contradiction.  Hence $u_0 \not=
  0$.

  \sui (iii) Since $u_0 \not= 0$ we can find $v \cl \corC_\cer^{n_{1,1}} \to
  \corC_\cer$ such that $u_1 = v \circ u_0 \cl \corC_I \to \corC_\cer$ is non
  zero.  Then $H^1(\cer;u_1) \cl H^1(\cer;\corC_I) \to H^1(\cer;\corC_\cer)$ is
  an isomorphism.  Since $u_1$ factorizes through $u$, it follows that the
  morphism $H^1(\cer;u)$ is injective. Using the cohomology sequence associated
  with $\corC_I \to[u] L \to F$,
$$
0 \to H^0(\cer;L) \to H^0(\cer;F) \to H^1(\cer; \corC_I)
 \to[H^1(\cer;u)]  H^1(\cer; L) ,
$$
we deduce that $\dim(H^0(\cer;L)) = \dim (H^0(\cer;F))$.  This contradicts the
hypothesis and proves the proposition.
\end{proof}

\section{Sheaves on the cylinder}
\label{sec:shcyl}

In this section we study sheaves on $M = \cer \times J$, where $J$ is an
interval, with a Lagrangian microsupport in general position.  Let $\Lambda
\subset \dT^*M$ be a closed conic Lagrangian submanifold (maybe non
connected). We set $C_\Lambda = \dot\pi_M(\Lambda)$.  In the generic case
$C_\Lambda$ is a curve with only cusps and double points as singularities.  Let
$t_0 \in J$ be given.  If $\cer \times \{t_0\}$ intersects $C_\Lambda$ only at
its smooth points and transversally, then we can find a diffeomorphism around
$\cer \times \{t_0\}$ which sends $C_\Lambda$ to a product $C \times J'$ for a
finite subset $C$ of $\cer$ and a small interval $J'$ around $t_0$.  By
Corollary~\ref{cor:opbeqv} the sheaves $F$ with $\dot\SSi(F) \subset \Lambda$
are of the type $\opb{p}(G)$ for $G \in \Derb(\corC_\cer)$, where $p \cl \cer
\times J' \to \cer$ is the projection.

Here we study a less generic situation, namely the case where $\cer \times
\{t_0\}$ intersects $C_\Lambda$ transversally except at one point.  We only
study the case we will need: $F$ is simple and its direct image to $J$
decomposes in some prescribed way. In this situation we show that we can find
$F$-linked points in disjoint components of $\Lambda \cap T^*(\cer \times J')$
for a neighborhood $J'$ of $t_0$.

\subsection{Locally constant components}

We choose an orientation of the circle $\cer$.  Let $J$ be an open interval of
$\R$.  Let $p \cl \cer \times J \to \cer$ and $q \cl \cer \times J \to J$ be the
projections.

The morphism $c_2 \cl \corC_\cer \to \corC_\cer[1]$ of~\eqref{eq:dtnilp} (for
$\alpha=1$ and $r=2$) pulls back to $\cer\times J$ and we extend
Definition~\ref{def:alpharcomp} to the relative situation as follows.

\begin{definition}
\label{def:alpharcomp_cyl}
We denote by $c_2 \cl \corC_{\cer\times J} \to \corC_{\cer\times J}[1]$ the
morphism corresponding to $1\in \corC$ through the isomorphisms
$$
\Hom(\corC_{\cer\times J}, \corC_{\cer\times J}[1])
\simeq  H^1(\cer\times J; \corC_{\cer\times J}) \simeq \corC .
$$
For $F\in \Derb(\corC_{\cer\times J})$ we define $c(F) = c_2 \otimes \id_F \cl F
\to F[1]$ and, for given $(\alpha,r) \in \corC^\times\times\N^*$, we obtain the
following morphism in $\Derb(\corC_J)$
$$
\roim{q}( c(F \otimes \opb{p}L_{1/\alpha,r})) \cl
\roim{q}(F \otimes \opb{p}L_{1/\alpha,r})
\to \roim{q}(F \otimes \opb{p}L_{1/\alpha,r})[1] .
$$
Taking its cohomology in degree $i\in \Z$ we define
\begin{equation}
\label{eq:defalpahrcom_cyl}
R^iq_{\alpha,r}(F) = \im(H^i\roim{q}( c(F \otimes \opb{p}L_{1/\alpha,r}))) 
\; \in \Mod(\corC_J) .
\end{equation}
In particular $R^iq_{\alpha,r}(F)$ is a subsheaf of $H^{i+1}\roim{q}(F \otimes
\opb{p}L_{1/\alpha,r})$.  As for $H^i_{\alpha,r}(\cdot)$ we see by
using~\eqref{eq:multc2_func} that $F \mapsto R^iq_{\alpha,r}(F)$ is a functor
from $\Derb(\corC_{\cer\times J})$ to $\Mod(\corC_J)$.
\end{definition}

\begin{lemma}
\label{lem:germ_alpharcomp_cyl}
Let $F\in \Derb(\corC_{\cer\times J})$.  For any $t\in J$ and $(\alpha,r) \in
\corC^\times\times\N^*$ we have a canonical isomorphism $(R^iq_{\alpha,r}(F))_t
\simeq H^i_{\alpha,r}(F|_{\opb{q}(t)})$.
\end{lemma}
\begin{proof}
  We set $F' = F \otimes \opb{p}L_{1/\alpha,r}$, $G = \roim{q}(F')$ and $d =
  \roim{q}( c(F'))$. Hence $d$ is a morphism from $G$ to $G[1]$.

  Let $a\cl A \to B$ be a morphism in $\Mod(\corC_J)$.  Taking the germ is an
  exact functor and we have a canonical isomorphism $(\im(a))_t \simeq
  \im(a_t)$.  We apply this to $A = H^iG$, $B = H^{i+1}G$ and $a = H^i(d)$.
Since taking the germs commutes with the cohomology, we obtain
$$
(R^iq_{\alpha,r}(F))_t \simeq
 \im \bigl(H^i(d_t) \cl H^i(G_t) \to H^{i+1}(G_t) \bigr).
$$
By the base change formula we have $G_t \simeq \rsect(\cer; F|_{\opb{q}(t)}
\otimes L_{1/\alpha,r})$ and $d_t = c(F|_{\opb{q}(t)} \otimes
L_{1/\alpha,r})$. The lemma follows.
\end{proof}

By the triangular inequality we have $\SSi(F) \subset \bigcup_{i\in \Z}
\SSi(H^i(F))$ for any sheaf $F$ on a manifold $N$.  In general this inclusion is
strict. If $\dim N = 1$ and the coefficient ring is a field, this is an equality
since $F \simeq \bigoplus_{i\in \Z} H^i(F)[-i]$ by
Lemma~\ref{lem:faiscdim1_compl_scinde}.  We deduce the following bound.  Let
$\cor$ be a field and let $F,G \in \Mod(\cor_\R)$ be constructible sheaves.  Let
$u \cl F \to G$ be a morphism in $\Mod(\cor_\R)$.  Then
\begin{equation}
\label{eq:faiscdim1_SSkercoker}
(\SSi(\ker(u)) \cup \SSi(\coker(u)) )
\subset ( \SSi(F) \cup \SSi(G) ) .
\end{equation}
Indeed, if we see $u$ as a morphism in $\Derb(\cor_\R)$ and denote by $F'$ the
cone of $u$, then $\ker(u) \simeq H^0(F')$ and $\coker(u) \simeq H^1(F')$.  Then
the bound follows from the triangular inequality and $\SSi(F') = \SSi(H^0(F'))
\cup \SSi(H^1(F'))$.

We apply~\eqref{eq:faiscdim1_SSkercoker} to the morphism $H^i(\roim{q}( c(F
\otimes \opb{p}L_{1/\alpha,r})))$ of Definition~\ref{def:alpharcomp_cyl}.  Since
$\opb{p}L_{1/\alpha,r}$ is locally constant, Theorem~\ref{th:opboim} and
Corollary~\ref{cor:opboim} give
\begin{equation}
\label{eq:SSRiqalphar}
\SSi(R^iq_{\alpha,r}(F) ) \subset  q_\pi\opb{q_d}\SSi(F)  .
\end{equation}

\subsection{Generic tangent point}
For the remaining part of Section~\ref{sec:shcyl} we consider a smooth conic
Lagrangian submanifold $\Lambda$ of $\dT^*M$ such that $\dot\pi_M(\Lambda)$ is
smooth and tangent to one fiber of $q$. We study the $F \in \Derb(\cor_M)$ which
are simple along $\Lambda$.

More specifically our interval $J$ is $J = \mo]-1,1[$.  We take the coordinate
$\theta \in \mo]-\pi,\pi]$ on $\cer$ and $t$ on $J$.  Let $\Omega \subset \cer
\times J$ be the open subset $\Omega = \{(\theta,t)$; $-1 < \theta < 1$, $t>
\theta^2 \}$.  We set $\Lambda_\Omega = \dot\SSi(\corC_\Omega)$; this is the
half-part of $T^*_{\partial \Omega}M$ which contains $(0,0;0,-1)$ (see
Example~\ref{ex:microsupport}-(iii)).  We let $\Lambda_0 \subset \dT^*\cer$ be a
conic Lagrangian submanifold, that is, a finite union of half-lines
$\{\theta_i\} \times \Rp$ or $\{\theta_i\} \times (-\Rp)$.  We set $C =
\dot\pi_\cer(\Lambda_0)$ and we assume $C \cap \ol{p(\Omega)} = \emptyset$, that
is, $C \subset (\cer \setminus [-1,1])$.  We define
\begin{equation}
\label{eq:defLambda_cyl}
\Lambda = (\Lambda_0 \times T^*_JJ) \sqcup \Lambda_\Omega .
\end{equation}
As remarked at the beginning of Section~\ref{sec:shcyl}, for a given generic
conic Lagrangian submanifold $\Lambda$ such that $\dot\pi_M(\Lambda)$ is
tangent to $\opb{q}(0)$, we can find coordinates such that $\Lambda$ is
described by~\eqref{eq:defLambda_cyl}, up to the choice of a side of
$T^*_{\partial \Omega}M$.  So~\eqref{eq:defLambda_cyl} only gives ``half'' of
the generic case of a tangent point.  The other half is given by $\Lambda^a$.
We only study the case~\eqref{eq:defLambda_cyl}. The results for $\Lambda^a$
are similar; see Remark~\ref{rem:dual_Lambda_cyl} below.

\begin{proposition}
\label{prop:description_F_simpleLambda_cyl}
Let $F \in \Derb(\cor_M)$ be such that $\dot\SSi(F) = \Lambda$ and $F$ is simple
along $\Lambda$. Then there exist $G \in \Derb(\cor_\cer)$ which is simple along
$\Lambda_0$ and $i\in \Z$ such that we have a distinguished triangle
\begin{equation}
\label{eq:descrFcyl1}
\corC_\Omega[i] \to[u] \opb{p}(G) \to F \to[+1].
\end{equation}
In particular we have
\begin{itemize}
\item [(a)] if $u = 0$, then $F \simeq \opb{p}G \oplus \corC_\Omega[i+1]$,
\item [(b)] if $u\not= 0$, there exist $G' \in \Mod(\cor_\cer)$ and $G'' \in
  \Derb(\cor_\cer)$, which are simple along their microsupports $\dot\SSi(G')$,
  $\dot\SSi(G'') \subset \Lambda_0$, and a morphism $u \cl \corC_\Omega \to
  \opb{p} G'$ such that $F \simeq \opb{p}G'' \oplus \coker(u)[i]$.
\end{itemize}
\end{proposition}
\begin{proof}
  (i) We set $W = \mo]-1,1\mc[ \times J$.  The inclusion of $\Omega$ in $W$ is
  diffeomorphic to the inclusion of $\mo]-1,1\mc[ \times \mo]0,1[$ in
  $(\mo]-1,1[)^2$.  Such a diffeomorphism sends $\Lambda_0$ to a product
  $(\mo]-1,1\mc[ \times \{0\}) \times (\{0\} \times \Rp)$ and, by
  Corollary~\ref{cor:opbeqv}, we are reduced to the dimension $1$.  By
  Corollary~\ref{cor:cons_sh_R} we obtain the following possibilities for
  $F|_W$. There exists $L \in \Derb(\cor)$ and $i\in \Z$ such that
\begin{align}
\label{eq:descrF0}
F|_W &\simeq  L_W \oplus  \corC_\Omega[i+1] \\
\label{eq:descrF1}
\text{or} \qquad
F|_W &\simeq L_W \oplus \corC_{W \setminus \Omega}[i] .
\end{align}
We have a morphism $v \cl \corC_{W \setminus \Omega} [i] \to \corC_\Omega[i+1]$
given by the excision triangle for the inclusion $\Omega \subset W$.  The cone
of $v$ is $\corC_W[i+1]$.  Using the decomposition~\eqref{eq:descrF0} or the
decomposition~\eqref{eq:descrF1} together with $v$ we obtain in both cases a
morphism $p \cl F|_W \to \corC_\Omega[i+1]$ such that the cone of $p$ is a
constant sheaf on $W$.

\sui (ii) Tensoring $p$ with $\corC_{\ol{\Omega}}$ and composing with $F \to
F_{\ol{\Omega}}$ we obtain $p' \cl F \to \corC_\Omega[i+1]$.  We define $F_1$ by
the triangle $F_1 \to F \to[p'] \corC_\Omega[i+1] \to[+1]$.  Then $F_1|_W$ is
constant and we deduce that $\SSi(F_1)$ is contained in $(\Lambda_0 \times
T^*_JJ)$. Hence $F_1 \simeq \opb{p}(G)$ for some $G \in \Derb(\cor_\cer)$.  We
deduce the triangle~\eqref{eq:descrFcyl1}.

\sui (iii) We set $G' = H^{-i}(G)$ and $G'' \simeq \bigoplus_{k\in \Z\setminus
  \{-i\}} H^k(G)[-k]$.  We have $G \simeq G'[i] \oplus G''$ by
Lemma~\ref{lem:faiscdim1_compl_scinde}.  Since $\Hom(\corC_\Omega, \corC_W[k])
\simeq H^k(\Omega; \corC_W)$ vanishes for $k\not=0$, we see that $u$ is the
composition of $H^{-i}(u) \cl \corC_\Omega \to \opb{p}(G')$ and $\opb{p}(G'[i])
\to G$.  If $u\not= 0$, we deduce the case~(b) of the proposition.  The case~(a)
is obvious.
\end{proof}

\begin{remark}
\label{rem:dual_Lambda_cyl}
Let $F \in \Derb(\cor_M)$ be such that $\dot\SSi(F) = \Lambda^a$ and $F$ is
simple along $\Lambda^a$.  Then there exist $G \in \Derb(\cor_\cer)$ and $i\in
\Z$ such that we have a distinguished triangle
\begin{equation}
\label{eq:dual_descrFcyl1}
F \to  \opb{p}(G) \to[v] \corC_{\ol{\Omega}}[i] \to[+1].
\end{equation}
The proof is similar to the proof of
Proposition~\ref{prop:description_F_simpleLambda_cyl}.  If $F$ is
cohomologically constructible, this result is actually a corollary of the
proposition. Indeed we can apply the proposition to $\DD F$, since
$\SSi(\DD F) = \SSi(F)^a$, and use $\DD(\DD F) \simeq F$.
\end{remark}

We will later look for $F$-linked points, which means understanding $u^\mu_p$
for an automorphism $u$ of $F$.  The following lemma is useful to decompose
$u$.  This is a standard result about triangulated categories (see for example
the proof of Lemma~10.1.5 in~\cite{KS90}).

\begin{lemma}
\label{lem:morph_triang}
Let $\catc$ be an abelian category and let $A \to[f] B \to[g] C \to[h] A[1]$ be
a distinguished triangle in $\Derb(\catc)$.  We assume that $\Hom(A,C)$ and
$\Hom(A,C[-1])$ vanish.  Then, for any $u\cl B \to B$ there exist unique
morphisms $u'\cl A \to A$ and $u'' \cl C \to C$ such that we have the
commutative diagram:
$$
\xymatrix{
A \ar[r]^f \ar[d]^{u'} & B \ar[r]^g \ar[d]^u 
& C \ar[r]^h \ar[d]^{u''} & A[1] \ar[d]^{u'[1]} \\
A \ar[r]^f  & B \ar[r]^g   & C \ar[r]^h  & A[1] .
}
$$
\end{lemma}
\begin{proof}
  Since $\Hom(A,C) \simeq 0$ we have $g \circ (u \circ f) = 0$.  Since
  $\Hom(\cdot,\cdot)$ sends distinguished triangles to long exact sequences we
  deduce the existence of $u'$.  Its unicity of $u'$ follows from $\Hom(A,C[-1])
  \simeq 0$ in the same way.  The same proof works for $u''$.
\end{proof}

\begin{lemma}
\label{lem:automF}
Let $F \in \Derb(\cor_M)$ be as in
Proposition~\ref{prop:description_F_simpleLambda_cyl} and let
$\corC_\Omega[i] \to[u] \opb{p}(G) \to F \to[+1]$ be the
triangle~\eqref{eq:descrFcyl1}.  Then any automorphism $v \cl F \isoto F$
induces unique automorphisms $v' \cl \corC_\Omega \isoto \corC_\Omega$ and
$v'' \cl \opb{p} G \isoto \opb{p} G$ such that $(v',v'',v)$ is an automorphism
of the triangle~\eqref{eq:descrFcyl1}.
\end{lemma}
\begin{proof}
  We have $\roim{p}(\corC_\Omega) \simeq 0$, hence
  $\RHom(\opb{p} G, \corC_\Omega) \simeq 0$ by adjunction.  The existence and
  unicity of $v'$ and $v''$ follow from Lemma~\ref{lem:morph_triang}.  The fact
  that they are isomorphisms follows from the same result applied to $v^{-1}$
  and from the unicity.
\end{proof}

\begin{proposition}
\label{prop:RiqalpharF_Lambda_cyl}
Let $F \in \Derb(\cor_M)$ be as in
Proposition~\ref{prop:description_F_simpleLambda_cyl} and let $(\alpha,r) \in
\corC^\times \times \N^*$ and $i\in \Z$ be given.  We set for short $L =
\opb{p}L_{1/\alpha,r}$ and $F' = \roim{q}(F \otimes L)$, $F'_i =
R^iq_{\alpha,r}(F)$.  Let $v \in \Aut(F)$ be given.  We set $v' =
\roim{q}(v\otimes \id_L)$ and $v'_i = R^iq_{\alpha,r}(v)$.  Then
$$
(v')^{\mu+}_{q_0} = v^\mu_{p_0} \cdot (\id_{F'})^{\mu+}_{q_0}
\quad \text{and} \quad
(v'_i)^{\mu+}_{q_0} = v^\mu_{p_0} \cdot (\id_{F'_i})^{\mu+}_{q_0},
$$
where $q_0 = (0;-1) \in T^*J$ and $p_0 = (0,0;0,-1) \in T^*M$.
\end{proposition}
Remark~\ref{rem:dual_Lambda_cyl} also applies to this proposition and the result
also holds if we assume $\dot\SSi(F) = \Lambda^a$ instead of $\dot\SSi(F) =
\Lambda$.
\begin{proof}
  (i) Let $\corC_\Omega[j] \to[u] \opb{p}(G) \to F \to[+1]$ be the
  triangle~\eqref{eq:descrFcyl1}.  The tensor product with $L$ and the direct
  image by $q$ give the distinguished triangle
\begin{equation}
\label{eq:RiqalpharF1}
G' \to F' \to[f] \corC_{]0,1[}^r[j] \to[+1],
\end{equation}
where $G' = \roim{q}(\opb{p}(G) \otimes L)$.  Let $v_1 \cl \corC_\Omega[j]
\isoto \corC_\Omega[j]$ and $v_2 \cl \opb{p} G \isoto \opb{p} G$ be the
morphisms given by Lemma~\ref{lem:automF}.  We set $w_i = \roim{q}(v_i \otimes
\id_L)$.  Hence $(w_2, v', w_1[1])$ is an automorphism of the
triangle~\eqref{eq:RiqalpharF1}.

\sui (ii) Since $\Aut(\corC_\Omega) \simeq \corC$ we have $v_1 = (v_1)^\mu_{p_0}
\, \id_{\corC_\Omega[j]}$.  Hence $v_1\otimes \id_L$ and then $w_1$ are also
$(v_1)^\mu_{p_0}$ times the identity morphism. We deduce
\begin{equation}
\label{eq:RiqalpharF2}
(w_1)^{\mu+}_{q_0} = (v_1)^\mu_{p_0} 
\cdot (\id_{\corC_{]0,1[}^r[j]})^{\mu+}_{q_0} .
\end{equation}
We remark that $G'$ is constant on $J$. In particular $p_0 \not\in \SSi(G')$
and by Lemma~\ref{lem:umup_DMp} we deduce from~\eqref{eq:RiqalpharF2} that
$(v')^{\mu+}_{q_0} = (v_1)^\mu_{p_0} \cdot (\id_{F'})^{\mu+}_{q_0}$.  By
Lemma~\ref{lem:umup_DMp} again we have $v^\mu_{p_0} = (v_1)^\mu_{p_0}$ and we
obtain the first equality of the lemma.

\sui (iii) Since $F' \simeq \bigoplus_{i\in\Z} H^i(F')[-i]$ we deduce from~(ii)
that $(H^iv')^{\mu+}_{q_0} = (v_1)^\mu_{p_0} \cdot (\id_{H^iF'})^{\mu+}_{q_0}$
for all $i\in \Z$.  Now the morphism $c = c(F \otimes L)$ of
Definition~\ref{def:alpharcomp_cyl} commutes with $v'$
by~\eqref{eq:multc2_func}. Hence the $H^i(c)$ commute with the $H^j(v')$ and the
second equality of the lemma follows from Lemma~\ref{lem:vmu_kercoker} below.
\end{proof}

\begin{lemma}
\label{lem:vmu_kercoker}
Let $G,G' \in \Mod(\corC_\R)$ be such that
$\dot\SSi(G), \dot\SSi(G') \subset \{0\} \times \Rp \subset T^*_0\R$.  We set
$q_0 = (0;1)$.  Let
$$
\xymatrix{
  \ker(c) \ar[r] \ar[d]_{b}  & G \ar[r]^{c} \ar[d]_{a} 
&  G' \ar[r] \ar[d]^{a'} &   \coker(c) \ar[d]^{b'}  \\
  \ker(c) \ar[r]  &  G \ar[r]^{c} & G' \ar[r] &   \coker(c) }
$$
be a commutative diagram with exact rows.  We assume that there exists
$\alpha \in \corC$ such that
$a^{\mu+}_{q_0} = \alpha \cdot (\id_G)^{\mu+}_{q_0}$ and
$(a')^{\mu+}_{q_0} = \alpha \cdot (\id_{G'})^{\mu+}_{q_0}$.  Then
$b^{\mu+}_{q_0}$ and $(b')^{\mu+}_{q_0}$ are also $\alpha$ times the identity
morphism.
\end{lemma}
\begin{proof}
  Let us set $F_0 = \corC_{[0,+\infty[}$, $F_1 = \corC_\R$ and $F_2 =
  \corC_{]-\infty,0]}$.  The lemma is clear if $c$ is of the form $F_0^{m_0} \to
  F_0^{n_0}$, $F_2^{m_2} \to F_2^{n_2}$, $F_2^{m_2} \to F_1^{n_1}$ or $F_1^{m_1}
  \to F_0^{n_0}$.  We check that nothing strange happens when we mix these
  components.

  \sui (i) By the hypothesis on the microsupports we can write $G =
  \bigoplus_{i=0}^2 F_i^{m_i}$ and $G' = \bigoplus_{i=0}^2 F_i^{n_i}$.  We write
  $a$ in matrix form, that is, $a = (a_{ij})$ with $a_{ij} \cl F_j^{m_j} \to
  F_i^{n_i}$, and $a',c$ in the same way.  We write $a_{ij}$ as a scalar matrix
  (it is understood that we multiply it by the canonical morphism from $F_j$ to
  $F_i$).  Since $\Hom(F_i,F_j) = 0$ for $i\leq j$ and $\Hom(F_2,F_0) = 0$, we
  obtain upper triangular matrices with $(0,2)$ entry zero.  The hypothesis
  means that $a_{00}$, $a_{22}$, $a'_{00}$ and $a'_{22}$ are $\alpha$ times the
  identity matrix and we have to prove the same property for the matrices $b$
  and $b'$.

  \sui (ii) Multiplying $c$ on the left and the right by invertible matrices we
  can assume from the beginning that it has the following canonical
  form. Writing $M_{n,m}^r$ for the $n\times m$ matrix with entries $0$ except
  $(M_{n,m}^r)_{ii} = 1$ for $i=1,\ldots,r$ and $M_{n,m}^{r,r',s}$ for the
  matrix with entries $0$ except $(M_{n,m}^{r,r',s})_{r+i, r'+i} = 1$ for
  $i=1,\ldots,s$, we have integers $r_0,r_1,r_2$ and $s_0,s_1$ such that $c_{ii}
  = M_{n_i,m_i}^{r_i}$ for $i=0,1,2$ and $c_{i,i+1} =
  M_{n_i,m_{i+1}}^{r_i,r_{i+1},s_i}$ for $i=0,1$.

  \sui (iii) For $k=1,\ldots, m_0+m_1+m_2$ we let $E(k)$ be the subsheaf of $G =
  \bigoplus_{i=0}^2 F_i^{m_i}$ generated by the $k^{th}$ component.  For
  $k=m_0+1,\ldots, m_0+m_1$ we let $E'(k)$ be the subsheaf of $G$ generated by
  the image of $F_2 \to F_1$ in the $k^{th}$ component.  We see that $\ker(c)$
  is the sum of the following four summands
  \begin{align*}
    K_0 &= \langle E(r_0+1),\ldots, E(m_0) \rangle 
      \simeq F_0^{m_0-r_0} ,  \\
    K_1 &= \langle E(m_0+r_1+s_0+1),\ldots, E(m_0+m_1) \rangle 
      \simeq  F_1^{m_1-r_1-s_{0}} ,  \\
    K_2 &= \langle E(m_0{+}m_1{+}r_2{+}s_1{+}1),
      \ldots, E(m_0{+}m_1{+}m_2) \rangle \simeq F_2^{m_2-r_2-s_{1}} ,  \\
    K &= \langle E'(m_0+r_1+1),\ldots, E'(m_0+r_1+s_0) \rangle 
      \simeq F_2^{s_0}.
  \end{align*}
  The matrix $b$ is the square submatrix of $a$ corresponding to the columns and
  rows with the indices $i$ such that $E(i)$ or $E'(i)$ appears in $\ker(c)$.
  The summand $K_1$ does not contribute in $\mu hom(\ker(c),\ker(c))$.  It is
  clear that $b$ acts by multiplication by $\alpha$ on $K_0$ and $K_2$ because
  $a_{00}$ and $a_{22}$ are $\alpha$ times the identity matrix.

  Looking at the $(0,1)$ term in the relation $c\, a = a' \, c$ we find that the
  square submatrix of $a_{11}$ formed by the rows and columns $r_1+1,\ldots,
  r_1+s_0$ is $\alpha$ times the identity matrix. Hence $b$ acts by
  multiplication by $\alpha$ on $K$ also and we have proved the lemma for
  $\ker(c)$.

  \sui (iv) The duality $\DD'(\cdot)$ sends our sheaves $F_i$ to similar sheaves
  and $\coker(c)$ to $\ker(\DD'c)$. We deduce the result for $\coker(c)$ and
  this concludes the proof.
\end{proof}

\subsection{Microlocal linked points - first case}

Let $J$ be an open interval of $\R$.  In the case $M = \cer \times J$ we can
now give a refinement of Proposition~\ref{prop:mulink-projdim1} by replacing the
direct image $\roim{q}F$ with $R^iq_{\alpha,r}(F)$.  Let $F \in
\Derb(\corC_{\cer\times J})$ be such that $\Lambda' = \dot\SSi(F)$ is a smooth
Lagrangian submanifold and $F$ is simple along $\Lambda'$.  Let $t_0 \leq t_1
\in J$ be given.  For $k=0,1$ we assume that there exists an open interval $J_k$
around $t_k$ such that we can find a diffeomorphism $\cer \times J_k \simeq \cer
\times \mo]-1,1[$ sending $\cer \times \{t_k\}$ to $\cer \times \{1\}$ and
$\Lambda' \cap T^*(\cer \times J_k)$ to the $\Lambda$ described
in~\eqref{eq:defLambda_cyl} or to $\Lambda^a$.

\begin{corollary}
\label{cor:mulink-projRqar}
Let $i\in \Z$ and $(\alpha,r) \in \corC^\times\times\N^*$ be given.  We assume
that $R^iq_{\alpha,r}(F)$ has a decomposition $R^iq_{\alpha,r}(F) \simeq G
\oplus \corC_{J'}$ where $G \in \Mod(\corC_{J})$ and $J'$ is an interval with
ends $t_0, t_1$.  Let $p_k = (x_k;\xi_k) \in \Lambda$ be such that $q(x_k) =
t_k$ and $q_\pi(\opb{q_d}(p_k)) \in \dot\SSi(\corC_{J'})$. Then $p_0$ and $p_1$
are $F$-linked over any open subset containing $\opb{q}([t_0,t_1])$.
\end{corollary}
\begin{proof}
  Let $W$ be an open subset of $M$ containing $\opb{q}([t_0,t_1])$ and let $v
  \cl F|_W \isoto F|_W$ be given.  We set $F'_i = R^iq_{\alpha,r}(F)$, $v'_i =
  R^iq_{\alpha,r}(v)$ and $q_k = q_\pi(\opb{q_d}(p_k))$ for $k=0,1$.  By
  Proposition~\ref{prop:RiqalpharF_Lambda_cyl} we have, for $k=0,1$,
\begin{equation}
\label{eq:mulink-projRqar1}
(v'_i)^{\mu+}_{q_k} =  v^\mu_{p_k} \cdot (\id_{F'_i})^{\mu+}_{q_k} .
\end{equation}

On the other hand, using the decomposition $F'_i \simeq G \oplus \corC_{J'}$ we
can write $v'_i$ in matrix form $v'_i = (v_{ab})_{a,b = 1,2}$.  Then
$(v'_i)^{\mu+}_{q_k}$ is given by the matrix $((v_{ab})^{\mu+}_{q_k})$.  We have
$v_{22} = \alpha \, \id_{\corC_{J'}}$ for some scalar $\alpha$ and it follows
that $(v_{22})^{\mu+}_{q_k} = \alpha\cdot (\id_{\corC_{J'}})^{\mu+}_{q_k}$.

By~\eqref{eq:mulink-projRqar1} the matrix $((v_{ab})^{\mu+}_{q_k})$ is the
multiple $v^\mu_{p_k}$ of the identity matrix and we deduce $\alpha =
v^\mu_{p_k}$, for $k=0,1$.  Hence $v^\mu_{p_0} = v^\mu_{p_1}$ and we have proved
that $p_0$ and $p_1$ are $F$-linked.
\end{proof}

\subsection{Back to locally constant components}

By Corollary~\ref{cor:mulink-projRqar} we can detect $F$-linked points under
some condition on $R^iq_{\alpha,r}(F)$.  In this paragraph and the next one we
study the situation where the condition is not satisfied and give another
instance of $F$-linked points.

We restrict to the case~(b) of
Proposition~\ref{prop:description_F_simpleLambda_cyl}. We have $G \simeq
\bigoplus_{j\in\Z} H^j(G)[-j]$ and the proposition implies that we also have $F
\simeq \bigoplus_{j\in\Z} H^j(F)[-j]$.  Hence there is no loss of generality if
we assume that $G \in \Mod(\corC_\cer)$, $u \cl \corC_\Omega \to \opb{p} G$ is a
non zero morphism and $F = \coker(u)$. In particular we have the exact sequence
in $\Mod(\corC_M)$
\begin{equation}
\label{eq:notation_u_G_F}
0 \to \corC_\Omega \to [u] \opb{p} G \to F \to 0 .
\end{equation}

We recall that, by Proposition~\ref{prop:cons_sh_cercle} there exist a finite
family $\{I_a\}_{a \in A}$ of bounded intervals and a locally constant sheaf $L$
on $\cer$ such that
\begin{equation}
\label{eq:decG}
G \simeq L \oplus \bigoplus_{a\in A} \oim{e}( \corC_{I_a}) .
\end{equation}
(Since $G$ is simple, the exponents of the $\corC_{I_a}$ must be $1$.)
By~\eqref{eq:SSRiqalphar} we have $\dot\SSi(R^0q_{\alpha,r}(F) ) \subset \{0\}
\times (-\Rp)$ for any $(\alpha,r) \in \corC^\times\times\N^*$ and the
decomposition of $R^0q_{\alpha,r}(F)$ in Corollary~\ref{cor:cons_sh_R} becomes
\begin{equation}
\label{eq:decR0qarF}
R^0q_{\alpha,r}(F) 
\simeq \corC_J^a \oplus \corC_{]-1,0]}^b \oplus \corC_{]0,1[}^c ,
\end{equation}
for some integers $a,b,c$.

\begin{proposition}
\label{prop:alpharcomp_cyl}
Let $u \cl \corC_\Omega \to \opb{p} G$ be a non zero morphism and $F =
\coker(u)$.  We assume that $\roim{q} F$ decomposes as $\roim{q} F \simeq V_J
\oplus \corC_{]0,1[}$ for some $V \in \Derb(\corC)$.  Then we have (at least)
one of the following possibilities
  \begin{itemize}
  \item [(a)] there exists $(\alpha,r) \not= (1,1)$ such that in the
    decomposition~\eqref{eq:decR0qarF} we have $b \geq 1$,
  \item [(b)] there exists $a\in A$ such that the morphism $\corC_\Omega \to
    \opb{p} \oim{e}( \corC_{I_a})$ induced by $u$ and~\eqref{eq:decG} is non
    zero.
  \end{itemize} 
\end{proposition}
\begin{proof}
  (i) We write $G = G_1 \oplus L$ according to~\eqref{eq:decG}, where $G_1 =
  \bigoplus_a \oim{e}( \corC_{I_a})$.  We decompose $u = (u_1,u_L)$ in the same
  way.

  We assume that (b) does not hold, that is, $u_1=0$.  Then $F = \opb{p}(G_1)
  \oplus \coker(u_L)$ and $R^0q_{\alpha,r}(F)$ is the sum of
  $R^0q_{\alpha,r}(\opb{p}(G_1))$, which is a constant sheaf on $J$, and
  $R^0q_{\alpha,r}(\coker(u_L))$.  Hence we are reduced to the case where $G =
  L$, which we assume from now on.

  \sui (ii) We set $I = \Omega \cap (\cer \times \{1/2\})$ and $F' = F|_{\cer
    \times \{1/2\}}$.  We have the exact sequence $0 \to \corC_I \to[u'] L \to
  F' \to 0$ where $u' = u|_{\cer \times \{1/2\}}$.  We also have $F|_{\cer
    \times \{-1/2\}} \simeq L$.  Hence, by Lemma~\ref{lem:germ_alpharcomp_cyl},
  the hypothesis on $\roim{q} F$ translates as $\dim(H^0(\cer; F')) =
  \dim(H^0(\cer; L)) + 1$.  By Proposition~\ref{prop:morph_elem0} (applied to
  $u'\cl \corC_I \to L$) there exists $(\alpha,r) \in \corC^\times \times \N^*$
  such that $(\alpha,r) \not= (1,1)$ and $\dim(H^0_{\alpha,r}(F')) <
  \dim(H^0_{\alpha,r}(L))$.  By Lemma~\ref{lem:germ_alpharcomp_cyl} again, we
  obtain $a+c < a+b$, hence $b \geq 1$, as required.
\end{proof}

\subsection{Microlocal linked points - second case}

In this paragraph we consider the case~(b) of
Proposition~\ref{prop:alpharcomp_cyl}.  We keep the notations of the
proposition; in particular $u \cl \corC_\Omega \to \opb{p} G$ is a non zero
morphism and $F = \coker(u)$.  We will show that the points of
$\SSi(\corC_\Omega)$ are $F$-linked with points of $\SSi(\opb{p} G)$.

We use the decomposition~\eqref{eq:decG} and we set $G_a =
\oim{e}(\corC_{I_a})$. We denote by $i_a \cl G_a \to G$, $p_a \cl G \to G_a$ the
natural morphisms given by~\eqref{eq:decG}.  For a morphism $u \cl \corC_\Omega
\to \opb{p} G$ we define $A_u = \{a\in A$; $p_a \circ u \not=0\}$.  Hence $u$
factorizes through $\bigoplus_{a\in A_u} G_a \oplus L$.  We say that $u$ is
minimal if $A_u$ is minimal (for the inclusion) in the set $P_{A,u}$ of subsets
of $A$ defined by $P_{A,u} = \{A_{\phi \circ u}$; $\phi \in \Aut(G)\}$.

For $\phi \in \Aut(G)$ we have $\coker(\phi \circ u) \simeq \coker(u)$.  Hence
we can assume that $u$ is minimal without changing $\coker(u)$.

\begin{proposition}
\label{prop:linkedpt-minimalset}
Let $u \cl \corC_\Omega \to \opb{p} G$ be a non zero morphism and $F =
\coker(u)$.  We assume that $u$ is minimal in the above sense with respect to
the decomposition~\eqref{eq:decG} and that $A_u \not= \emptyset$.  Let $p_0 \in
\dot\SSi(\corC_\Omega)$ and $p_1 \in \dot\SSi(\opb{p}(\oim{e}( \corC_{I_a})))$
be given, where $a \in A_u$.  Then $p_0$ and $p_1$ are $F$-linked over $\cer
\times J$.
\end{proposition}
\begin{proof}
  (i) We set $G_0 = L$ and $A^+ = A \sqcup \{0\}$.  By Lemma~\ref{lem:automF}
  any automorphism $v \cl F \isoto F$ induces automorphisms $v' \cl \corC_\Omega
  \isoto \corC_\Omega$ and $v'' \cl \opb{p} G \isoto \opb{p} G$ such that $u
  \circ v' = v'' \circ u$.

  \sui (ii) We have $\End(\corC_\Omega) \simeq \corC$ and we can write $v' =
  \alpha\, \id$.  Let us write $u$ and $v''$ in matrix notation, that is, $u =
  (u_a)_{a\in A^+}$ and $v'' = (v''_{ba})_{a,b \in A^+}$ with $u_a \cl
  \corC_\Omega \to \opb{p} G_a$ and $v''_{ba} \cl \opb{p} G_a \to \opb{p} G_b$.
  In particular $v''_{aa} \in \End(G_a)$, where we identify $\End(G_a)$ and
  $\End(\opb{p}G_a)$.  By Lemma~\ref{lem:hom-imdir-interv}, for $a\in A$ we can
  write $v''_{aa} = \beta_a \id_{G_a} + n_a$ with $\beta_a \in \corC$ and $n_a$
  a nilpotent element of $\End(G_a)$.

  \sui (iii) Now we prove that $\beta_a = \alpha$ for all $a\in A_u$.  The
  relation $u \circ v' = v'' \circ u$ gives, for any $a\in A_u$,
$$
\alpha \cdot  u_a = (\beta_a \id_{G_a} + n_a) \circ u_a 
 + \sum_{b \in A^+ \setminus \{a\}} v''_{ab} \circ u_b .
$$
Let us assume that $\beta_{a_0} \not= \alpha$ for some $a_0 \in A_u$.  Then
$w_{a_0} = (\alpha-\beta_{a_0}) \id_{G_{a_0}} -n_{a_0}$ is invertible and
$u_{a_0} = \sum_{b \in A^+ \setminus \{a_0\}}w_{a_0}^{-1} \circ v''_{a_0 b}
\circ u_b$.  Hence we can write $u = \phi \circ u'$ where $u' \cl \corC_\Omega
\to \bigoplus_{a\in A^+} G_a$ and $\phi \cl G \to G$ are defined as follows.  We
set $u'_b = u_b$ for $b\in A^+ \setminus \{a_0\}$ and $u'_{a_0} = 0$.  We set
$\phi = \id_G + n$ where $n = (n_{ba})_{a,b \in A^+}$ is the nilpotent element
given by $n_{ba} = 0$ except $n_{a_0 b} = w_{a_0}^{-1} \circ v''_{a_0b}$ for $b
\in A^+ \setminus \{a_0\}$.  Then $\phi$ is invertible and this contradicts the
minimality of $u$. We thus have proved $\beta_a = \alpha$ for all $a\in A_u$.

\sui (iv) By Lemma~\ref{lem:umup_DMp} we have $v^\mu_{p_0} = (v')^\mu_{p_0}$ and
$v^\mu_{p_1} = (v'')^\mu_{p_1}$.  We see easily that $(v')^\mu_{p_0} = \alpha$.
By Lemma~\ref{lem:hom-imdir-interv}-(c) we have $(v'')^\mu_{p_1} = \beta_a$ and
the result follows from~(iii).
\end{proof}

\section{Simple sheaves with a skyscraper direct image}
\label{sec:skyscrim}

We remain in the setting of the previous section.  We set $M = \cer\times \R$
and we denote by $q\cl M \to \R$ the projection.  We let $\Lambda \subset
\dT^*M$ be a smooth closed connected conic Lagrangian submanifold.  We assume
that $\pi_M(\Lambda)$ is tangent to $\opb{q}(0)$ at two points and is on the
same side of $\opb{q}(0)$ near both points (see~\eqref{eq:hypLambda} below).  We
assume that there exists $F \in \Derb(\corC_{M})$ such that $F$ is simple along
$\Lambda$, $\roim{q}F$ has $\corC_{\{0\}}$ as a direct summand and $F$ is
constant outside a compact set (see~\eqref{eq:hypFimdir_cst}).  We prove in
Theorem~\ref{thm:clisoDUVF} that, if $\pi_M(\Lambda)$ has only two cusps, the
full subcategory of simple sheaves in $\Derb_\Lambda(\corC_{M})$ contains
uncountably many isomorphism classes.  For this we check that the hypothesis of
Corollary~\ref{cor:clisoDUVF} are satisfied and we first prove in
Proposition~\ref{prop:otherFlinkedpoint} that we have enough $F$-linked points.

The hypothesis that $\pi_M(\Lambda)$ has only two cusps is actually not really
used before the proof of Theorem~\ref{thm:clisoDUVF} and it appears in the proof
as follows: the points of $\Lambda$ which are not above the cusps and where $F$
has a constant shift are connected.

\subsection{Hypothesis on $\Lambda$ and $F$}
\label{sec:hypLambdaF}

Let $\Lambda \subset \dT^*M$ be a smooth conic Lagrangian submanifold such that
$\Lambda/\Rp$ is a circle.  We set $C_\Lambda = \dot\pi_{M}(\Lambda)$.  For $t$
in $\R$ we set $\fib_t = \opb{q}(t)$.  We assume
\begin{equation}
  \label{eq:hypLambda}
\left\{  \hspace{-4mm} \begin{minipage}[c]{11cm}
\begin{itemize}
\item [(a)] $C_\Lambda$ is a curve with only cusps and ordinary double points
  as singularities,
\item [(b)] the map $\Lambda/\Rp \to C_\Lambda$ is $1{:}1$ outside the double
  points,
\item [(c)] $C_\Lambda$ has exactly two cusps, $c_0$, $c_1$,
\item [(d)] $\fib_0$ does not contain any cusp or double point of $C_\Lambda$,
\item [(e)] $\fib_0$ is tangent to $C_\Lambda$ at two points, say $z_0$, $z_1$,
  and $\Lambda \cap T^*_{z_0}M = \{0\} \times \Rp$, $\Lambda \cap T^*_{z_1}M =
  \{0\} \times (-\Rp)$,
\item [(f)] near $z_0$ and $z_1$ we have
  $C_\Lambda \subset \opb{q}([0,+\infty[)$,
\item [(g)] the other intersections of $\fib_0$ and $C_\Lambda$ are
  transversal,
\item [(h)] any fiber $\fib_t$, for $t\not= 0$, contains at most one
  ``accident'' among: tangent point to $C_\Lambda$, double point or cusp.
\end{itemize}
  \end{minipage} \right.
\end{equation}

The inclusion of $\Lambda$ in $T^*M$ near $z_1$ and the inclusion of $\Lambda^a$
in $T^*M$ near $z_0$ are diffeomorphic to the situation described
in~\eqref{eq:defLambda_cyl}.  We can choose coordinates $(\theta,t)$ on $M$ and
$(\theta,t;\eta,\tau)$ on $T^*M$, with $\theta\in \mo]-\pi,\pi]$ and $t\in \R$,
so that the geometry is described as follows.  We set $J = \mo]-2,2[$ and $\bb =
\opb{q}(J)$.  We let $z_0 = (-2,0)$, $z_1 = (2,0)$, $z'_a$, $a\in A$, be the
intersection points of $\fib_0$ and $C_\Lambda$ ($A$ is a finite set).  The
curve $C_\Lambda \cap \bb$ is the disjoint union
\begin{equation*}
C_\Lambda \cap \bb = P_0 \sqcup P_1
\sqcup \bigsqcup_{a \in A} (\{z'_a\} \times J), 
\end{equation*}
where $P_0$ and $P_1$ are the parabolas
\begin{align*}
 P_0 &= \{ (\theta, 2(\theta+2)^2); \; \theta\in \mo]-3,-1[\}, \\ 
 P_1 &= \{ (\theta, 2 (\theta-2)^2); \;
 \theta\in \mo]1,3\mc[ \} .
\end{align*}
We let $\Omega_0$, $\Omega_1$ be the open subsets bounded by $P_0$, $P_1$, that
is, $\Omega_0 = \{ (\theta,y)\in \bb$; $y > 2 (\theta+2)^2\}$ and
$\Omega_1 = \{ (\theta,y)\in \bb$; $y > 2 (\theta-2)^2\}$.  We also set
$\Lambda_{P_0} = \Lambda \cap \opb{\pi_M}(P_0)$ and
$\Lambda_{P_1} = \Lambda \cap \opb{\pi_M}(P_1)$.

We denote by $z'_i = (\theta'_i,1)$, $z''_i = (\theta''_i,1)$ the intersections
of $P_i$ and $\fib_{1}$, for $i= 0, 1$.  We also denote by $p_0$, $p_1$,
$p'_0$, $p''_0$, $p'_1$, $p''_1$ the points of $\Lambda$ (well-defined up to a
positive multiple) above $z_0$, $z_1$, $z'_0$, $z''_0$, $z'_1$, $z''_1$.

\medskip

In this section we consider a simple sheaf $F \in \Derb(\corC_{M})$ along
$\Lambda$ and we assume
\begin{equation}
  \label{eq:hypFimdir_cst}
\left\{  \hspace{-4mm} \begin{minipage}[c]{11cm}
\begin{itemize}
\item [(a)] $\dot\SSi(F) = \Lambda$ and $F \in \Dersf_\Lambda(\cor_M)$,
\item [(b)] there exists $L \in \Derb(\corC)$ such that $(\roim{q}F)|_J \simeq
  \corC_{\{0\}} \oplus L_J$,
\item [(c)] there exists $A>0$ such that $F|_{M \setminus (\cer \times
    [-A,A])}$ has constant cohomology sheaves.
\end{itemize}
  \end{minipage} \right.
\end{equation}

Since $F$ is simple along $\Lambda$, it has a shift $s(p) \in \Z$ at any $p\in
\Lambda$.  As recalled after Definition~\ref{def:simple_pure}, the shift is
locally constant outside the cusps and changes by $1$ when $p$ crosses a cusp.
Let us denote by $\Xi_{c_0} = \Lambda \cap T^*_{c_0}M$ and $\Xi_{c_1} = \Lambda
\cap T^*_{c_1}M$ the half-lines of $\Lambda$ over the cusps.  Hence $s(p)$ takes
two values over $\Lambda \setminus (\Xi_{c_0} \cup \Xi_{c_1})$, say $s$ and
$s+1$.  We let $\Lambda_0$ be the connected component of $\Lambda \setminus
(\Xi_{c_0} \cup \Xi_{c_1})$ where the shift is $s$ and we let $\Lambda_1$ be the
other component.  We also recall that $\cor_{\mo]-\infty,0[}$ has shift $-1/2$
at $(0;1)$ and $\cor_{\mo]-\infty,0]}$ has shift $1/2$ at $(0;-1)$ (see
Example~\ref{ex:shift}).  In particular the constant sheaf on a closed or open
interval has the same shifts at both ends. The constant sheaf on a half-closed
interval has different shifts at both ends.

\begin{remark}
\label{rem:dualite}
The hypothesis~\eqref{eq:hypFimdir_cst} is stable by duality up to the symmetry
$(\theta,t) \mapsto (-\theta,t)$ of $\cer \times \R$ and up to changing
$\Lambda$ by $\Lambda^a$.  Indeed we have, for a general $F \in \Derb(\corC_M)$,
$\SSi(\DD F) = (\SSi(F))^a$ and $\roim{q}(\DD F) \simeq \DD (\reim{q}F)$.
Moreover, on $\R$ we have $\DD(\corC_{\{0\}}) \simeq \corC_{\{0\}}$.
\end{remark}

\subsection{Some $F$-linked points}

By Proposition~\ref{prop:cons_sh_cercle}, for any $t\in \R$ the sheaf
$F|_{\fib_t}$ can be decomposed uniquely (though non canonically) as a direct
sum of locally constant sheaves and sheaves of the type $\oim{e}(\corC_I)[i]$
for intervals $I\subset \R$ and shifts $i\in \Z$, where $e \cl \R \to \cer$ is
the natural map.
Moreover, if $\fib_t$ meets $C_\Lambda$ only at smooth points and transversally,
then Lemma~\ref{lem:bonne_im_inv} implies that $F|_{\fib_t}$ is simple.  In this
case, for a given $(x;\xi) \in \dT^*\fib_t$ there exists at most one interval
$I$ in the decomposition such that $(x;\xi) \in \SSi(\oim{e}(\corC_I))$.

We say that two points $(x;\xi)$, $(x';\xi') \in \Lambda$ are $F$-conjugate if
there exists $t$ such that 
\begin{itemize}
\item [-] $q(x) = q(x') = t$ and $\fib_t$ meets $C_\Lambda$ only at smooth points
  and transversally,
\item [-] $F|_{\fib_t}$ has a direct summand $\oim{e}(\corC_I)[i]$ where $I$ has
  ends $a$ and $b$ such that $x=e(a)$ and $x' = e(b)$,
\item [-] $(x;\xi_1)$, $(x';\xi'_1) \in \dot\SSi(F|_{\fib_t})$ where $\xi_1$,
  $\xi'_1$ are the projections of $\xi$, $\xi'$ to $T^*\fib_t$.
\end{itemize}
By unicity of the decomposition, if $(x;\xi)$ has a conjugate point, it is
unique (up to a positive multiple).
In particular, the points $p'_0$, $p''_0$, $p'_1$, $p''_1$
have conjugate points, that we denote by $\ol{p'_0}$, $\ol{p''_0}$,
$\ol{p'_1}$, $\ol{p''_1}$.

We will also use the following notations:
\begin{equation}
\label{eq:not_WiUi}
\begin{alignedat}{2}
W_0 &= \mo]-3,-1\mc[ \times J, 
&\quad U_0 &= W_0 \setminus \ol{\Omega_0} ,  \\
W_1 &= \mo]1,3\mc[ \times J, 
&\quad U_1 &= W_1 \setminus \ol{\Omega_1} .
\end{alignedat}
\end{equation}
We recall that $W_0 \cap C_\Lambda = P_0$ and $W_1 \cap C_\Lambda = P_1$.

\begin{lemma}
\label{lem:autour_des_deux_parab}
There exist $L^0 ,L^1 \in \Derb(\corC)$ such that
\begin{equation}
\label{eq:aut_parab1}
F|_{W_0} \simeq \corC_{U_0}[1] \oplus L^0_{W_0}
\quad \text{and} \quad
F|_{W_1} \simeq \corC_{\ol{U_1} \cap W_1}[1] \oplus L^1_{W_1} .
\end{equation}
\end{lemma}
\begin{proof}
  (i) By duality the second isomorphism is reduced to the first one (see
  Remark~\ref{rem:dual_Lambda_cyl}) and we only prove the first one.  As in
  part~(i) of the proof of
  Proposition~\ref{prop:description_F_simpleLambda_cyl} we know that $F|_{W_0}$
  is the sum of a constant sheaf, say $L^0_{W_0}$, and $\corC_V[d]$ where $d$
  is some integer and $V = \ol{\Omega_0}$ or $V = U_0$.  However, if
  $V = \ol{\Omega_0}$, the decomposition
  $F|_{W_0} \simeq \corC_V[d] \oplus L^0_{W_0}$ extends to $\bb$ and
  $\roim{q}(\corC_V[d])$ is a direct summand of $\roim{q}F|_J$.  Since
  $\roim{q}(\corC_{\ol{\Omega_0}}[d]) \simeq \corC_{[0,2[}[d]$, this
  contradicts~\eqref{eq:hypFimdir_cst}-(b). Hence $V = U_0$.

  \sui (ii) It remains to see that $d=1$. There exists a non zero morphism
  $u\cl \corC_{\ol{\Omega_0}}[d-1] \to F|_{W_0}$ and a distinguished triangle
  $\corC_{\ol{\Omega_0}}[d-1] \to F|_{W_0} \to L^0_{W_0} \oplus \corC_{W_0}[d]
  \to \corC_{\ol{\Omega_0}}[d]$.
  Then the morphism $u$ extends to $\bb$ as
  $u'\cl \corC_{\ol{\Omega_0}}[d-1] \to F|_{\bb}$.  Let $F'$ be the cone of
  $u'$.  Then $\SSi(F') = (\Lambda \cap T^*\bb) \setminus \Lambda_{P_0}$ and
  $\SSi(\roim{q}F')$ does not contain $(0;1)$.  It follows that $\corC_{\{0\}}$
  cannot be a direct summand of $\roim{q}F'$.  Since $\roim{q}F'$ is the cone
  of
$$
\roim{q}(u') \cl \corC_{[0,2[}[d-1] \to \roim{q}F|_J \simeq
  \corC_{\{0\}} \oplus L_J ,
$$
we deduce that we have a non zero morphism
$ \corC_{[0,2[}[d-1] \to \corC_{\{0\}}$.  This implies that $d=1$ and the lemma
is proved.
\end{proof}

It follows from Lemma~\ref{lem:autour_des_deux_parab} and Example~\ref{ex:shift}
that the points of $\Lambda_{P_0}$ have shift $1/2$ and the points of
$\Lambda_{P_1}$ have shift $3/2$.  Hence $s=1/2$, $\Lambda_{P_0} \subset
\Lambda_0$ and $\Lambda_{P_1} \subset \Lambda_1$.

\begin{lemma}
\label{lem:shift_extremites}
We have either $\ol{p'_0} = p''_0$ or $F$ has different shifts at $p'_0$ and
$\ol{p'_0}$.
In the same way $\ol{p''_0} = p'_0$ or $F$ has different shifts at $p''_0$ and
$\ol{p''_0}$ and the same assertions hold for the points $p'_1$, $p''_1$.
\end{lemma}
\begin{proof}
  By symmetry between $p'_0$ and $p''_0$ and by Remark~\ref{rem:dualite} it is
  enough to prove the first assertion. 

  \sui (i) Let us write $p'_0 = (z'_0; \xi'_0)$ and $\ol{p'_0} = (\ol{z'_0},
  \ol{\xi'_0})$.  By the definition of conjugate point we can write $F|_{\fib_1}
  \simeq \oim{e}\corC_I[d] \oplus F'$ where $e \cl \R \to \fib_1 = \cer$ is the
  quotient map and $I$ is an interval of $\R$ with ends $a,b$ such that $e(a) =
  z'_0$ and $e(b) = \ol{z'_0}$.  Then the lemma is equivalent to $\ol{z'_0} =
  z''_0$ or $I$ is half-closed.

  \sui (ii) Using~\eqref{eq:aut_parab1}, we see that $F|_{\fib_1 \cap W_0}$ has
  a direct summand which is $\corC_{\fib_1 \cap U_0}[1]$.  We have $\fib_1 \cap
  U_0 = \mo]a',z'_0\mc[ \cup \mo]z''_0,b'\mc[$ where $a' = (-3,1)$ and
  $b'=(-1,1)$. It follows that $\corC_{]a',z'_0[}[1]$ is a direct summand of
  $(\oim{e}\corC_I[d])|_{W_0}$. Hence $d=1$ and $I$ is open near its end
  projecting to $z'_0$.

\newcommand{\clI}{{\ol{I}}}

\sui (iii) Now we assume that $I$ is not half-closed and we prove that
$\ol{z'_0} = z''_0$.  By~(ii) $I$ is open and, setting $G = (\DD' F)[1]$, $G' =
\DD' F'$, we have $G|_{\fib_1} \simeq \oim{e}\corC_\clI \oplus G'$.  We let
$1_{\clI}$ be the canonical section of $\corC_{\clI}$ and, using the given
decomposition of $G|_{\fib_1}$, we set
$$
s = (1_{\clI},0) \in H^0(\R;\corC_{\clI}) \oplus H^0(\cer;G')
\simeq H^0(\cer; G|_{\fib_1}) .
$$
By proper base change we have
$H^0(\cer; G|_{\fib_1}) \simeq (H^0(\roim{q}G))_1$ and $s$ corresponds to a
germ $\bar s$ of $H^0(\roim{q}G)$ at $1$.  By the hypothesis on $\roim{q}F$ the
restriction $\roim{q}G|_J$ is the sum of $\corC_{\{0\}}[-1]$ and a constant
sheaf. Hence we can extend $\bar s$ to
$$
s' \in H^0(J;\roim{q}G) \simeq H^0(B;G) .
$$
Both $s$ and $s'$ induce sections of the cohomology sheaves in degree $0$ that
we denote respectively by $s_0 \in \sect(\cer; H^0G|_{\fib_1})$ and $s'_0 \in
\sect(B;H^0G)$ and we have $s'_0|_{\fib_1} = s_0$.

\sui (iv) The decomposition~\eqref{eq:aut_parab1} gives
$H^0G|_{W_0} \simeq \corC_{\ol{U_0} \cap W_0} \oplus \corC^n_{W_0}$ for some
integer $n$.  Let us write $s'_0|_{W_0} = (s_1,s_2)$ according to this
decomposition.  Since $\ol{U_0} \cap W_0$ is connected we have either $s_1 = 0$
or $\supp(s_1) = \ol{U_0} \cap W_0$.  Hence either $s_1|_{\fib_1} = 0$ or
$\supp(s_1)|_{\fib_1} = \mo]a',z'_0] \cup [z''_0,b'[$.

Since $s'_0|_{\fib_1} = s_0$ is induced by the section $1_{\clI}$ and
$\corC_{]a',z'_0]}$ is a direct summand of $(\oim{e}\corC_\clI)|_{W_0}$, we
must have $s_1 \not=0$.  It follows that $\corC_{[z''_0,b'[}$ is also a direct
summand of $(\oim{e}\corC_\clI)|_{W_0}$ and we obtain $\ol{z'_0} = z''_0$ as
claimed.
\end{proof}

\begin{lemma} 
\label{lem:p0p1lies}
Let $V$ be any neighborhood of $\fib_0$. Then the points $p_0$ and $p_1$ are
$F$-linked over $V$.
\end{lemma}
\begin{proof}
  This is a special case of Proposition~\ref{prop:mulink-projdim1}, with
  $t_0 = t_1 =0$ and $J = \{0\}$.
\end{proof}

\begin{proposition}
\label{prop:otherFlinkedpoint}
We set $U = \opb{q}(\mo]-\infty,1[)$. Then there exists
$p \in \Lambda \cap T^*U$ such that $p_1$ and $p$ are $F$-linked over $U$ and
$p \not\in \Xi_{c_0} \cup \Xi_{c_1} \cup \Lambda_{P_0} \cup \Lambda_{P_1}$.
\end{proposition}
\begin{proof}
  (i) We have the excision distinguished triangle $\corC_{U_0} \to \corC_{W_0}
  \to \corC_{\ol{\Omega_0}} \to[a] \corC_{U_0}[1]$.  Using the
  decomposition~\eqref{eq:aut_parab1} the morphism $a$ induces $b \cl
  \corC_{\ol{\Omega_0}} \to F|_{W_0}$.  Then $b$ extends uniquely to a morphism
  $c \cl \corC_{\ol{\Omega_0}} \to F|_U$ and we define $F' \in \Derb(\corC_U)$ by
  the distinguished triangle
\begin{equation}
\label{eq:otherFlinkedpoint1}
\corC_{\ol{\Omega_0}\cap U} \to[c] F|_U \to F'   \to[+1] .
\end{equation}
Then $F'|_{W_0}$ is constant and
$\dot\SSi(F') = (\Lambda \cap T^*U) \setminus \Lambda_{P_0}$.  We deduce that
$\dot\SSi(\roim{q}(F'))|_J \subset \{0\} \times \Rp$ and $\corC_{\{0\}}$ cannot
be a direct summand of $\roim{q}(F')$.  The direct image
of~\eqref{eq:otherFlinkedpoint1} by $q$ gives a distinguished triangle
$\corC_{[0,1[} \to[c'] \roim{q}(F|_U) \to \roim{q}(F') \to[+1]$ and the
morphism $\corC_{[0,1[} \to \corC_{\{0\}}$ induced by $c'$ is not $0$.  Hence
$\roim{q}(F')|_J$ is the sum of $ \corC_{]0,1[}$ and a constant sheaf.

\sui (ii) Since $\dot\SSi(F') = (\Lambda \cap T^*U) \setminus \Lambda_{P_0}$,
the sheaf $F'|_{\opb{q}(\mo]-2,1[)}$ satisfies the hypothesis of
Proposition~\ref{prop:description_F_simpleLambda_cyl} with $\Omega = \Omega_1$.
By~\eqref{eq:aut_parab1} we are in the case~(b) of the proposition.  Hence there
exist $G_0 \in \Mod(\corC_\cer)$, a non zero morphism $u \cl \corC_{\Omega_1}
\to \opb{p} G_0$ and $G'' \in \Derb(\corC_\cer)$ such that
$F'|_{\opb{q}(\mo]-2,1[)} \simeq \opb{p} G''|_{\opb{q}(\mo]-2,1[)} \oplus
F_0[1]$, where $F_0 \in \Mod(\corC_{\opb{q}(\mo]-2,1[)})$ is the cokernel of
$u$:
\begin{equation}
\label{eq:otherFlinkedpoint0}
0 \to \corC_{\Omega_1} \to[u] \opb{p} G_0|_{\opb{q}(\mo]-2,1[)} \to F_0 \to 0 .
\end{equation}
Since $\roim{q}(\opb{p} G'')$ is constant, we deduce from~(i) that
$\roim{q}(F_0)$ is the sum of $ \corC_{]0,1[}$ and a constant sheaf on
$\mo]-2,1[$.

We are thus in the situation of Proposition~\ref{prop:alpharcomp_cyl} and we
obtain the conclusion~(a) or~(b) of this proposition. We study these cases
in~(iv) and~(v) below.

\sui (iii) We consider a morphism $f\cl F \to F$ and we will relate
$f^\mu_{p_1}$ with some other $f^\mu_p$.  We remark that
$\RHom(\corC_{\ol{\Omega_0}},F') \simeq 0$ because $F'$ is constant on $W_0$.
Applying Lemma~\ref{lem:morph_triang}, as in the proof of
Lemma~\ref{lem:automF}, we deduce that $f$ induces unique morphisms $f' \cl F'
\to F'$ and $\omega_0 \cl \corC_{\ol{\Omega_0}} \to \corC_{\ol{\Omega_0}}$
defining an automorphism of the triangle~\eqref{eq:otherFlinkedpoint1}. 

By Lemma~\ref{lem:umup_DMp} we have $f^\mu_{p_1} = (f')^\mu_{p_1}$ and
$f^\mu_{p} = (f')^\mu_{p}$ if $p \not\in \SSi(\corC_{\Omega_0})$.  Hence it is
enough to find $p$ such that $p_1$ and $p$ are $F'$-linked.

\sui (iv) As remarked before Proposition~\ref{prop:linkedpt-minimalset} we can
replace $u$ by $\phi \circ u$, with $\phi \in \Aut(G)$,
in~\eqref{eq:otherFlinkedpoint0} without changing $\coker(u)$. Hence we can
assume that $u$ is minimal. If $A_u \not= \emptyset$, we choose $p \in
\dot\SSi(\opb{p}(\oim{e}( \corC_{I_a}))$ for any $a \in A_u$.
Proposition~\ref{prop:linkedpt-minimalset} implies that $p_1$ and $p$ are
$F_0$-linked over $U$.  Since $F_0[1]$ is a direct summand of
$F'|_{\opb{q}(\mo]-2,1[)}$, the automorphism $f'$ induces an automorphism $f_0$
of $F_0$.  By Lemma~\ref{lem:umup_DMp} we have $(f_0)^\mu_{p_1} =
(f')^\mu_{p_1}$ and $(f_0)^\mu_{p} = (f')^\mu_{p}$.  Hence $p_1$ and $p$ are
also $F'$-linked, as required.  It is clear that $p \not\in \Xi_{c_0} \cup
\Xi_{c_1} \cup \Lambda_{P_0} \cup \Lambda_{P_1}$.

\sui (v) If $A_u = \emptyset$, then we are in the case~(a) of
Proposition~\ref{prop:alpharcomp_cyl}.  Hence there exists $(\alpha,r) \not=
(1,1)$ such that $R^0q_{\alpha,r}(F_0) \simeq \corC_{]-2,1[}^a \oplus
\corC_{]-2,0]}^b \oplus \corC_{[0,1[}^c$ with $b\geq 1$.  Since $F_0[1]$ is a
direct summand of $F'|_{\opb{q}(\mo]-2,1[)}$ it follows that
$R^{-1}q_{\alpha,r}(F')|_{\mo]-2,1[}$ has $\corC_{]-2,0]}$ as direct summand.

By the hypothesis~(c) of~\eqref{eq:hypFimdir_cst}, $F$, hence $F'$, is constant
on $\cer\times \mo]-\infty,T[$ for $T\ll0$.  It follows that
$(R^{-1}q_{\alpha,r}(F'))_{]-\infty,1/2[}$ has compact support.  By
Corollary~\ref{cor:cons_sh_R} we can decompose
$R(^{-1}q_{\alpha,r}(F'))_{]-\infty,1/2[}$ into a direct sum of constant sheaves
on intervals.  We obtain one interval, say $I$, of the form $I = [t_0,0]$ or $I
= \mo]t_0,0]$, for some $t_0 \in \mo]-\infty,-2]$.  By the hypothesis~(h)
of~\eqref{eq:hypLambda} the fiber $\fib_{t_0}$ contains no cusp or double point
and there exists a unique $p = (x;\xi) \in \Lambda$ (up to a positive multiple)
such that $q(x) = t_0$ and $q_\pi(\opb{q_d}(p)) \in \dot\SSi(\corC_{I})$.  By
Corollary~\ref{cor:mulink-projRqar} the points $p$ and $p_1$ are $F'$-linked
over $U$.  It is clear that $p \not\in \Xi_{c_0} \cup \Xi_{c_1} \cup
\Lambda_{P_0} \cup \Lambda_{P_1}$ and this concludes the proof.
\end{proof}

\subsection{Simple sheaves along $\Lambda$}

Let $\Lambda \subset \dT^*M$ be a smooth conic Lagrangian submanifold
satisfying~\eqref{eq:hypLambda}.  We recall that $\Derb_\Lambda(\corC_{M})$ is
the full subcategory of $\Derb(\corC_{M})$ formed by the $G$ such that
$\dot\SSi(G) \subset \Lambda$.

\begin{theorem}
\label{thm:clisoDUVF}
We assume that there exists $F \in \Derb(\corC_{M})$ which
satisfies~\eqref{eq:hypFimdir_cst}. We set $U = M \setminus \fib_1$ and $V =
\opb{q}(]0,2[)$. Then the category $\Der(U,V;F)$ contains uncountably many
isomorphism classes.  In particular the full subcategory of simple sheaves in
$\Derb_\Lambda(\corC_{M})$ contains uncountably many isomorphism classes.
\end{theorem}
\begin{proof}
  We will apply Corollary~\ref{cor:clisoDUVF}. For this we prove that there
  exist a path $\gamma \cl [0,1] \to \Lambda$, such that $\gamma$, $H = \fib_1$,
  $U$ and $V$ satisfy~\eqref{eq:hyppathgamma} and~\eqref{eq:hypUVF}.

  \sui (i) By Proposition~\ref{prop:otherFlinkedpoint} there exists $p_2 =
  (z_2;\xi_2) \in \Lambda \cap T^*U$ such that $p_1$ and $p_2$ are $F$-linked
  over $U$ and $p_2 \not\in \Xi_{c_0} \cup \Xi_{c_1} \cup \Lambda_{P_0} \cup
  \Lambda_{P_1}$.  By Lemma~\ref{lem:p0p1lies} it follows that $p_2$, $p_0$ and
  $p_1$ are $F$-linked over $U$.

  The point $p_2$ is in $\Lambda_0$ or $\Lambda_1$.  We assume $p_2 \in
  \Lambda_0$, the other case being similar.  Let $\gamma \cl [0,1] \to \Lambda$
  be an embedding such that $\gamma(0) = p_0$, $\gamma(1) =p_2$, $\gamma([0,1])
  \subset \Lambda_0$ and $\pi_M \circ \gamma$ is an immersion (hence $\pi_M
  \circ \gamma$ describes the portion of the curve $C_\Lambda$ between $z_0$ and
  $z_2$ which does not contain any cusp).  Then $\gamma$ meets either the
  half-line $\Rp\cdot p'_0$ or the half-line $\Rp\cdot p''_0$.  We
  assume it meets $\Rp\cdot p'_0$, the other case being similar.  By
  Lemma~\ref{lem:shift_extremites} we have either $\ol{p'_0} = p''_0$ or
  $\ol{p'_0}$ belongs to $\Lambda_1$. In both cases the path $\gamma$ does not
  meet $\Rp\cdot \ol{p'_0}$.

  \sui (ii) We can decompose $F|_H$ as $F|_H \simeq G' \oplus G''$ where $G' =
  \oim{e}(\corC_I)$ satisfies $\dot\SSi(G') = \Rp\cdot p'_0 \sqcup \Rp\cdot
  \ol{p'_0}$.  We can also find a submersion $r \cl V \to H$ such that
  $\dot\SSi(F|_V) = r_d(\opb{r_\pi}(\dot\SSi(F|_H)))$ and we deduce $F|_V \simeq
  \opb{r}(F|_H)$ by Corollary~\ref{cor:opbeqv}.  Then, setting $F' =
  \opb{r}(G')$ and $F'' = \opb{r}(G'')$, we see that~\eqref{eq:hyppathgamma}
  and~\eqref{eq:hypUVF} are satisfied.  Now the first assertion of the theorem
  follows from Corollary~\ref{cor:clisoDUVF}.

  \sui (iii) The property ``$\dot\SSi(F_1) \subset \Lambda$ and $F_1$ is simple
  along $\Lambda$'' is local. Hence $\Der(U,V;F)$ is a subcategory of the
  category of simple sheaves in $\Derb_\Lambda(\corC_{M})$.  The second
  assertion of the theorem follows.
\end{proof}

\section{Proof of the three cusps conjecture}
\label{sec:proof}

\newcommand{\Phiproj}{\Psi}
\newcommand{\Phistd}{\Phi}

In this section we prove the three cusps conjecture.  As explained in the
introduction we will use the following result of~\cite{GKS10}.  Let $I$ be an
open interval containing $0$ and let $\Phi \cl \dT^*M \times I \to \dT^*M$ be a
Hamiltonian isotopy such that $\Phi_0 = \id$ and $\Phi_t$ is homogeneous for the
action of $\Rp$, that is, $\Phi_t(x;\lambda\,\xi) = \lambda \cdot \Phi_t(x;\xi)$
for all $\lambda \in \Rp$ and all $(x;\xi) \in \dT^*M$, where we set as usual
$\Phi_t(\cdot) = \Phi(\cdot,t)$.  The graph of $\Phi_t$ is a conic Lagrangian
submanifold $\Lambda_{\Phi_t}$ of $\dT^*M^2$ (up to multiplication by $-1$ in the
fibers of the second $M$). We can find a unique conic Lagrangian submanifold
$\Lambda_\Phi$ of $\dT^*(M^2\times I)$ such that $\Lambda_{\Phi_t} = i_{t,d}
(\opb{i_{t,\pi}} (\Lambda_\Phi))$ for all $t\in I$, where $i_t \cl M^2\times \{t\}
\to M^2 \times I$ is the inclusion.

\begin{theorem}[Theorem~3.7 of~\cite{GKS10}]
\label{thm:qhi}
There exists a unique $K \in \Der(\cor_{M^2 \times I})$ such that $\dot\SSi(K)
= \Lambda_\Phi$ and $\opb{i_0}K = \cor_{\Delta}$ where $\Delta$ is the diagonal of
$M^2$.
\end{theorem}
In~\cite{GKS10} it is also proved that $K_t = \opb{i_t}K$ is invertible for the
composition.  We recall that, for three manifolds $M_1,M_2,M_3$ and
$K \in \Derb(\cor_{M_1 \times M_2})$, $L \in \Derb(\cor_{M_2 \times M_3})$ we
define $K \circ L = \reim{q_{13}}(\opb{q_{12}}K \otimes \opb{q_{23}}L)$, where
$q_{ij}$ are the projections from the triple product to $M_i\times M_j$.  Then
there exists $K_t^{-1}$ such that
$K_t \circ K_t^{-1} \simeq K_t^{-1} \circ K_t \simeq \cor_\Delta$.  Choosing
$M_1=M_2=M$ and $M_3$ a point, we also obtain an action of $K_t$ on
$\Der(\cor_M)$ by $F \mapsto K_t \circ F$.  This is an equivalence with
$K_t^{-1} \circ \cdot$ as inverse.  When $M$ is compact this induces an
equivalence, for  any closed conic $\Lambda \subset \dT^*M$
\begin{equation}
\label{eq:comp=eq}
\Derb_\Lambda(\cor_{M}) \isoto \Derb_{\Phi_t(\Lambda)}(\cor_{M}) .
\end{equation}
(If $M$ is not compact, this is still true for locally bounded derived
categories.)  Actually the sheaf $K_t$ is simple along $\Lambda_{\Phi_t}$.  The
composition of two simple sheaves is simple (see~\S7 in~\cite{KS90}).  We
recall the notation $\Dersf_\Lambda(\cor_M)$ of
Definition~\ref{def:simple_pure}.  Then~\eqref{eq:comp=eq} induces an
equivalence
\begin{equation}
\label{eq:comp=eqSimp}
\Dersf_\Lambda(\cor_{M}) \isoto \Dersf_{\Phi_t(\Lambda)}(\cor_{M}) .
\end{equation}

We can give another formulation of this equivalence as follows (see
Corollary~3.13 of~\cite{GKS10}).  We define
$\Lambda' \subset \dT^*(M \times I)$ by
$$
\Lambda' = \Lambda_\Phi \circ \Lambda
= \reim{p_{13}}( \Lambda_\Phi \cap \opb{p_{2}} \Lambda^a) ,
$$
where $p_{13} \cl T^*(M^2 \times I) \to T^*(M \times I)$ and
$p_2 \cl T^*(M^2 \times I) \to T^*M$ are the projections to the factors
corresponding to the indices.  We have
$\Phi_t(\Lambda) = j_{t,d} (\opb{j_{t,\pi}} (\Lambda'))$ where
$j_t \cl M \times \{t\} \to M \times I$ is the inclusion.  We can deduce
from~\eqref{eq:comp=eq} that $\opb{j_t}$ gives an equivalence
\begin{equation}
\label{eq:restr=eq}
\opb{j_t} \cl \Derb_{\Lambda'}(\cor_{M}) \isoto \Derb_{\Phi_t(\Lambda)}(\cor_{M}) .
\end{equation}
Conversely, the composition of $\opb{j_t}$ and the inverse of $\opb{j_0}$ is
the equivalence $F \mapsto K_t \circ F$ of~\eqref{eq:comp=eq}.

In~\cite{GKS10} (see Corollary~1.7) it is also remarked that the global
sections of $K_t \circ F$ do not depend on $t$:
\begin{equation}
\label{eq:cohom_indpt-t}
\rsect(M; K_t \circ F ) \simeq \rsect(M; F)
\qquad \text{for all $t\in I$}.
\end{equation}

\begin{lemma}
\label{lem:qhisym}
In the situation of Theorem~\ref{thm:qhi} we assume that $\Phi_t$ is homogeneous
for the action of $\Rm$ (and not only $\Rp$) for all $t\in I$.  Then, for any $F
\in \Der(\cor_M)$ and any $t\in I$, there exists an isomorphism $\DD(K_t \circ
F) \simeq K_t \circ \DD(F)$.
\end{lemma}
\begin{proof}
  We set $\Lambda = \dot\SSi(F)$.  We have $\dot\SSi(\DD F) = \Lambda^a$ by
  Corollary~\ref{cor:opboim}.  We define $\Lambda' = \Lambda_\Phi \circ
  \Lambda$.  Since $\Phi_t$ is homogeneous for the action of $\Rm$ for all $t$,
  we find $(\Lambda')^a = \Lambda_\Phi \circ \Lambda^a$.

  Applying~\eqref{eq:restr=eq} with $\Lambda'$ of $(\Lambda')^a$ we find
  $F' \in \Derb_{\Lambda'}(\cor_{M})$ and
  $G \in \Derb_{(\Lambda')^a}(\cor_{M})$ such that $\opb{j_0}(F') \simeq F$ and
  $\opb{j_0}(G) \simeq \DD(F)$.

  For any $t\in I$, the map $j_t$ is non-characteristic for $\Lambda'$ and we
  have $\DD(\opb{j_t}(F')) \simeq \opb{j_t}(\DD (F'))$ by
  Theorem~\ref{th:opboim}.  In particular $\opb{j_0}(\DD (F')) \simeq \DD(F)$.
  Hence $G$ and $\DD(F') \in \Derb_{(\Lambda')^a}(\cor_{M})$ satisfy
  $\opb{j_0}(G) \simeq \opb{j_0}(\DD (F'))$.  Since $\opb{j_0}$
  in~\eqref{eq:restr=eq} is an equivalence, we deduce that $\DD(F') \simeq G$.
  Then
  $K_t \circ \DD(F) \simeq \opb{j_t}(G) \simeq \opb{j_t}(\DD(F')) \simeq
  \DD(\opb{j_t}(F')) \simeq \DD(K_t \circ F)$, as claimed.
\end{proof}

\begin{lemma}
\label{lem:autodual}
Let $G \in \Derb(\cor_\R)$ be a constructible sheaf.  We assume that $G$ has
compact support, that there exists an isomorphism $G \simeq \DD(G)$ and that
$\rsect(\R;G) \simeq \cor$.  Then there exist $x_0 \in \R$ and a decomposition
$G \simeq \cor_{\{x_0\}} \oplus \bigoplus_{a\in A} \cor^{n_a}_{I_a}[d_a]$ where the
$I_a$ are half-closed intervals.
\end{lemma}
\begin{proof}
  By Lemma~\ref{lem:faiscdim1_compl_scinde} and Corollary~\ref{cor:cons_sh_R}
  there exists a finite family $A$ of bounded intervals and integers such that
  $G_t \simeq \bigoplus_{a\in A} \cor^{n_a}_{I_a}[d_a]$. Then
  $\rsect(\R;G_t) \simeq \bigoplus_{a\in A} \rsect(\R;\cor_{I_a})^{n_a}[d_a]$.
  Since $\rsect(\R;G_t) \simeq \cor$ all the intervals $I_a$ but one, say
  $I_b$, have cohomology zero, which means that they are half-closed.  If the
  remaining interval $I_b$ is open or is closed with non empty interior, then
  $\ol{I_b}$ or $\Int(I_b)$ also appears in the family $I_a$, $a\in A$, because
  $\DD(G_t) \simeq G_t$ and the decomposition is unique.  But in this case
  $\ol{I_b}$ or $\Int(I_b)$ will also contribute to $\rsect(\R;G_t)$ and this
  is impossible.  Hence $I_b$ is reduced to one point, say $x_0$, and we obtain
  the lemma.
\end{proof}

Now we can prove the three cusps conjecture.  We denote by
$PT^*\sph = \dT^*\sph/\Rm$ the projectivized cotangent bundle of $\sph$.  Let
$I$ be an open interval containing $0$ and let
$\ol{\Phiproj} \cl PT^*\sph \times I \to PT^*\sph$ be an isotopy of $PT^*\sph$.
It lifts to a homogeneous Hamiltonian isotopy
$\Phiproj \cl \dT^*\sph \times I \to \dT^*\sph$ which satisfies
$\Phiproj_0 = \id$ and
$\Phiproj_t(x;\lambda\,\xi) = \lambda \cdot \Phiproj_t(x;\xi)$ for all
$\lambda \in \Rm$ (not only for $\lambda \in \Rp$) and all
$(x;\xi) \in \dT^*\sph$.  Let $P_0\in \sph$ be a given point.  We set
$\Lambda_0 = \dT^*_{P_0}\sph$, $\Lambda_t = \Phiproj_t(\Lambda_0)$ and
$C_t = \pi_\sph(\Lambda_t)$.

If $\Phiproj_t$ is generic, then the map $\Lambda_t/\Rp \to C_t$ is generically
$2{:}1$ and $C_t$ has an odd number of cusps.  We will apply
Theorem~\ref{thm:clisoDUVF} and we need to modify $\Lambda_t$ so that the
projection to the base is $1{:}1$.  Let $\Phistd \cl \dT^*\sph \times I \to
\dT^*\sph$ be a generic homogeneous Hamiltonian isotopy (homogeneous only for
the action of $\Rp$).  We set $\Lambda_{t,\varepsilon} =
\Phistd_\varepsilon(\Lambda_t)$ and $C_{t,\varepsilon} =
\pi_\sph(\Lambda_{t,\varepsilon})$. Then the map $\Lambda_{t,\varepsilon}/\Rp
\to C_{t,\varepsilon}$ is generically $1{:}1$ and, for $\varepsilon$ small,
$C_{t,\varepsilon}$ has twice more cusps than $C_t$.  The following lemma will
be used to check that we do not change the direct image of a sheaf when we apply
$\Phistd_\varepsilon$.

\begin{lemma}
\label{lem:proj_invariante}
Let $M$, $N$ be two manifolds and $I$ be an open interval containing $0$.  Let
$q \cl M\times N\times I \to N\times I$ be the projection.  Let $\Lambda \subset
\dT^*(M\times N \times I)$ be a closed conic subset such that
\begin{itemize}
\item [-] $\Lambda \cap ((M\times N) \times T^*_tI) = \emptyset$ for all $t\in
  I$,
\item [-] $\Lambda' = q_\pi( \Lambda \cap (M \times T^*(N\times I)))$ is a
  product $\Lambda' = \Lambda'' \times T^*_II$.
\end{itemize}
Let $F \in \Derb(\cor_{M\times N\times I})$ be such that $\dot\SSi(F) \subset
\Lambda$ and $q$ is proper on $\supp(F)$. Then $\roim{q'}(F_t) \simeq
\roim{q'}(F_0)$ for all $t\in I$, where $F_t = F|_{\times N \times \{t\}}$ and
$q' \cl M\times N \to N$ is the projection.
\end{lemma}
\begin{proof}
  By Theorem~\ref{th:opboim} the hypothesis imply that $\dot\SSi(\roim{q}(F))
  \subset \Lambda'$.  Since $\Lambda'$ is a product the lemma follows from
  Corollary~\ref{cor:opbeqv}.
\end{proof}

\begin{theorem}
\label{thm:threecusps}
Let $t\in I$ be such that $C_t$ is a curve with only cusps and double points as
singularities. Then $C_t$ has at least three cusps.
\end{theorem}
\begin{proof}
  (i) By~\eqref{eq:comp=eqSimp} the categories $\Dersf_{\Lambda_0}(\corC_\sph)$
  and $\Dersf_{\Lambda_t}(\corC_\sph)$ are equivalent.  By
  Corollary~\ref{cor:fais_sphere0} the category $\Dersf_{\Lambda_0}(\corC_\sph)$
  has countably many isomorphism classes.  Hence it is enough to prove that, if
  $C_t$ has only one cusp, then $\Dersf_{\Lambda_t}(\corC_\sph)$ has uncountably
  many isomorphism classes.

  \sui (ii) We recall that $P_0$ is our starting point
  ($\Lambda_0 = T^*_{P_0}\sph$). Let us set $F_0 = \corC_{P_0}[-1]$.  Let
  $K_t \in \Derb(\corC_{\sph\times\sph})$ be the sheaf associated with
  $\Phiproj_t$ by Theorem~\ref{thm:qhi}. We set $F_t = K_t \circ F_0$.  Then
  $\dot\SSi(F_t) = \Lambda_t$.  We have $\DD(F_0) \simeq F_0$ and by
  Lemma~\ref{lem:autodual} we deduce that $\DD(F_t) \simeq F_t$ for all
  $t\in I$.  By~\eqref{eq:cohom_indpt-t} we also have
  $\rsect(\sph; F_t) \simeq \rsect(\sph; F_0) \simeq \corC$ for all $t\in I$.

  \sui (iii) We choose a Morse function $q \cl \sph \to \R$ with only one
  minimum $P_-$ and one maximum $P_+$ such that $P_-, P_+ \not\in C_\Lambda$.

  We set $G_t = \roim{q}(F_t)$.  We have isomorphisms $\DD(G_t) \simeq
  \roim{q}(\DD F_t) \simeq G_t$ and $\rsect(\R;G_t) \simeq \rsect(\sph; F_t)
  \simeq \corC$.  By Lemma~\ref{lem:autodual} there exist $x_0 \in \R$ and a
  decomposition $G_t \simeq \corC_{\{x_0\}} \oplus \bigoplus_{a\in A}
  \corC^{n_a}_{I_a}[d_a]$ where the $I_a$ are half-closed intervals.  We set
  $x_\pm = q(P_\pm)$.  Since $F_t$ is constant near $P_-$, we have $G_t \simeq L
  \otimes (\corC_{[x_-,+\infty[} \oplus \corC_{]x_-,+\infty[}[-1])$ near $x_-$,
  where $L = (F_t)_{P_-}$.  The same holds near $x_+$.  Hence $x_0 \not= x_-$
  and $x_0 \not= x_+$.

  We set $M = \sph \setminus \{P_-, P_+\}$.  The restriction $F_t|_M$ is
  constant outside a compact set.  We identify $]x_-,x_+[$ with $\R$ and $M$
  with $\cer\times \R$ in such a way that $q$ is the projection $M\to \R$ and
  $x_0 = 0$.  For a generic choice of $q$, any fiber $\cer \times \{x\}$ of $q$
  contains at most one ``accident'' among: tangent point to $C_\Lambda$, double
  point or cusp.

  \sui (iv) We set $G'_t = \roim{(q|_M)}(F_t|_M)$.  Then $\corC_{\{0\}}$ is a
  direct summand of $G'_t$ and $T^*_0\R \subset \SSi(G'_t)$.  Hence
  $\Lambda \cap T^*_{\opb{q}(0)}\sph \not= \emptyset$.  By the assumption on
  $q$ it follows that $\opb{q}(0)$ contains no cusp or double point and that
  $\opb{q}(0)$ is tangent to $C_\Lambda$ at one point, say $(0,0)$.

  We choose another Hamiltonian isotopy $\Phistd$ such that
  $\Phistd_\varepsilon(\Lambda_t)$ is generic for $\varepsilon$ small and
  $\Phistd_\varepsilon(0,0;0,\pm 1) = (\pm\varepsilon,0;0, \pm 1)$.  We let
  $K'_\varepsilon$ be the sheaf associated with $\Phistd_\varepsilon$ and we set
  $F'_\varepsilon = K'_\varepsilon \circ F_t$.  By
  Lemma~\ref{lem:proj_invariante} we have $\roim{q}(F'_\varepsilon) \simeq G'_t$.

  \sui (v) Let us assume that $C_t$ has one cusp. Hence
  $\dot\pi_{\sph}(\Phistd_\varepsilon(\Lambda_t))$ has two cusps and the
  hypothesis of Theorem~\ref{thm:clisoDUVF} are satisfied for $\Lambda = \dT^*M
  \cap \Phistd_\varepsilon(\Lambda_t)$ and $F = F'_\varepsilon|_M$.  It follows
  that the category $\Der(U,V;F)$ contains uncountably many isomorphism classes,
  where $U = M \setminus \opb{q}(0)$ and $V = \opb{q}(]0,\eta[)$ for some
  $\eta>0$.  Since $F$ is constant outside some compact set and the objects of
  $\Der(U,V;F)$ are isomorphic to $F$ over $U$, they are also constant outside
  some compact set. Hence they extend to $\sph$ in such a way that the extension
  is constant near $P_-$ and $P_+$.  In particular their extensions belong to
  $\Dersf_{\Phistd_\varepsilon(\Lambda_t)}(\corC_\sph)$.  Hence the category
  $\Dersf_{\Phistd_\varepsilon(\Lambda_t)}(\corC_\sph)$ contains uncountably
  many objects.  Since it is equivalent to $\Derb_{\Lambda_t}(\corC_\sph)$, the
  theorem is proved.
\end{proof}

\providecommand{\bysame}{\leavevmode\hbox to3em{\hrulefill}\thinspace}

\vspace*{1cm}
\noindent
\parbox[t]{21em}
{\scriptsize{
\noindent
St{\'e}phane Guillermou\\
Institut Fourier, Universit{\'e} Grenoble I, \\
email: Stephane.Guillermou@ujf-grenoble.fr
}}

\end{document}